%% file: main.tex
\documentclass[11pt]{simplearticle}

\makeatletter

\input{preamble.tex}

\input{metainfo.tex}

\makeatother

\begin{document}


\maketitlepage

\tableofcontents{}
\input{sec0.tex}
\input{sec1.tex}
\input{sec2.tex}
\input{sec3.tex}
\input{sec4.tex}

\input{biblio.tex}
\end{document}

%% file: preamble.tex
\usepackage[inline]{enumitem}

\DeclareMathAlphabet{\mathbbold}{U}{bbold}{m}{n}

\usepackage{mquiver}
\usepackage{tikz-cd}
\usepackage[all]{xy}
\usepackage{xstring}

\usetikzlibrary{decorations.pathreplacing,shapes.misc}
\usetikzlibrary{nfold}

\newcommand{\mymakebox}[2]{\text{\text{\makebox[#1][c]{\ensuremath{#2}}}}}
\newcommand{\mydisplaybreak}{\,}
\newcommand{\norm}[1]{\left\Vert #1\right\Vert }
\newcommand{\abs}[1]{|#1|}
\newcommand{\absB}[1]{\Bigl|#1\Bigr|}
\newcommand{\ev}{\operatorname{ev}}
\newcommand{\as}{\operatorname*{as}}
\newcommand{\const}{\operatorname{const}}
\newcommand{\diag}{\operatorname*{diag}}
\newcommand{\pr}{\operatorname{pr}}
\newcommand{\mat}{\operatorname*{mat}}
\newcommand{\qmat}{\operatorname*{qmat}}
\newcommand{\spn}{\operatorname{span}}
\newcommand{\fun}{\operatorname*{fun}}
\newcommand{\strlim}{\operatorname*{strlim}}
\newcommand{\colim}{\operatorname*{colim}}
\newcommand{\dist}{\operatorname{dist}}
\newcommand{\prp}{\operatorname{prop}}
\renewcommand{\sc}{\operatorname{sc}}
\newcommand{\supp}{\operatorname{supp}}
\newcommand{\id}{\operatorname{id}}
\newcommand{\Id}{\operatorname{Id}}
\newcommand{\incl}{\operatorname{incl}}
\newcommand{\GEFC}{\mathbb{\mathtt{GEFC}}}
\newcommand{\hGEFC}{\mathtt{hGEFC}}
\newcommand{\DEFC}{\mathbb{\mathtt{DEFC}}}
\newcommand{\LEFC}{\mathbb{\mathtt{LEFC}}}
\newcommand{\Cstar}{\mathtt{C^{*}}}
\newcommand{\hAsy}{\mathtt{hAsy}}
\newcommand{\EFC}{\Cstar^{\Cstar}}
\newcommand{\Ext}{\mathrm{Ext}}
\newcommand{\mathsb}[1]{\boldsymbol{\mathsf{#1}}}
\newcommand{\C}{\mathfrak{C}^{\,}}
\newcommand{\myfonta}[1]{\mathsf{#1}}
\newcommand{\nbhd}{\myfonta N}
\newcommand{\ball}{\myfonta B}
\newcommand{\asadj}{\dashv_{as}}
\newcommand{\mmph}{\text{{␣}}}
\newcommand{\uu}{\dot{\oplus}}
\newcommand{\tensorunit}{\mathbbold1}
\newcommand{\bbzero}{\mathbbold0}
\newcommand{\terminal}{\operatorname{!}}
\newcommand{\rc}{\operatorname{rc}}
\newcommand{\lc}{\operatorname{lc}}
\newcommand{\oK}{\mathbb{K}}
\newcommand{\fK}{\mathbf{K}}
\newcommand{\fM}{\mathbf{M}}
\newcommand{\oM}{\mathbb{M}}
\newcommand{\uM}{M}
\newcommand{\uK}{\mathbb{K}}
\newcommand{\mmdc}{,}
\newcommand{\hspc}[1]{\hspace{#1ex}}
\newcommand{\middlepipe}{~\middle|~}
\newcommand{\discr}{\mathrm{discr}}
\newcommand{\noloc}{\nobreak
\mspace{4mu plus 1mu}
{:}
\nonscript\mkern-\thinmuskip
\mathpunct{}
\mspace{2mu}
}
\newcommand{\singlecdot}{\,\cdot\,}
\newcommand{\myllbracket}{[[}
\newcommand{\myrrbracket}{]]}
\newcommand{\bbrackets}[1]{\myllbracket#1\myrrbracket}

%% file: metainfo.tex
\title{Scalability and asymptotic adjunction}

\authors{
		\ainfo{Georgii~S.~Makeev}
		{School of Mathematical Sciences, Key Laboratory of MEA (Ministry of
Education) and Shanghai Key Laboratory of PMMP, East China Normal
University, Shanghai 200241, People’s Republic of China.
}
		{makeev.gs@gmail.com}
		{}
}


\date{October 2025}

\begin{abstract}
In this paper, we introduce relative Roe functors and show that for
every pair of scalable locally compact metric spaces with bounded
coarse geometry, the functor of continuous functions and the relative
Roe functor, both associated with this pair, are asymptotically adjoint.
While this asymptotic adjunction is weaker than the genuine one, it
retains sufficient categorical properties to be intuitive and useful
in applications. These results can be used to provide an unsuspended
description of the Connes-Higson $E$-theory, establish connections
between $E_{1}$-theory and extension theory, and express $K$-homology
of compact metric spaces in terms the corresponding metric cones.
\end{abstract}

%% file: sec0.tex
\section*{Introduction\addcontentsline{toc}{section}{Introduction}}

The homotopy category of asymptotic homomorphisms is a crucial component
of the Connes-Higson $E$-theory. Originally, the morphisms in this
category were defined in terms of asymptotic families~\cite{connes-higson1990},
but later~\cite{GHT}  proposed an alternative functorial approach
which we now briefly recall. For a $C^{*}$-algebra $B$, let $\mathfrak{T}B$
be the $C^{*}$-algebra of continuous bounded functions on $[0,\infty)$
taking values in~$B$, and let $\mathfrak{T}_{0}B$ be the ideal
in~$\mathfrak{T}B$ consisting of the functions vanishing at infinity.
The quotient $C^{*}$-algebra $\mathfrak{A}B=\mathfrak{T}B/\mathfrak{T}_{0}B$
is called the \emph{asymptotic algebra} of $B$. It is clear that
$\mathfrak{T}_{0}$ and $\mathfrak{T}$ are endofunctors of the category
of $C^{*}$-algebras and $*$-homomorphisms, and one can show that
so is $\mathfrak{A}$.  For every $n\geq0$, let $\mathfrak{A}^{n}$
denote the $n$-fold asymptotic algebra functor $\mathfrak{A}^{n}\coloneqq\mathfrak{A}\dots\mathfrak{A}$
($n$~times). An \emph{asymptotic homomorphism} from $A$ to $B$
is a $*$-homomorphism $A\to\mathfrak{A}^{n}B$ for some $n\geq0$.
Two $*$-homomorphisms $\varphi_{0},\varphi_{1}\colon A\to\mathfrak{A}^{n}B$
are called \emph{$n$-homotopic} (written $\varphi_{0}\simeq_{n}\varphi_{1}$)
if there is a $*$-homomorphism $\Phi\colon A\to\mathfrak{A}^{n}IB$
making the following diagrams commutative for~$j=0,1$:
\[
\input{diagrams/d219.tikz}
\]
where $IB\coloneqq C([0,1],B)$ stands for the cylinder of $B$, and
$e_{t}\colon IB\to B$ denotes the evaluation at point $t\in[0,1]$.
The $n$-homotopy classes of such $*$-homomorphisms are denoted by
$\bbrackets{A,B}_{n}\coloneqq\hom(A,\mathfrak{A}^{n}B)/\simeq_{n}$;
and the class of $\varphi\colon A\to\mathfrak{A}^{n}B$ in $\bbrackets{A,B}_{n}$
is written as $\bbrackets{\varphi}_{n}$.

Let $\Id$ be the identity endofunctor of $C^{*}$-algebras. Denote
by $\alpha\colon\Id\Rightarrow\mathfrak{A}$ the natural transformation
whose component at the object $B$ sends $b$ to the class of the
constant function $t\mapsto b$ in $\mathfrak{A}B$. For every $n\geq0$,
there is a map of sets
\[
\bbrackets{A,B}_{n}\to\bbrackets{A,B}_{n+1}\colon\bbrackets{\varphi}_{n}\mapsto\bbrackets{\alpha\mathfrak{A}^{n}B\circ\varphi}_{n+1},
\]
where $\alpha\mathfrak{A}^{n}B$ denotes the component of $\alpha$
at $\mathfrak{A}^{n}B$. Introduce the set $\bbrackets{A,B}\coloneqq\colim_{n}\bbrackets{A,B}_{n}$,
and denote by $\bbrackets{\varphi}$ the class of $\varphi\colon A\to\mathfrak{A}^{n}B$
in $\bbrackets{A,B}$. The \emph{homotopy category of asymptotic
homomorphisms} (denoted by $\hAsy$) is the category with objects
$C^{*}$-algebras, arrows from $A$ to $B$ elements of the set $\bbrackets{A,B}$,
and composition given by
\[
\bbrackets{B\xrightarrow{\psi}\mathfrak{A}^{m}C}\circ\bbrackets{A\xrightarrow{\varphi}\mathfrak{A}^{n}B}\coloneqq\bbrackets{A\xrightarrow{\mathfrak{A}^{n}\psi\circ\varphi}\mathfrak{A}^{n+m}B}.
\]
The colimit construction in the definition of $\bbrackets{A,B}$ is
as convenient from the algebraic standpoint, as obscure from the analytic
one. This, nonetheless is alleviated by the following fact: if $A$
is separable then $\bbrackets{A,B}_{1}\cong\bbrackets{A,B}$ (see~\cite[Theorem 2.16]{GHT}).

Denote by $\mathbb{K}$ the $C^{*}$-algebra of compact operators
in a separable infinite dimensional Hilbert space, and recall that
a $C^{*}$-algebra $B$ is called \emph{stable} if $B\cong\mathbb{K}\otimes B$.
Abusing notation, denote again by $\mathbb{K}$ the endofunctor which
sends a $C^{*}$-algebra $B$ to its stabilization $B\otimes\mathbb{K}$.
If $B$ is stable then the set $\bbrackets{A,B}$ can be endowed with
the structure of a commutative monoid with the sum $\bbrackets{\varphi\colon A\to\mathfrak{A}^{n}B}+\bbrackets{\psi\colon A\to\mathfrak{A}^{n}B}$
defined as the class of the asymptotic homomorphism
\[
A\to\mathfrak{A}^{n}B\colon a\mapsto\mathfrak{A}^{n}\theta_{2}\left(\begin{array}{cc}
\varphi(a) & 0\\
0 & \psi(a)
\end{array}\right),
\]
where $\theta_{2}\colon M_{2}B\to B$ denotes the stability isomorphism,
with $M_{2}B$ standing for the $C^{*}$-algebra of $2\textrm{-by-}2$-matrices
with entries in~$B$.

The \emph{Connes-Higson $E$-theory category} is a category with
objects separable $C^{*}$-algebras, and arrows from $A$ to $B$
elements of the set $E_{0}(A,B)\coloneqq\bbrackets{S\mathbb{K}A,S\mathbb{K}B}$,
where $S\coloneqq C_{0}(\mathbb{R},\mmph)$ stands for the \emph{suspension}
functor. The suspension actually determines an abelian group on $E_{0}(A,B)$:
the inverse element is induced by the pre-composition (or post-composition)
with the homomorphism $f\mapsto f(-\singlecdot)$. It is easy to show
that the functor $\mathbb{K}$ in the left hand side can be discarded,
i.e., 
\[
E_{0}(A,B)\cong\bbrackets{SA,S\mathbb{K}B}.
\]
Similarly to the operator $K$-theory, one can inductively define
the sequence of bifunctors
\[
E_{n}(A,B)\coloneqq E_{n-1}(SA,B),
\]
from $C^{*}$-algebras to abelian groups, which forms a generalized
(co)homology theory satisfying the Bott periodicity 
\[
E_{n+2}(A,B)\cong E_{n}(A,B).
\]
In fact, for separable nuclear $C^{*}$-algebras, $E$-theory is isomorphic
to the Kasparov $KK$-theory~\cite[Theorem~25.6.3]{blackadar1998}.

Recall that two functors $L\colon\mathcal{C}\leftrightarrows\mathcal{C}'\noloc R$
are called adjoint (written $L\dashv R$) if there exist two natural
transformations $\eta\colon\Id_{\mathcal{C}}\Rightarrow RL$ and $\varepsilon\colon LR\Rightarrow\Id_{\mathcal{C}'}$,
called respectively the \emph{unit} and \emph{counit}, which render
commutative the diagrams \begin{alignat*}{2}
\input{diagrams/d211.tikz} & \qquad\qquad & \input{diagrams/d212.tikz}
\end{alignat*}
or equivalently, if there exists an isomorphism $\mathcal{C}'(LA,B)\cong\mathcal{C}(A,RB)$
natural in~$A$ and~$B$. 

For a proper metric space $\mathsf{X}$, consider the endofunctor
$\C_{\mathsf{X}}\coloneqq C_{0}(\mathsf{X},\mmph)$ which is isomorphic
to the tensoring with $C_{0}(\mathsf{X})$. It is known~\cite{uuye2013homotopical}
that $\C_{\mathsf{X}}$ fails to have a right adjoint in $\hAsy$
even in the simplest case when $\mathsf{X}=\mathbb{R}$, i.e., when
$\C_{\mathsf{X}}$ is the suspension functor. 

In this paper we present a workaround to deal with this problem and
introduce a similar notion called \emph{asymptotic adjunction},
which, despite not being a genuine adjunction, still retains many
of its categorical properties.  This approach is based on the categorical
framework of \emph{good endofunctors}~\cite{makeev_gmaec} inspired
by the following crucial properties of~$\mathfrak{A}$:
\begin{itemize}
\item $\mathfrak{A}$ preserves pullbacks along a pair of $*$-homomorphisms
one of which is surjective (see~\cite[Lemma~2.5]{GHT});
\item there exists a natural transformation $I\mathfrak{A}\Rightarrow\mathfrak{A}I$
making the following diagram commutative for~$j=0,1$:
\[
\input{diagrams/d502.tikz}
\]
where $\ev_{t}\colon I\Rightarrow\Id$ is the natural transformation
of evaluation at $t\in[0,1]$ (cf.~\cite[proof of Lemma~2.9]{GHT}).
\end{itemize}
In~\cite{makeev_gmaec} we introduce the bimonoidal category of good
endofunctors and defined homotopies of natural transformations between
them. Furthermore, we show that every good endofunctor $F$ induces
the following equivalence relation on $*$-homomorphisms: $\varphi_{0},\varphi_{1}\colon A\to FB$
are called \emph{$F$-homotopic} (written $\varphi_{0}\simeq_{F}\varphi_{1}$)
if there is a $*$-homomorphism $\Phi\colon A\to FIB$ making the
following diagram commutative for~$j=0,1$:
\[
\input{diagrams/d227.tikz}
\]
which as one can see is just a verbatim generalization of $n$-homotopy
defined above. Thus, every good endofunctor $F$ induces the set of
generalized morphisms and generalized asymptotic morphisms between
$C^{*}$-algebras $A$ and $B$ defined as:
\begin{alignat*}{2}
[A,F,B]\coloneqq\hom(A,FB)/\simeq_{F} & ,\qquad & \bbrackets{A,F,B}\coloneqq\colim_{n}[A,F\mathfrak{A}^{n}\mathbb{K},B].
\end{alignat*}
We say that good endofunctors $L$ and $R$ are \emph{asymptotically
adjoint} (written $L\asadj R$) if there exist $\eta\colon\Id\Rightarrow RL$
and $\varepsilon\colon LR\Rightarrow\mathfrak{A}\mathbb{K}$ such
that the following diagrams commute up to homotopy:%
\begin{alignat*}{2}
\input{diagrams/d213.tikz} & \qquad\qquad & \input{diagrams/d214.tikz}
\end{alignat*}
where $\alpha\colon\Id\Rightarrow\mathfrak{A}$ is as above, $\iota_{00}\colon\Id\Rightarrow\mathbb{K}$ is the corner embedding,
and $\alpha\iota_{00}L$ (resp. $R\alpha\iota_{00}$) stands for the
horizontal composition of $\alpha$, $\iota_{00}$, and $1_{L}$ (resp.
$1_{R}$, $\alpha$, $\iota_{00}$) (detailed explanation of this
notation will be given in the beginning of the next section). In this
case, one can easily obtain the isomorphism $\bbrackets{LA,\Id,B}\cong\bbrackets{A,R,B}$.

In this paper, for every discrete metric space $X$ of bounded geometry,
we introduce the \emph{uniform Roe functor} $\mathfrak{M}^{u}_{X}$
which sends a $C^{*}$-algebra $B$ to the ``uniform Roe algebra''
with coefficients in $B$. Every subspace $X_{0}\subset X$ gives
rise to the functor $\mathfrak{M}^{u}_{X\supset X_{0}}$ which sends
$B$ to the ideal $\mathfrak{M}^{u}_{X\supset X_{0}}B\subset\mathfrak{M}^{u}_{X}B$
generated by the matrices supported near $X_{0}\times X_{0}$. Taking
the componentwise quotient we obtain the \emph{relative uniform
Roe functor} $\mathfrak{N}^{u}_{X,X_{0}}\coloneqq\mathfrak{M}^{u}_{X}/\mathfrak{M}^{u}_{X\supset X_{0}}$.
For a locally compact metric space $\mathsf{X}$
and a closed subspace $\mathsf{X}_{0}\subset\mathsf{X}$, we denote
by $\C_{\mathsf{X},\mathsf{X}_{0}}\coloneqq C_{0}(\mathsf{X},\mathsf{X}_{0},\mmph)$
the functor which sends a $C^{*}$-algebra $B$ to continuous $B$-valued
functions on $\mathsf{X}$ vanishing on $\mathsf{X}_{0}$ and at infinity.
It will be shown that if the pair of locally compact metric spaces
$(\mathsf{X},\mathsf{X}_{0})$ is scalable and $\mathsf{X}$
has bounded coarse geometry, then the functor $\C_{\mathsf{X},\mathsf{X}_{0}}$
admits a right asymptotic adjoint, which is $\mathfrak{N}^{u}_{X,X_{0}}$
for every pair of discrete metric spaces $(X,X_{0})$ coarsely equivalent
to~$(\mathsf{X},\mathsf{X}_{0})$.

This result yields several interesting applications discussed in Section~\ref{subsec:appl},
that are immediately accessible to the reader without the need to
delve into the technicalities.

The paper is structured as follows. Section~\ref{sec:prelim} recalls
the category machinery of good endofunctors developed in~\cite{makeev_gmaec}.
Section~\ref{sec:Roe-functors} introduces relative Roe functors
and shows that they fit into the framework discussed above. Our main
result appears in Section~\ref{sec:main}, where we prove that tensoring
with continuous functions on a scalable pair of locally
compact metric spaces of bounded coarse geometry admits a right asymptotic
adjoint, which is the precisely the relative uniform Roe functor associated
with the corresponding discretization. We discuss applications of
this result in the same section. Finally, Section~\ref{sec:tech}
establishes that the colimit construction is unnecessary for separable
$C^{*}$-algebras.

%% file: diagrams/d219.tikz
\begin{tikzcd}[ampersand replacement=\&]
{A} \& {\mathfrak{A}^{n}IB} \\
 \& {\mathfrak{A}^{n}B,}
\arrow["\varphi_{j}"', from=1-1, to=2-2]
\arrow["\Phi", from=1-1, to=1-2]
\arrow["\mathfrak{A}^{n}e_{j}", from=1-2, to=2-2]
\end{tikzcd}

%% file: diagrams/d211.tikz
\begin{tikzcd}[ampersand replacement=\&]
{L} \& {LRL} \\
 \& {{L,}}
\arrow["{L\eta}", Rightarrow, from=1-1, to=1-2]
\arrow["{=}"', Rightarrow, from=1-1, to=2-2]
\arrow["{\varepsilon L}", Rightarrow, from=1-2, to=2-2]
\end{tikzcd}

%% file: diagrams/d212.tikz
\begin{tikzcd}[ampersand replacement=\&]
{R} \& {RLR} \\
 \& {{R,}}
\arrow["{\eta R}", Rightarrow, from=1-1, to=1-2]
\arrow["{=}"', Rightarrow, from=1-1, to=2-2]
\arrow["{R\varepsilon}", Rightarrow, from=1-2, to=2-2]
\end{tikzcd}

%% file: diagrams/d502.tikz
\begin{tikzcd}[ampersand replacement=\&]
{I\mathfrak{A}} \&  \& {\mathfrak{A}I} \\
 \& {\mathfrak{A},}
\arrow[Rightarrow, from=1-1, to=1-3]
\arrow["\ev_{j}\mathfrak{A}"', Rightarrow, from=1-1, to=2-2]
\arrow["\mathfrak{A}\ev_{j}", Rightarrow, from=1-3, to=2-2]
\end{tikzcd}

%% file: diagrams/d227.tikz
\begin{tikzcd}[ampersand replacement=\&]
{A} \& {FIB} \\
 \& {FB,}
\arrow["\varphi_{j}"', from=1-1, to=2-2]
\arrow["\Phi", from=1-1, to=1-2]
\arrow["Fe_{j}", from=1-2, to=2-2]
\end{tikzcd}

%% file: diagrams/d213.tikz
\begin{tikzcd}[ampersand replacement=\&]
{L} \& {LRL} \\
 \& {\mathfrak{A}\mathbb{K}L,}
\arrow["L\eta", Rightarrow, from=1-1, to=1-2]
\arrow["\alpha\iota_{00}L"', Rightarrow, from=1-1, to=2-2]
\arrow["\varepsilon L", Rightarrow, from=1-2, to=2-2]
\end{tikzcd}

%% file: diagrams/d214.tikz
\begin{tikzcd}[ampersand replacement=\&]
{R} \& {RLR} \\
 \& {R\mathfrak{A}\mathbb{K},}
\arrow["\eta R", Rightarrow, from=1-1, to=1-2]
\arrow["R\alpha\iota_{00}"', Rightarrow, from=1-1, to=2-2]
\arrow["R\varepsilon", Rightarrow, from=1-2, to=2-2]
\end{tikzcd}

%% file: sec1.tex
\section{\label{sec:prelim}Good endofunctors}

This section establishes the categorical foundations of our approach
developed in~\cite{makeev_gmaec}. It presents a slightly simplified
review of the main definitions and results from that work which will
be used in the present paper. It also serves as a compressed introduction
for the reader who seeks an overview before diving into the technicalities
of~\cite{makeev_gmaec}.

Let us first fix some basic notation. Consider the categories, functors,
and natural transformations as in the diagram below:
\[
\input{diagrams/d371.tikz}
\]
We denote by:
\begin{itemize}
\item $FB$ and $F\varphi$ the value of the functor $F$ respectively at
the object $B$ and the morphism $\varphi$ in the category~$\mathcal{C}_{1}$;
\item $GF$ the composition of the functors $F$ and $G$;
\item $\alpha B\colon FB\to F'B$ the component of the natural transformation
$\alpha$ at the object~$B$; \item $\beta F$, $G\alpha$ the whiskering of $\beta$ and $F$, and $G$
and $\alpha$, respectively;
\item $\beta\alpha\colon GF\Rightarrow G'F'$ the horizontal composition
of $\beta$ and $\alpha$;
\item $\alpha'\circ\alpha\colon F\Rightarrow F''$ the vertical composition
of $\alpha$ and $\alpha'$.
\end{itemize}
We write $B\in\in\mathcal{C}$ and $f\in\mathcal{C}$ to express the
fact that $B$ is an object and $f$ is an arrow in the category~$\mathcal{C}$.

\subsection{Monoidal categories}
\begin{defn}
A \emph{monoidal category} is a tuple $(\mathcal{C},\otimes,\tensorunit,\alpha^{\otimes},\lambda^{\otimes},\rho^{\otimes})$
consisting of: 
\begin{itemize}
\item a category $\mathcal{C}$; 
\item a functor $\otimes\colon\mathcal{C}\times\mathcal{C}\to\mathcal{C}$
called the \emph{monoidal product}; 
\item an object $\tensorunit\in\in\mathcal{C}$ called the \emph{monoidal
unit}; 
\item a natural isomorphism $\begin{tikzcd}(X\otimes Y)\otimes Z\xrightarrow{\alpha^{\otimes}_{X,Y,Z}}\end{tikzcd}X\otimes(Y\otimes Z)$
for all $X,Y,Z\in\in\mathcal{C}$ called the \emph{associativity isomorphism}; 
\item natural isomorphisms $\tensorunit\otimes X\xrightarrow{\lambda^{\otimes}_{X}}X\xleftarrow{\rho^{\otimes}_{X}}X\otimes\tensorunit$
for all $X\in\in\mathcal{C}$ called respectively the \emph{left unit
isomorphism} and the \emph{right unit isomorphism.}
\end{itemize}
These data are required to make the following diagrams commutative
for all $W,X,Y,Z\in\in\mathcal{C}$: \begin{alignat*}{3}
 & \textup{(Pentagon axiom)} & \qquad &  & \input{diagrams/d443.tikz}\\
 & \textup{(Middle unity axiom)} & \qquad &  & \input{diagrams/d442.tikz}
\end{alignat*}
A monoidal category is called \emph{strict} if the associativity
and unit isomorphisms are identities. In this case, $(\mathcal{C},\otimes,\tensorunit,\alpha^{\otimes},\lambda^{\otimes},\rho^{\otimes})$
will be abbreviated to~$(\mathcal{C},\otimes,\tensorunit)$.
\end{defn}

\begin{defn}\label{def:monoidal-functor}
 For monoidal categories $\mathcal{C}$
and $\mathcal{C}'$, a \emph{monoidal functor} 
\[
(F,F^{2},F^{0})\colon\mathcal{C}\to\mathcal{C}'
\]
consists of:
\begin{itemize}
\item a functor $F\colon\mathcal{C}\to\mathcal{C}'$; 
\item a natural transformation $F^{2}_{X,Y}\colon FX\otimes FY\to F(X\otimes Y)$
with $X,Y\in\in\mathcal{C}$; 
\item a morphism $F^{0}\colon\tensorunit_{\mathcal{C}'}\to F\tensorunit_{\mathcal{C}}$.
\end{itemize}
These data are required to make the following diagrams commutative
for all $X,Y,Z\in\in\mathcal{C}$: 
\[
\input{diagrams/d444.tikz}
\]
\begin{alignat*}{2}
\input{diagrams/d445.tikz} & \qquad\qquad & \input{diagrams/d446.tikz}
\end{alignat*}
A monoidal functor $(F,F^{2},F^{0})$ is called \emph{strong} (resp.
\emph{strict}) if both $F_{0}$ and $F_{2}$ are isomorphisms (resp.
identities).
\end{defn}

\subsection{Functor categories}

For categories $\mathcal{C}$ and $\mathcal{C}'$, denote by $\mathcal{C}'^{\mathcal{C}}$
the category with functors from $\mathcal{C}$ to $\mathcal{C}'$
as objects, and natural transformations between them as morphisms. 

If some property holds for every component of the natural transformation
$\alpha\colon F\Rightarrow G$, then we say that $\alpha$ enjoys
this property \emph{componentwise}. For example, the obvious natural
transformation $\mathfrak{T}_{0}\Rightarrow\mathfrak{T}$ between
the functors from the introduction is a componentwise ideal inclusion,
and $\mathfrak{T}\Rightarrow\mathfrak{A}$ is a componentwise quotient
projection.
\begin{lem}\label{lem:cw-epic-epic}
If a natural transformation is componentwise
monic (resp. componentwise epic), then it is monic (resp. epic). 
\end{lem}

Let $\mathcal{C}$ and $\mathcal{J}$ be categories. For every $\mathcal{J}$-shaped
diagram $F_{\bullet}\colon\mathcal{J}\to\mathcal{C}$, we denote by
$\pr_{j}\colon\lim_{i}F_{i}\to F_{j}$ the $j$th leg of the limit
cone over this diagram (provided it exists). 
\begin{defn}\label{def:ab}
Let $\mathcal{C}$, $\mathcal{C}'$, $\mathcal{C}''$
and $\mathcal{J}$ be categories, with $\mathcal{J}$ small. Define
the \emph{canonical arrows} (provided the limits below exist)
\begin{alignat}{3}
\rc\colon(\lim_{j}F_{j})G & \Rightarrow & \lim_{j}(F_{j}G), & \quad & \textrm{ where } & F_{\bullet}\colon\mathcal{J}\to\mathcal{C}''^{\mathcal{C}'},~G\colon\mathcal{C}\to\mathcal{C}',\label{eq:canar1}\\
\lc\colon G(\lim_{j}F_{j}) & \Rightarrow & \lim_{j}(GF_{j}), &  & \textrm{ where } & F_{\bullet}\colon\mathcal{J}\to\mathcal{C}'^{\mathcal{C}},~G\colon\mathcal{C}'\to\mathcal{C}''\label{eq:canar11}
\end{alignat}
as the unique arrows making the following diagrams commutative:
\begin{alignat}{2}
\input{diagrams/d464.tikz} & \qquad & \input{diagrams/d465.tikz}\label{eq:canar2}
\end{alignat}
\end{defn}

\begin{defn}\label{def:oplus}
Let $\mathcal{C}$ be a category with binary products
(which we denote by $\oplus$). For endofunctors $F,F',G,G'\colon\mathcal{C}\to\mathcal{C}$
and natural transformations $\alpha\colon F\Rightarrow F'$ and $\beta\colon G\Rightarrow G'$,
let us introduce:
\begin{itemize}
\item the endofunctor $F\oplus G$ defined as
\[
F\oplus G\colon(A\xrightarrow{\varphi}B)\mapsto(FA\oplus GA\xrightarrow{F\varphi\oplus G\varphi}FB\oplus GB)
\]
\item the natural transformation $\alpha\oplus\beta\colon F\oplus G\Rightarrow F'\oplus G'$
whose component at $B\in\in\mathcal{C}$ is defined as \[
(\alpha\oplus\beta)B\coloneqq\alpha B\oplus\beta B.
\]
\end{itemize}
It is easy to check that $\oplus$ is well-defined, functorial, and
that $F\oplus F'$ is a product of $F$ and~$F'$ in~$\mathcal{C}^{\mathcal{C}}$.
\end{defn}

\begin{defn}\label{def:Delta}
Let $\mathcal{C}$ be a category with binary products,
and let $X\in\in\mathcal{C}$. Denote by $\Delta$ the unique arrow
making the following diagram commutative:
\[
\input{diagrams/d466.tikz}
\]
\end{defn}

\subsection{$C^{*}$-algebras and decent endofunctors}

Denote by $\Cstar$ the category of $C^{*}$-algebras and $*$-homomorphisms.
It is well-known that $\Cstar$ is complete. In this paper we work
only with the maximal tensor product of $C^{*}$-algebras, which we
denote by $\otimes$, and which turns $\Cstar$ into a (non-strict)
monoidal category $(\Cstar,\otimes,\mathbb{C},\alpha^{\otimes},\lambda^{\otimes},\rho^{\otimes})$
with
\begin{alignat*}{1}
 & \alpha^{\otimes}\colon(a\otimes b)\otimes c\mapsto a\otimes(b\otimes c);\\
 & \lambda^{\otimes}\colon z\otimes a\mapsto za;\\
 & \rho^{\otimes}\colon a\otimes z\mapsto za.
\end{alignat*}

\begin{defn}
Let $A$ and $B$ be $C^{*}$-algebras. Two $*$-homomorphisms $\varphi_{0},\varphi_{1}\colon A\to B$
are called \emph{homotopic} if there is a $*$-homomorphism $\Phi\colon A\to IB$
making the following diagram commutative for~$j=0,1$:
\[
\input{diagrams/d447.tikz}
\]
where $IB\coloneqq C([0,1],B)$ is the \emph{cylinder} of $B$,
and $e_{t}\colon IB\to B$ stands for the evaluation at~$t\in[0,1]$.
\end{defn}

Denote by $\bbzero\coloneqq\{0\}$ the zero $C^{*}$-algebra, which
is clearly a zero object in $\Cstar$. 
\begin{defn}\label{def:decentFunctor}
 We call $F\in\in\Cstar^{\Cstar}$ a \emph{decent
endofunctor} if:
\begin{itemize}
\item for every pullback 
\[
\input{diagrams/d448.tikz}
\]
with $\varphi_{2}$ a split epimorphism\footnote{An arrow is called a \emph{split epimorphism} if it has a pre-inverse.},
the canonical arrow $F(B_{1}\underset{B}{\oplus}B_{2})\xrightarrow{\lc}FB_{1}\underset{FB}{\oplus}FB_{2}$\footnote{This formula is a special case of (\ref{eq:canar11}) with $\mathcal{C}$
the terminal category.} is an isomorphism;
\item the unique arrow $F\bbzero\xrightarrow{\terminal}\bbzero$ is an isomorphism.
\end{itemize}
We denote by $\DEFC$ the full subcategory in $\EFC$ consisting of
decent endofunctors. One can show that decent endofunctors are closed
under composition; hence, $\DEFC$ inherits the strict monoidal structure
from $\EFC$, and $(\DEFC,\singlecdot,\Id)$ becomes a strict monoidal
subcategory of $(\EFC,\singlecdot,\Id)$, where $\cdot$ denotes the
composition of endofunctors which is usually written simply as concatenation. \end{defn}

\begin{defn}
For $A\in\in\Cstar$, denote by $\mathfrak{O}_{A}\in\in\EFC$ the
endofunctor of the left maximal tensoring 
\[
\mathfrak{O}_{A}\coloneqq A\otimes\mmph\colon(B\xrightarrow{\varphi}B')\mapsto(A\otimes B\xrightarrow{A\otimes\varphi}A\otimes B'),
\]
and for a $*$-homomorphism $\varphi\colon A\to A'$, define the natural
transformation
\[
\mathfrak{O}_{\varphi}\colon\mathfrak{O}_{A}\Rightarrow\mathfrak{O}_{A'},\qquad\mathfrak{O}_{\varphi}B\colon A\otimes B\xrightarrow{\varphi\otimes B}A'\otimes B.
\]
\end{defn}

\begin{lem}\label{lem:OA-decent-exact}
For every $A\in\in\Cstar$, the endofunctor
$\mathfrak{O}_{A}$ is decent and exact\footnote{A functor is called \emph{exact} if it preserves short exact sequences.}.
\end{lem}

\begin{defn}\label{def:bimon-C-DEFC}
Introduce the monoidal functor
\[
(\mathfrak{O},\mathfrak{O}^{2},\mathfrak{O}^{0})\colon\Cstar\to\DEFC,
\]
where:
\begin{itemize}
\item $\mathfrak{O}\colon\Cstar\to\DEFC\colon(A\xrightarrow{\varphi}A')\mapsto(\mathfrak{O}_{A}\xRightarrow{\mathfrak{O}_{\varphi}}\mathfrak{O}_{A'})$;
\item $(\mathfrak{O}^{2})_{A,A'}\colon\mathfrak{O}_{A}\mathfrak{O}_{A'}\Rightarrow\mathfrak{O}_{A\otimes A'},~(\mathfrak{O}^{2})_{A,A'}B\colon A\otimes(A'\otimes B)\xrightarrow{(\alpha^{\otimes})^{-1}}(A\otimes A')\otimes B$;
\item $\mathfrak{O}^{0}\colon\Id\Rightarrow\mathfrak{O}_{\tensorunit},~\mathfrak{O}^{0}B\colon B\xrightarrow{(\lambda^{\otimes})^{-1}}\tensorunit\otimes B$.
\end{itemize}
One can show that $\mathfrak{O}$ is fully faithful, and we call it
the \emph{tensor embedding functor}.
\end{defn}

\subsection{\label{subsec:defc-examples}Examples of decent endofunctors}

\begin{defn}\label{def:bbK}
We define the endofunctor $\fM_{n}\colon\Cstar\to\Cstar$
which sends a $C^{*}$-algebra $B$ to the $C^{*}$-algebra of $n$-by-$n$-matrices
with entries in $B$.

Similarly, one can define the \emph{stabilization} endofunctor $\fK\colon\Cstar\to\Cstar$
which sends a $C^{*}$-algebra $B$ to the ``compact $\mathbb{N}\textrm{-by-}\mathbb{N}$-matrices''
with entries in~$B$. The formal definition of $\fK$ will be given
in Section~\ref{subsec:hilbmod-prelim} (Definition~\ref{def:bbK-1}).
\end{defn}

\begin{defn}\label{def:Cb}
Let $C_{b}(\mathsf{X},B)$ be the $C^{*}$-algebra
of continuous bounded functions on a locally compact Hausdorff space
$\mathsf{X}$ taking values in $B$, and let $C_{0}(\mathsf{X},B)$
be the ideal in $C_{b}(\mathsf{X},B)$ comprised of the functions
vanishing at infinity. These $C^{*}$-algebras give rise to the endofunctors
$\mathfrak{C}^{b}_{\mathsf{X}}\coloneqq C_{b}(\mathsf{X},\mmph)$
and $\C_{\mathsf{X}}\coloneqq C_{0}(\mathsf{X},\mmph)$.

Every continuous (resp. proper\footnote{A map is called proper if the inverse image of a pre-compact set is
pre-compact.} continuous) map $f\colon\mathsf{X}\to\mathsf{Y}$ induces the natural
transformation $\mathfrak{C}^{b}_{f}\colon\mathfrak{C}^{b}_{\mathsf{Y}}\Rightarrow\mathfrak{C}^{b}_{\mathsf{X}}$
(resp. $\C_{f}\colon\C_{\mathsf{Y}}\Rightarrow\C_{\mathsf{X}}$) given
by the formula $g\mapsto g\circ f$.
\end{defn}

It is easy to show that the endofunctors $\fM_{n}$, $\fK$, $\mathfrak{C}^{b}_{\mathsf{X}}$,
$\C_{\mathsf{X}}$ are decent.

\subsection{Good endofunctors}
\begin{defn}\label{def:labeled-endof}
A \emph{labeled endofunctor} is a pair
$(F,\{\kappa^{A,F}\}_{A\in\in\Cstar})$, where $F\in\in\text{\ensuremath{\Cstar}}^{\text{\ensuremath{\Cstar}}}$
and $\{\kappa^{A,F}\}$ is a family of natural transformations $\left\{ \kappa^{A,F}\colon\mathfrak{O}_{A}F\Rightarrow F\mathfrak{O}_{A}\right\} _{A\in\in\Cstar}$
referred to as the \emph{labeling}, such that for every $*$-homomorphism
$\varphi\colon A\to A'$, the following diagrams commute:
\begin{equation}
\input{diagrams/d449.tikz}\label{eq:kappa-def-2}
\end{equation}
\begin{alignat}{2}
\input{diagrams/d450.tikz} & \qquad & \input{diagrams/d451.tikz}\label{eq:kappa-def22}
\end{alignat}
The composition of labeled endofunctors is defined as 
\[
(F,\{\kappa^{A,F}\}_{A})(G,\{\kappa^{A,G}\}_{A})\coloneqq(FG,\{\mathfrak{O}_{A}FG\xRightarrow{\kappa^{A,F}G}F\mathfrak{O}_{A}G\xRightarrow{F\kappa^{A,G}}FG\mathfrak{O}_{A}\}_{A}).
\]
It is associative by the left-hand side diagram in~\ref{eq:kappa-def22}.

Later on we shall often abuse notation and abbreviate $(F,\{\kappa^{A,F}\}_{A\in\in\Cstar})$
to~$F$. 
\end{defn}

\begin{defn}\label{def:lblnattr}
Let $(F,\{\kappa^{A,F}\}_{A})$ and $(G,\{\kappa^{A,G}\}_{A})$
be two labeled endofunctors. A \emph{labeled natural transformation}
\[
\alpha\colon(F,\{\kappa^{A,F}\}_{A})\Rightarrow(G,\{\kappa^{A,G}\}_{A})
\]
is a natural transformation $\alpha\colon F\Rightarrow G$ of the
underlying endofunctors, such that for all $A\in\in\Cstar$, the following
diagram commutes: 
\[
\input{diagrams/d452.tikz}
\]
\end{defn}

\begin{example}\label{exa:labelings}
Let $D$ be a $C^{*}$-algebra and let $\mathsf{X}$
and $\mathsf{Y}$ be locally compact Hausdorff spaces. The following
natural transformation turn $\Id$, $\fM_{n}$, $\fK$, $\mathfrak{C}^{b}_{\mathsf{X}}$,
$\C_{\mathsf{X}}$, $\mathfrak{O}_{D}$ into labeled endofunctors:
\begin{itemize}
\item $\kappa^{A,\Id}\coloneqq\id_{A}\colon A\to A$;
\item $\{\kappa^{A,\fM_{n}}\}$ and $\{\kappa^{A,\fK}\}$ given by the formula
 $a\otimes(m_{i,j})\mapsto(a\otimes m_{i,j})$;
\item $\{\kappa^{A,\mathfrak{C}^{b}_{\mathsf{X}}}\}$ and $\{\kappa^{A,\C_{\mathsf{X}}}\}$
given by the formula $a\otimes f\mapsto[x\mapsto a\otimes f(x)]$;
\item $\{\kappa^{A,\mathfrak{O}_{D}}\}$ given by the formula $a\otimes(d\otimes b)\mapsto d\otimes(a\otimes b)$.
\end{itemize}
For every continuous (resp. proper continuous) function $f\colon\mathsf{X}\to\mathsf{Y}$,
the natural transformation $\mathfrak{C}^{b}_{f}\colon\mathfrak{C}^{b}_{\mathsf{Y}}\Rightarrow\mathfrak{C}^{b}_{\mathsf{X}}$
(resp. $\C_{f}\colon\C_{\mathsf{Y}}\Rightarrow\C_{\mathsf{X}}$) is
obviously labeled.
\end{example}

\begin{defn}\label{cor:lefc-2cat}
Define the category $\LEFC$ with objects labeled
endofunctors and arrows labeled natural transformations. It naturally
carries the structure of a strict monoidal category $(\LEFC,\singlecdot,\Id)$
where $\cdot$ stands for the composition of labeled endofunctors
(which we usually abbreviate to concatenation). \end{defn}

\begin{defn}
A labeled endofunctor $(F,\{\kappa^{A,F}\}_{A\in\in\Cstar})$, where
$F$ is decent, is called a \emph{good endofunctor}.
\end{defn}

\begin{defn}
Define the category $\GEFC$ with objects good endofunctors and arrows
labeled natural transformations. $\GEFC$ is complete and carries
a structure of a strict monoidal category $(\GEFC,\singlecdot,\Id)$.
\end{defn}

Good endofunctors contain an important subcategory consisting of endofunctors
which behave like tensor multiplication. They are defined as follows. 
\begin{defn}
We say that $T\in\in\GEFC$ is \emph{tensor-type} if there is an
isomorphism $\omega_{T}\colon T\Rightarrow\mathfrak{O}_{T\mathbb{C}}$
such that 
\begin{equation}
\omega_{T}\mathbb{C}=(\rho^{\otimes}_{T\mathbb{C}})^{-1}\colon T\mathbb{C}\to T\mathbb{C}\otimes\mathbb{C}.\label{eq:tt}
\end{equation}
It is clear that the composition of two tensor-type endofunctors is
again tensor-type; hence, such endofunctors form a monoidal subcategory
in $\GEFC$.
\end{defn}

\begin{example}\label{exa:tt}
It is easy to show that $\fM_{n}$, $\fK$, $\C_{\mathsf{X}}$
discussed in Section~\ref{subsec:defc-examples} are tensor type.
Note also that $\mathfrak{C}^{b}_{\mathsf{X}}$ is not tensor-type,
unless $\mathsf{X}$ is compact.
\end{example}

\begin{defn}
Let $T,F\in\in\GEFC$, with $T$ tensor-type. We define the natural
transformation $\kappa^{T,F}$ as the composite
\[
\kappa^{T,F}\colon TF\xRightarrow{\omega_{T}F}\mathfrak{O}_{T\mathbb{C}}F\xRightarrow{\kappa^{T\mathbb{C},F}}F\mathfrak{O}_{T\mathbb{C}}\xRightarrow{F\omega^{-1}_{T}}FT.
\]
One can show, that $\kappa^{T,F}$ defined above is a labeled natural
transformation. \end{defn}

\begin{rem}\label{rem:tt-commute}
Tensor-type endofunctors commute. Specifically,
$\kappa^{T,S}\colon TS\Rightarrow ST$ is a natural isomorphism for
every two tensor-type endofunctors $T$ and~$S$.
\end{rem}

\subsection{Quotients of good endofunctors}
\begin{defn}\label{def:quot}
Let $F,G\in\in\EFC$ and let $\iota\colon F\Rightarrow G$
be a componentwise ideal inclusion. Define the \emph{quotient endofunctor}
$G/F$ at the object $B$ as $(G/F)B\coloneqq GB/FB$ and at the morphism
$\varphi\colon A\to B$ by a diagram chase on
\begin{equation}
\input{diagrams/d453.tikz}\label{eq:diagram-search-quotient-functor}
\end{equation}
Note that the diagram above also defines the componentwise quotient
projection $G\Rightarrow G/F$.

\end{defn}

\begin{lem}\label{lem:quotient-criterion}
For a componentwise ideal inclusion
$\iota\in\GEFC(F,G)$, there exists a unique family 
\[
\{\kappa^{A,G/F}\colon\mathfrak{O}_{A}(G/F)\Rightarrow(G/F)\mathfrak{O}_{A}\}_{A\in\in\Cstar}
\]
 satisfying the following properties:
\begin{itemize}
\item $\{\kappa^{A,G/F}\}_{A}$ turns $G/F$ into a good endofunctor;
\item the quotient projection $G\Rightarrow G/F$ is a labeled natural transformation.
\end{itemize}
\end{lem}

\begin{lem}\label{lem:abg}
Let $F,G,H\in\in\GEFC$, and let $\alpha$, $\beta$,
$\gamma$ be some natural transformations between the underlying endofunctors
as in the diagram 
\[
\input{diagrams/d503.tikz}
\]
with $\alpha$ monic. Then $\gamma$ is labeled whenever $\alpha$
and $\beta$ are.
\end{lem}

\begin{lem}\label{lem:26}
Let $F_{0},F,F_{1},G\in\in\GEFC$, and let $\iota$,
$\pi$, $\alpha$ be natural transformations between the underlying
endofunctors as in the solid part of the diagram 
\[
\input{diagrams/d462.tikz}
\]
in which the upper row is a componentwise short exact sequence, and
the left-hand triangle commutes. Then:
\begin{itemize}
\item there exists a unique $\beta\colon F_{1}\Rightarrow G$ making the
right-hand triangle commutative;
\item if $\iota$, $\pi$ and $\alpha$ are labeled, then so is $\beta$.
\end{itemize}
\end{lem}

\subsection{Well-pointed endofunctors}
\begin{defn}
We say that an endofunctor $F\in\in\DEFC$ is \emph{well-pointed
}if for every locally compact Hausdorff space $\mathsf{X}$, the natural
transformation $F\C_{\mathsf{X}}\Rightarrow\prod_{\mathsf{X}}F$ induced
by $\{F\ev_{x}\}_{x\in\mathsf{X}}$ is monic, where $\ev_{x}\colon\C_{\mathsf{X}}\Rightarrow\Id$
is the evaluation at $x\in\mathsf{X}$.
\end{defn}

\begin{example}\label{exa:wp}
Here we give some examples of well-pointed endofunctors:
\begin{itemize}
\item $\mathfrak{C}^{b}_{\mathsf{X}}$ and $\C_{\mathsf{X}}$ for any locally
compact Hausdorff space $\mathsf{X}$;
\item $\mathfrak{O}_{A}$ for any $A\in\in\Cstar$;
\item $\fM_{n}$ and $\fK$.
\end{itemize}
\end{example}

\begin{rem}
Note that the quotients of well-pointed endofunctors need not be well-pointed.
For example, let $f\in\mathfrak{T}I\mathbb{C}\cong C_{b}([1,\infty)\times[0,1])$
be such that 
\begin{alignat*}{2}
\supp(f)\in\left\{ (t,s)\middlepipe t\geq1,~\frac{1}{3t}\leq s\leq\frac{1}{t}\right\}  & \quad\textrm{and}\quad & f\left(t,\frac{1}{2t}\right)=1\textrm{ for all }t\geq1.
\end{alignat*}
Denote by $\widetilde{f}$ the class of $f$ in $\mathfrak{A}I\mathbb{C}$.
Then $\mathfrak{A}\ev_{s}\mathbb{C}(\widetilde{f})=0$ for all $s\in[0,1]$
but $\norm{\widetilde{f}}\geq1$, which implies that $\mathfrak{A}\coloneqq\mathfrak{T}/\mathfrak{T}_{0}$
is not well-pointed.
\end{rem}

\begin{prop}\label{prop:unique-labeling}
If $F\in\in\DEFC$ is a well-pointed
endofunctor, then there is at most one family $\{\kappa^{A,F}\colon\mathfrak{O}_{A}F\Rightarrow F\mathfrak{O}_{A}\}_{A\in\in\Cstar}$
which turns $F$ into a labeled endofunctor.
\end{prop}

\begin{prop}\label{prop:wp-lbl}
Let $F,G\in\in\GEFC$ such that the underlying
endofunctor of $G$ is well-pointed. Then every natural transformation
$\alpha\colon F\Rightarrow G$ between the underlying endofunctors
is labeled.
\end{prop}

\begin{defn}
Let $\mathcal{C}$ and $\mathcal{C}'$ be complete categories. We
say that a functor $F\colon\mathcal{C}\to\mathcal{C}'$ is
\begin{itemize}
\item \emph{product-separating} if $F\prod_{j}B_{j}\xrightarrow{\lc}\prod_{j}FB_{j}$
is a monomorphism for any family $\{B_{j}\}$ of objects in~$\mathcal{C}$;
\item \emph{monic-preserving} if it preserves monomorphisms.
\end{itemize}
\end{defn}

\begin{example}
$\mathfrak{C}^{b}_{\mathsf{X}}$, $\C_{\mathsf{X}}$, $\fM_{n}$ and
$\fK$ are obviously product-separating and monic-preserving. Note
also that $\mathfrak{O}_{A}$ is not monic-preserving in general~\cite[page 254]{Ped_pullback99}.
\end{example}

\begin{lem}\label{lem:psmp-compos}
Product-separating monic-preserving endofunctors
in $\DEFC$ are well-pointed and closed under composition. \end{lem}

\subsection{Homotopies of natural transformations}

Introduce\footnote{The endofunctor $I$ may look similar to that discussed in the introduction,
but being endowed with a labeling (see Example~\ref{exa:labelings})
it now represents a slightly different concept.} the good endofunctor $I\coloneqq\C_{[0,1]}$, which
we call the \emph{cylinder functor}, and define the natural transformation
$\ev_{t}\colon I\Rightarrow\Id$ of evaluation at $t\in[0,1]$.

\begin{defn}
We call two labeled natural transformations $\gamma_{0},\gamma_{1}\in\GEFC(F,G)$
\emph{homotopic} (written $\gamma_{0}\simeq\gamma_{1}$) if there
is $\Gamma\in\GEFC(F,GI)$ such that the following diagram commutes
for~$j=0,1$:
\[
\input{diagrams/d454.tikz}
\]
It can be shown that $\simeq$ is a monoidal congruence\footnote{A \emph{monoidal congruence} is an equivalence relation on parallel
arrows of a monoidal category which is respected by composition and
monoidal composition.} on $\GEFC$, and we shall denote by $[\gamma]$ the homotopy class
of $\gamma\in\GEFC$. 
\end{defn}

\begin{defn}
Introduce the strict monoidal category $\hGEFC\coloneqq\GEFC/\simeq$,
which we call the \emph{homotopy category of natural transformations}.
It has objects good endofunctors, morphisms homotopy
classes of labeled natural transformations, and composition and monoidal
composition given respectively by the formulas
\begin{alignat*}{3}
[\alpha]\circ[\beta] & \coloneqq[\alpha\circ\beta], & \qquad\qquad & [\alpha][\beta] & \coloneqq[\alpha\beta].
\end{alignat*}
\end{defn}

\begin{lem}\label{lem:stronger-homot}
Let $F,G,H\in\GEFC$, and let $\alpha\in\GEFC(F,GIH)$
such that
\[
\input{diagrams/d497.tikz}
\]
commutes for~$j=0,1$. Then $\alpha_{0}\simeq\alpha_{1}$.
\end{lem}

\subsection{\label{subsec:stab}Stabilization}

Recall the endofunctors $\fM_{n}$ and $\fK$ from Section~\ref{subsec:defc-examples},
and define the $C^{*}$-algebra of $n\textrm{-by-}n$-matrices with
complex entries $\oM_{n}\coloneqq\fM_{n}\mathbb{C}$, and the $C^{*}$-algebra
of compact operators on an infinite-dimensional Hilbert space $\oK\coloneqq\fK\mathbb{C}$.
Let us introduce the following $*$-homomorphisms: 
\begin{itemize}
\item $\widehat{\iota}_{00}\colon\mathbb{C}\to\oK\colon z\mapsto\diag(z,0,0,\dots)$
is the corner embedding;
\item $\widehat{\theta}\colon\oK\otimes\oK\to\oK$ is the isomorphism induced
by a unitary isomorphism 
\[
l_{2}(\mathbb{N})\otimes l_{2}(\mathbb{N})\xrightarrow{\cong}l_{2}(\mathbb{N});
\]
\item $\widehat{\theta}_{n}\colon\oM_{n}\otimes\oK\to\oK$ is the stability
isomorphism induced by a unitary isomorphism 
\[
l_{2}(\{1,\dots,n\})\otimes l_{2}(\mathbb{N})\xrightarrow{\cong}l_{2}(\mathbb{N});
\]
\item $\widehat{\mu}\colon\oK\oplus\oK\xrightarrow{}\oM_{2}\otimes\oK\xrightarrow{\widehat{\theta}_{2}}\oK$,
in which the first map is the obvious diagonal inclusion.
\end{itemize}
These $*$-homomorphisms give rise to the natural transformations
\begin{alignat*}{2}
 & \theta_{n}\colon\fM_{n}\fK\Rightarrow\fK, & \qquad\qquad & \theta\colon\fK\fK\Rightarrow\fK,\\
 & \iota_{00}\colon\Id\Rightarrow\fK,\hspc 3 &  & \mu\colon\fK\oplus\fK\xRightarrow{}\fK,
\end{alignat*}
where, for example, $\theta_{n}$ is given by \[
\fM_{n}\fK\xRightarrow{\cong}\mathfrak{O}_{\oM_{n}}\mathfrak{O}_{\oK}\xRightarrow{\mathfrak{O}^{2}}\mathfrak{O}_{\oM_{n}\otimes\oK}\xRightarrow{\mathfrak{O}_{\widehat{\theta}_{n}}}\mathfrak{O}_{\oK}\xRightarrow{\cong}\fK,
\]
and $\theta$, $\mu$, and $\iota_{00}$ are defined similarly. 

In fact, we shall not care about the exact formulas for these natural
transformations, except for $\iota_{00}$, which is given at components
as the ordinary corner embedding:
\[
\iota_{00}B\colon B\to\fK B\colon b\mapsto\diag(b,0,0,\dots).
\] \subsection{Generalized morphisms}

\begin{defn}
Let $A$, $B$ be $C^{*}$-algebras, let $F\in\in\GEFC$, and let
$\varphi_{0},\varphi_{1}\colon A\to FB$ be two $*$-homomorphisms.
We say that they are $F$-\emph{homotopic} (written $\varphi_{0}\simeq_{F}\varphi_{1}$)
if there is a $*$-homomorphism $\Phi\colon A\to FIB$ such that the
following diagram commutes for $j=0,1$:
\[
\input{diagrams/d455.tikz}
\]
\end{defn}

One can show that $F$-homotopy is an equivalence relation. We denote
by 
\[
\left[A,F,B\right]\coloneqq\hom(A,FB)/\simeq_{F}
\]
 the set of $F$-homotopy classes of $*$-homomorphisms from $A$
to $FB$, and write $[\varphi]$ for the class of $\varphi\colon A\to FB$
in $[A,F,B]$.

Define the binary operation $+$ on $[A,F\fK,B]$ by setting $[\varphi]+[\psi]\coloneqq[\varphi\uu\psi]$,
where $\varphi\uu\psi$ is the following composite:
\[
\varphi\uu\psi\colon A\xrightarrow{\Delta}A\oplus A\xrightarrow{\varphi\oplus\psi}F\fK B\oplus F\fK B=(F\fK\oplus F\fK)B\xrightarrow{\lc^{-1}B}F(\fK\oplus\fK)B\xrightarrow{F\mu B}F\fK B,
\]
where $\Delta$ is as in Definition~\ref{def:Delta} (which is explicitly
defined by the formula $\Delta(a)=(a,a)$). One can show that this
binary operation turns $[A,F\fK,B]$ into a commutative monoid whose
elements will be referred to as \emph{generalized morphisms}.
\begin{thm}\label{thm:bullet-bilin-assoc}
There is a well-defined associative
bilinear operation
\[
\bullet\colon[B,G\fK,C]\times[A,F\fK,B]\to[A,FG\fK,C]\colon([\psi],[\varphi])\mapsto[\psi]\bullet[\varphi]\coloneqq[\psi\bullet\varphi],
\]
where $\psi\bullet\varphi$ is the following composite:
\begin{equation}
\psi\bullet\varphi\colon A\xrightarrow{\varphi}F\fK B\xrightarrow{F\fK\psi}F\fK G\fK C\xrightarrow{F\kappa^{\fK,G}\fK C}FG\fK^{2}C\xrightarrow{FG\theta C}FG\fK C.\label{bullet}
\end{equation}
Furthermore, the corner embedding $\iota_{00}\colon\Id\Rightarrow\fK$
gives rise to the neutral element, i.e., 
\begin{alignat*}{2}
[\varphi]\bullet[\iota_{00}A]=[\varphi] & \qquad\textrm{and}\qquad & [\iota_{00}B]\bullet[\varphi]=[\varphi].
\end{alignat*}

\end{thm}

\begin{defn}\label{def:alpha-star}
Let $F,G\in\in\GEFC$, and let $\alpha\colon F\Rightarrow G$
be a natural transformation (not necessarily labeled) between the
underlying endofunctors. Denote by $\alpha_{*}$ the following map
of sets: 
\[
\alpha_{*}\colon[A,F,B]\to[A,G,B]\colon[A\xrightarrow{\varphi}FB]\mapsto[A\xrightarrow{\varphi}FB\xrightarrow{\alpha B}GB].
\]
\end{defn}

\begin{lem}\label{lem:weak-rangle}
The following statements hold: 
\begin{enumerate}[label=\textup{(\roman*)}]
\item The map $\alpha_{*}$ is well-defined;
\item If $\alpha\simeq\alpha'$, then $\alpha_{*}=\alpha_{*}'$;
\item $(\beta\circ\alpha)_{*}=\beta_{*}\circ\alpha{}_{*}$ for all $\alpha\colon F\Rightarrow G$
and $\beta\colon G\Rightarrow H$.
\end{enumerate}
\end{lem}

\begin{defn}
For $A,B\in\in\Cstar$ and $\alpha\in\GEFC(F,G)$, let us define the
map
\[
\left\langle \alpha\right\rangle \coloneqq(\alpha\fK)_{*}\colon[A,F\fK,B]\to[A,G\fK,B]\colon[\varphi]\mapsto[\alpha\fK B\circ\varphi].
\]
\end{defn}

\begin{thm}\label{thm:58}
The following statements hold: 
\begin{enumerate}[label=\textup{(\roman*)}]
\item\label{enu:mm}The map $\left\langle \alpha\right\rangle $ is a morphism
of monoids;
\item\label{enu:hmt}The map $\alpha\mapsto\left\langle \alpha\right\rangle $
is homotopy invariant, i.e., $\left\langle \alpha\right\rangle =\left\langle \beta\right\rangle $
whenever $\alpha\simeq\beta$;
\item\label{enu:v}If $\alpha\in\GEFC(F,G)$ and $\beta\in\GEFC(G,H)$,
then $\left\langle \beta\right\rangle \circ\left\langle \alpha\right\rangle =\left\langle \beta\circ\alpha\right\rangle $;
\item\label{enu:nat}If $\alpha\in\GEFC(F,F')$ and $\beta\in\GEFC(G,G')$,
then the following diagram commutes:
\[
\input{diagrams/d456.tikz}
\]
\end{enumerate}
\end{thm}

\subsection{Asymptotic adjunction}

We now proceed to define an asymptotic version of generalized morphisms.
\begin{defn}
Let us introduce the good endofunctors $\mathfrak{T}\coloneqq\mathfrak{C}^{b}_{[0,\infty)}$
and $\mathfrak{T}_{0}\coloneqq\mathfrak{C}_{[0,\infty)}$, and the
obvious componentwise ideal inclusion $\mathfrak{T}_{0}\Rightarrow\mathfrak{T}$.
We call the quotient $\mathfrak{A}\coloneqq\mathfrak{T}/\mathfrak{T}_{0}$
(with the labeling as in Lemma~\ref{lem:quotient-criterion}) the
\emph{asymptotic algebra} functor. \end{defn}

\begin{notation}\label{nota:fun-as}
It will be convenient to denote a function $t\mapsto f(t)$
by $\fun_{t}f(t)$. Similarly, for $f\in\mathfrak{T}B$, we denote
by $\as_{t}(f(t))$ the image of $t\mapsto f(t)$ in the quotient
$\mathfrak{A}B=\mathfrak{T}B/\mathfrak{T}_{0}B$.
\end{notation}

\begin{lem}\label{lem:fA-kappa}
The unique\footnote{That $\{\kappa^{A,\mathfrak{T}}\}_{A}$ is unique follows from Example~\ref{exa:wp}
and Proposition~\ref{prop:unique-labeling}.} labeling family $\{\kappa^{A,\mathfrak{T}}\}_{A}$
and the canonical\footnote{That is, a family with the properties in Lemma~\ref{lem:quotient-criterion}.}
labeling family $\{\kappa^{A,\mathfrak{A}}\}_{A}$ are given respectively
by the formulas
\begin{alignat*}{3}
\kappa^{A,\mathfrak{T}}\colon & \mathfrak{O}_{A}\mathfrak{T}\Rightarrow\mathfrak{T}\mathfrak{O}_{A}, & \qquad & \kappa^{A,\mathfrak{T}}B\colon &  & a\otimes\fun_{t}f(t)\mapsto\fun_{t}(a\otimes f(t)),\\
\kappa^{A,\mathfrak{A}}\colon & \mathfrak{O}_{A}\mathfrak{A}\Rightarrow\mathfrak{A}\mathfrak{O}_{A}, & \qquad & \kappa^{A,\mathfrak{A}}B\colon &  & a\otimes\as_{t}f(t)\mapsto\as_{t}(a\otimes f(t)).
\end{alignat*}
\end{lem}

\begin{proof}
Straightforward.
\end{proof}

We introduce the following labeled natural transformations:
\begin{itemize}
\item $\const\colon\Id\Rightarrow\mathfrak{T},~\const B\colon B\to\mathfrak{T}B\colon b\mapsto[t\mapsto b]$;
\item $\as\colon\mathfrak{T}\Rightarrow\mathfrak{A}$ the componentwise
quotient projection;
\item $\alpha\colon\Id\xRightarrow{\const}\mathfrak{T}\xRightarrow{\as}\mathfrak{A}$;
\item $\iota_{00}\colon\Id\Rightarrow\fK$ the corner embedding as in Section~\ref{subsec:stab}.
\end{itemize}
\begin{lem}\label{lem:alpha-A}
 There exists a homotopy $\alpha\mathfrak{A}\simeq\mathfrak{A}\alpha$.
\end{lem}

\begin{proof}
See~\cite[Proposition~2.8 ]{GHT}.
\end{proof}

\begin{defn}
For $A,B\in\in\Cstar$ and $F\in\in\hGEFC$, we introduce the monoid \begin{alignat*}{1}
\bbrackets{A,F,B} & \coloneqq\colim\left([A,F\fK,B]\xrightarrow{\left\langle F\alpha\right\rangle }[A,F\mathfrak{A}\fK,B]\xrightarrow{\left\langle F\alpha\mathfrak{A}\right\rangle }[A,F\mathfrak{A}^{2}\fK,B]\to\cdots\right),
\end{alignat*}
and for a $*$-homomorphism $\varphi\colon A\to F\mathfrak{A}^{n}\fK B$,
denote by $\bbrackets{\varphi}$ its class in $\bbrackets{A,F,B}$.
\end{defn}

\begin{defn}\label{def:asadj}
Let $S,N\in\in\hGEFC$. We say that there is an
\emph{asymptotic adjunction} $S\asadj N$ if there exist two labeled
natural transformations $\eta\colon\Id\Rightarrow NS$ and $\varepsilon\colon SN\Rightarrow\mathfrak{A}\fK$
(which we call the \emph{unit} and \emph{counit}) making the following
diagrams commutative in~$\hGEFC$:
\begin{alignat*}{2}
\input{diagrams/d457.tikz} & \qquad\qquad & \input{diagrams/d458.tikz}
\end{alignat*}
\end{defn}

\begin{thm}\label{thm:adj2brackets}
Let $A$ and $B$ be $C^{*}$-algebras, let
$S,N\in\in\hGEFC$, and let $S\asadj N$ be an asymptotic adjunction
witnessed by a unit $\eta$ and a counit~$\varepsilon$. Then the
formulas
\begin{align}
\Phi\colon\bbrackets{SA,\Id,B}\to\bbrackets{A,N,B} & \colon\bbrackets{SA\xrightarrow{\varphi}\mathfrak{A}^{n}\fK B}\mapsto\bbrackets{A\xrightarrow{\eta A}NSA\xrightarrow{N\varphi}N\mathfrak{A}^{n}\fK B},\label{eq:mut-izom}\\
\Psi\colon\bbrackets{A,N,B}\to\bbrackets{SA,\Id,B} & \colon\bbrackets{A\xrightarrow{\psi}N\mathfrak{A}^{n}\fK B}\nonumber \\
\mapsto\myllbracket SA\xrightarrow{S\psi}SN\mathfrak{A}^{n}\fK B & \xrightarrow{\varepsilon\mathfrak{A}^{n}\fK B}\mathfrak{A}\fK\mathfrak{A}^{n}\fK B\xrightarrow{\mathfrak{A}\kappa^{\fK,\mathfrak{A}^{n}}\fK B}\mathfrak{A}^{n+1}\fK\fK B\xrightarrow{\mathfrak{A}^{n+1}\theta B}\mathfrak{A}^{n+1}\fK B\myrrbracket\nonumber 
\end{align}
define mutually inverse isomorphisms of monoids natural in~$A$ and~$B$.
\end{thm}

%% file: diagrams/d371.tikz
\begin{tikzcd}[ampersand replacement=\&]
{\mathcal{C}_{3}} \&  \& {\mathcal{C}_{2}} \&  \& {\mathcal{C}_{1}.}
\arrow[""{name=0, anchor=center, inner sep=0}, "G'"{description}, from=1-3, to=1-1]
\arrow[""{name=1, anchor=center, inner sep=0}, "G''"{description}, shift left, curve={height=-30pt}, from=1-3, to=1-1]
\arrow[""{name=2, anchor=center, inner sep=0}, " G"{description}, shift right, curve={height=30pt}, from=1-3, to=1-1]
\arrow[""{name=3, anchor=center, inner sep=0}, "F'"{description}, from=1-5, to=1-3]
\arrow[""{name=4, anchor=center, inner sep=0}, "F''"{description}, shift left, curve={height=-30pt}, from=1-5, to=1-3]
\arrow[""{name=5, anchor=center, inner sep=0}, " F"{description}, shift right, curve={height=30pt}, from=1-5, to=1-3]
\arrow["\beta", Rightarrow, fixedlength,  from=2, to=0]
\arrow["\beta'", Rightarrow, fixedlength,  from=0, to=1]
\arrow["\alpha", Rightarrow, fixedlength,  from=5, to=3]
\arrow["\alpha'", Rightarrow, fixedlength,  from=3, to=4]
\end{tikzcd}

%% file: diagrams/d443.tikz
\begin{tikzcd}[column sep=tiny, ampersand replacement=\&]  
 \& {{(W\otimes(X\otimes Y))\otimes Z}} \\
{{((W\otimes X)\otimes Y)\otimes Z}} \&  \& {{W\otimes((X\otimes Y)\otimes Z)}} \\
{{(W\otimes X)\otimes(Y\otimes Z)}} \&  \& {{W\otimes(X\otimes(Y\otimes Z)),}}
\arrow["\alpha_{W,X\otimes Y,Z}^{\otimes}", from=1-2, to=2-3]
\arrow["\alpha_{W,X,Y}^{\otimes}\otimes Z", from=2-1, to=1-2]
\arrow["\alpha_{W\otimes X,Y,Z}^{\otimes}"', from=2-1, to=3-1]
\arrow["W\otimes\alpha_{X,Y,Z}^{\otimes}", from=2-3, to=3-3]
\arrow["\alpha_{W,X,Y\otimes Z}^{\otimes}", from=3-1, to=3-3]
\end{tikzcd}

%% file: diagrams/d442.tikz
\begin{tikzcd}[column sep=tiny, ampersand replacement=\&]  
{(X\otimes\tensorunit)\otimes Y} \& {\phantom{((W\otimes X)\otimes Y)\otimes Z}} \& {X\otimes(\tensorunit\otimes Y)} \\
{\phantom{((W\otimes X)\otimes Y)\otimes Z}} \& {X\otimes Y.} \& {\phantom{((W\otimes X)\otimes Y)\otimes Z}}
\arrow["\alpha_{X,\tensorunit,Y}^{\otimes}", from=1-1, to=1-3]
\arrow["\rho_{X}^{\otimes}\otimes Y"', from=1-1, to=2-2]
\arrow["X\otimes\lambda_{Y}^{\otimes}", from=1-3, to=2-2]
\end{tikzcd}

%% file: diagrams/d444.tikz
\begin{tikzcd}[ampersand replacement=\&]
{(FX\otimes FY)\otimes FZ} \&  \& {FX\otimes(FY\otimes FZ)} \\
{F(X\otimes Y)\otimes FZ} \&  \& {FX\otimes F(Y\otimes Z)} \\
{F((X\otimes Y)\otimes Z)} \&  \& {F(X\otimes(Y\otimes Z)),}
\arrow["\alpha_{FX,FY,FZ}^{\otimes}", from=1-1, to=1-3]
\arrow["F_{X,Y}^{2}\otimes FZ"', from=1-1, to=2-1]
\arrow["FX\otimes F_{Y,Z}^{2}", from=1-3, to=2-3]
\arrow["F_{X\otimes Y,Z}^{2}"', from=2-1, to=3-1]
\arrow["F_{X,Y\otimes Z}^{2}", from=2-3, to=3-3]
\arrow["F\alpha_{X,Y,Z}^{\otimes}"', from=3-1, to=3-3]
\end{tikzcd}

%% file: diagrams/d445.tikz
\begin{tikzcd}[ampersand replacement=\&]
{\tensorunit_{\mathcal{C}}\otimes FX} \& {FX} \\
{F\tensorunit_{\mathcal{C}'}\otimes FX} \& {F(\tensorunit_{\mathcal{C}'}\otimes X),}
\arrow["\lambda_{FX}^{\otimes}", from=1-1, to=1-2]
\arrow["F^{0}\otimes FX"', from=1-1, to=2-1]
\arrow["F_{\tensorunit,X}^{2}"', from=2-1, to=2-2]
\arrow["F\lambda_{X}^{\otimes}"', from=2-2, to=1-2]
\end{tikzcd}

%% file: diagrams/d446.tikz
\begin{tikzcd}[ampersand replacement=\&]
{FX\otimes\tensorunit_{\mathcal{C}}} \& {FX} \\
{FX\otimes F\tensorunit_{\mathcal{C}'}} \& {F(X\otimes\tensorunit_{\mathcal{C}'}).}
\arrow["\rho_{FX}^{\otimes}", from=1-1, to=1-2]
\arrow["FX\otimes F^{0}"', from=1-1, to=2-1]
\arrow["F_{X,\tensorunit}^{2}"', from=2-1, to=2-2]
\arrow["F\rho_{X}^{\otimes}"', from=2-2, to=1-2]
\end{tikzcd}

%% file: diagrams/d464.tikz
\begin{tikzcd}[ampersand replacement=\&]
{(\lim_{j}F_{j})G} \& {\lim_{j}(F_{j}G)} \\
 \& {F_{j}G,}
\arrow["\rc", Rightarrow, dashed, from=1-1, to=1-2]
\arrow["\pr_{j}G"', Rightarrow, from=1-1, to=2-2]
\arrow["\pr_{j}", Rightarrow, from=1-2, to=2-2]
\end{tikzcd}

%% file: diagrams/d465.tikz
\begin{tikzcd}[ampersand replacement=\&]
{G(\lim_{j}F_{j})} \& {\lim_{j}(GF_{j})} \\
 \& {GF_{j}.}
\arrow["\lc", Rightarrow, dashed, from=1-1, to=1-2]
\arrow["G\pr_{j}"', Rightarrow, from=1-1, to=2-2]
\arrow["\pr_{j}", Rightarrow, from=1-2, to=2-2]
\end{tikzcd}

%% file: diagrams/d466.tikz
\begin{tikzcd}[ampersand replacement=\&]
 \& {X} \&  \\
{X} \& {{X\oplus X}} \& {{X.}}
\arrow["{=}"', from=1-2, to=2-1]
\arrow["\Delta", dashed, from=1-2, to=2-2]
\arrow["{=}", from=1-2, to=2-3]
\arrow["{\pr_{1}}", from=2-2, to=2-1]
\arrow["{\pr_{2}}"', from=2-2, to=2-3]
\end{tikzcd}

%% file: diagrams/d447.tikz
\begin{tikzcd}[ampersand replacement=\&]
{A} \& {IB} \\
 \& {B,}
\arrow["\Phi", from=1-1, to=1-2]
\arrow["\varphi_{j}"', from=1-1, to=2-2]
\arrow["e_{j}", from=1-2, to=2-2]
\end{tikzcd}

%% file: diagrams/d448.tikz
\begin{tikzcd}[ampersand replacement=\&]
{B_{1}\oplus_{B}B_{2}} \& {B_{2}} \\
{B_{1}} \& {B,}
\arrow[from=1-1, to=1-2]
\arrow[""', from=1-1, to=2-1]
\arrow["\varphi_{2}", two heads, from=1-2, to=2-2]
\arrow["\varphi_{1}", from=2-1, to=2-2]
\end{tikzcd}

%% file: diagrams/d449.tikz
\begin{tikzcd}[ampersand replacement=\&]
{\mathfrak{O}_{A}F} \&  \& {F\mathfrak{O}_{A}} \\
{\mathfrak{O}_{A'}F} \&  \& {F\mathfrak{O}_{A'},}
\arrow["\kappa^{A,F}", Rightarrow, from=1-1, to=1-3]
\arrow["\mathfrak{O}_{\varphi}F"', Rightarrow, from=1-1, to=2-1]
\arrow["F\mathfrak{O}_{\varphi}", Rightarrow, from=1-3, to=2-3]
\arrow["\kappa^{A',F}", Rightarrow, from=2-1, to=2-3]
\end{tikzcd}

%% file: diagrams/d450.tikz
\begin{tikzcd}[column sep=large, ampersand replacement=\&]  
{\mathfrak{O}_{A}\mathfrak{O}_{A'}F} \& {\mathfrak{O}_{A}F\mathfrak{O}_{A'}} \& {F\mathfrak{O}_{A}\mathfrak{O}_{A'}} \\
{\mathfrak{O}_{A\otimes A'}F} \&  \& {F\mathfrak{O}_{A\otimes A'}\mmdc}
\arrow["\mathfrak{O}_{A}\kappa^{A',F}", Rightarrow, from=1-1, to=1-2]
\arrow["\mathfrak{O}_{A,A'}^{2}F"', Rightarrow, from=1-1, to=2-1]
\arrow["\kappa^{A,F}\mathfrak{O}_{A'}", Rightarrow, from=1-2, to=1-3]
\arrow["F\mathfrak{O}_{A,A'}^{2}", Rightarrow, from=1-3, to=2-3]
\arrow["\kappa^{A\otimes A',F}", Rightarrow, from=2-1, to=2-3]
\end{tikzcd}

%% file: diagrams/d451.tikz
\begin{tikzcd}[ampersand replacement=\&]
{\Id F} \& {F} \& {F\Id} \\
{\mathfrak{O}_{\mathbb{C}}F} \&  \& {F\mathfrak{O}_{\mathbb{C}}.}
\arrow["\mathfrak{O}^{0}F"', Rightarrow, from=1-1, to=2-1]
\arrow[" ="', Rightarrow, from=1-2, to=1-1]
\arrow[" =", Rightarrow, from=1-2, to=1-3]
\arrow["F\mathfrak{O}^{0}", Rightarrow, from=1-3, to=2-3]
\arrow["\kappa^{\mathbb{C},F}", Rightarrow, from=2-1, to=2-3]
\end{tikzcd}

%% file: diagrams/d452.tikz
\begin{tikzcd}[ampersand replacement=\&]
{\mathfrak{O}_{A}F} \&  \& {F\mathfrak{O}_{A}} \\
{\mathfrak{O}_{A}G} \&  \& {G\mathfrak{O}_{A}.}
\arrow["\kappa^{A,F}", Rightarrow, from=1-1, to=1-3]
\arrow["\mathfrak{O}_{A}\alpha"', Rightarrow, from=1-1, to=2-1]
\arrow["\alpha\mathfrak{O}_{A}", Rightarrow, from=1-3, to=2-3]
\arrow["\kappa^{A,G}"', Rightarrow, from=2-1, to=2-3]
\end{tikzcd}

%% file: diagrams/d453.tikz
\begin{tikzcd}[ampersand replacement=\&]
{0} \& {FA} \& {GA} \& {(G/F)A} \& {0} \\
{0} \& {FB} \& {GB} \& {(G/F)B} \& {0.}
\arrow[from=1-1, to=1-2]
\arrow["\iota A", from=1-2, to=1-3]
\arrow["F\varphi", from=1-2, to=2-2]
\arrow[from=1-3, to=1-4]
\arrow["G\varphi", from=1-3, to=2-3]
\arrow[from=1-4, to=1-5]
\arrow[dashed, from=1-4, to=2-4]
\arrow[from=2-1, to=2-2]
\arrow["\iota B", from=2-2, to=2-3]
\arrow[from=2-3, to=2-4]
\arrow[from=2-4, to=2-5]
\end{tikzcd}

%% file: diagrams/d503.tikz
\begin{tikzcd}[ampersand replacement=\&]
{F} \& {G} \\
 \& {H}
\arrow["\gamma", Rightarrow, from=1-1, to=1-2]
\arrow["\beta"', Rightarrow, from=1-1, to=2-2]
\arrow["\alpha", Rightarrow, from=1-2, to=2-2]
\end{tikzcd}

%% file: diagrams/d462.tikz
\begin{tikzcd}[ampersand replacement=\&]
{0} \& {F_{0}} \& {F} \& {F_{1}} \& {0} \\
 \&  \& {G}
\arrow[Rightarrow, from=1-1, to=1-2]
\arrow["\iota", Rightarrow, from=1-2, to=1-3]
\arrow[" 0"', Rightarrow, from=1-2, to=2-3]
\arrow["\pi", Rightarrow, from=1-3, to=1-4]
\arrow["\alpha", Rightarrow, from=1-3, to=2-3]
\arrow[Rightarrow, from=1-4, to=1-5]
\arrow["\beta", Rightarrow, dashed, from=1-4, to=2-3]
\end{tikzcd}

%% file: diagrams/d454.tikz
\begin{tikzcd}[ampersand replacement=\&]
{F} \& {GI} \\
 \& {G.}
\arrow["\gamma_{j}"', Rightarrow, from=1-1, to=2-2]
\arrow["\Gamma", Rightarrow, from=1-1, to=1-2]
\arrow["G\ev_{j}", Rightarrow, from=1-2, to=2-2]
\end{tikzcd}

%% file: diagrams/d497.tikz
\begin{tikzcd}[ampersand replacement=\&]
{F} \& {GIH} \\
 \& {GH}
\arrow["\alpha_{j}"', Rightarrow, from=1-1, to=2-2]
\arrow["\alpha", Rightarrow, from=1-1, to=1-2]
\arrow["G\ev_{j}H", Rightarrow, from=1-2, to=2-2]
\end{tikzcd}

%% file: diagrams/d455.tikz
\begin{tikzcd}[ampersand replacement=\&]
{A} \& {FIB} \\
 \& {FB.}
\arrow["\Phi", from=1-1, to=1-2]
\arrow["\varphi_{j}"', from=1-1, to=2-2]
\arrow["F\ev_{j}B", from=1-2, to=2-2]
\end{tikzcd}

%% file: diagrams/d456.tikz
\begin{tikzcd}[ampersand replacement=\&]
{[B,G\fK,C]\times[A,F\fK,B]} \&  \& {[A,FG\fK,C]} \\
{[B,G'\fK,C]\times[A,F'\fK,B]} \&  \& {[A,F'G'\fK,C].}
\arrow["\bullet", no head, from=1-1, to=1-3]
\arrow["\left\langle \beta\right\rangle \times\left\langle \alpha\right\rangle "', from=1-1, to=2-1]
\arrow["\left\langle \alpha\beta\right\rangle ", from=1-3, to=2-3]
\arrow["\bullet", no head, from=2-1, to=2-3]
\end{tikzcd}

%% file: diagrams/d457.tikz
\begin{tikzcd}[ampersand replacement=\&]
{S} \& {SNS} \\
 \& {\mathfrak{A}\fK S,}
\arrow["S\eta", Rightarrow, from=1-1, to=1-2]
\arrow["\alpha\iota_{00}S"', Rightarrow, from=1-1, to=2-2]
\arrow["\varepsilon S", Rightarrow, from=1-2, to=2-2]
\end{tikzcd}

%% file: diagrams/d458.tikz
\begin{tikzcd}[ampersand replacement=\&]
{N} \& {NSN} \\
 \& {N\mathfrak{A}\fK.}
\arrow["\eta N", Rightarrow, from=1-1, to=1-2]
\arrow["N\alpha\iota_{00}"', Rightarrow, from=1-1, to=2-2]
\arrow["N\varepsilon", Rightarrow, from=1-2, to=2-2]
\end{tikzcd}

%% file: sec2.tex
\section{\label{sec:Roe-functors}Roe functors}

In this section we define Roe functors (relative Roe functors) associated
to metric spaces (pairs of metric spaces) with bounded geometry, and
show that they carry the structure of good endofunctors. 

Before proceeding, we should warn the reader of a slight change in
notation. Recall that in the previous section we discussed the endofunctors
$\fK$, $\fM_{n}$, $\C_{\mathsf{X}}$ and the $C^{*}$-algebras $\oK=\fK\mathbb{C}$,
$\oM_{n}=\fM_{n}\mathbb{C}$, $C_{0}(\mathsf{X})=\C_{\mathsf{X}}\mathbb{C}$.
These notations were adopted from~\cite{makeev_gmaec}, where it
was essential to distinguish between the following pairs of isomorphic
but non-equal endofunctors: $\fK\cong\mathfrak{O}_{\oK}$, $\fM_{n}\cong\mathfrak{O}_{\oM_{n}}$,
and $\C_{\mathsf{X}}\cong\mathfrak{O}_{C_{0}(\mathsf{X})}$. In the
present paper this specification seems redundant and potentially confusing,
and we adjust our notation as follows:
\begin{itemize}
\item $\uK$ and $\uM_{n}$ will denote the endofunctors which were previously
written as $\fK$ and $\fM_{n}$;
\item We shall implicitly identify the following isomorphic endofunctors:
\begin{equation}
\uK\cong\mathfrak{O}_{\uK\mathbb{C}},\qquad\uM_{n}\cong\mathfrak{O}_{\uM_{n}\mathbb{C}},\qquad\C_{\mathsf{X}}\cong\mathfrak{O}_{C_{0}(\mathsf{X})}.\label{eq:convention}
\end{equation}
For example, we shall often write $f\otimes b\in\C_{\mathsf{X}}B$
instead of $(t\mapsto f(t)b)\in\C_{\mathsf{X}}B$ for $f\in C_{0}(\mathsf{X})$
and $b\in B$.
\end{itemize}

\subsection{\label{subsec:hilbmod-prelim}Preliminaries on Hilbert $C^{*}$-modules}

We start by recalling some definitions and basic facts from the theory
of Hilbert $C^{*}$-modules~\cite{lance1995}. Let $B$ be a $C^{*}$-algebra
and $E$ a Hilbert $B$-module. For every $\xi\in E$, denote by $\abs{\xi}$
the positive element $\abs{\xi}\coloneqq\langle\xi,\xi\rangle^{1/2}\in B$,
and recall that for $\xi,\eta\in E$, the following inequality holds: \begin{equation}
\abs{\langle\xi,\eta\rangle}\leq\norm{\xi}\abs{\eta}.\label{eq:Cauchy}
\end{equation}
For two Hilbert $B$-modules $E$ and $F$, denote by $\mathcal{L}_{B}(E,F)$
and $\mathcal{K}_{B}(E,F)$ respectively the set of adjointable and
compact (in the sense of Hilbert $C^{*}$-modules) operators from
$E$ to $F$. When $F=E$, we use the abbreviations $\mathcal{L}_{B}(E)$
and $\mathcal{K}_{B}(E)$. The $C^{*}$-algebra $\mathcal{L}_{B}(B)$,
where $B$ is viewed as a Hilbert $B$-module, is called the \emph{multiplier
algebra} of $B$ and is denoted by $\mathcal{M}(B)$. The \emph{strict
topology} on $\mathcal{M}(B)$ is generated by the seminorms
\[
m\mapsto\norm{mb},\qquad m\mapsto\norm{m^{*}b},\qquad\textrm{ for }b\in B.
\]
A net $\{m_{\lambda}\}\subset\mathcal{M}(B)$ \emph{strictly converges}
to $m\in\mathcal{M}(B)$ if $\norm{m_{\lambda}b-mb}\to0$ and $\norm{m^{*}_{\lambda}b-m^{*}b}\to0$
for every $b\in B$. In this case we write $\strlim_{\lambda}m_{\lambda}=m$
and say that $m$ is the \emph{strict limit} of $m_{\lambda}$.
A \emph{strictly Cauchy} sequence is defined similarly. The multiplier
algebra is \emph{strictly complete}, i.e., every strictly Cauchy
net strictly converges. If $\{m_{\lambda}\}$ and $\{m_{\lambda}'\}$
are two strictly convergent nets in $\mathcal{M}(B)$, and $\{m_{\lambda}\}$
is bounded, then
\[
\strlim(m_{\lambda}m_{\lambda}')=\strlim(m_{\lambda})\strlim(m_{\lambda}').
\]
For a countable set $X$, construct the following Hilbert $B$-module:
\[
l_{2}(X,B)\coloneqq\left\{ f\colon X\to B\middlepipe\sum_{x}f(x)^{*}f(x)\textrm{ is norm convergent}\right\} ,
\]
which is usually called the \emph{standard Hilbert $B$-module},
especially when $X=\mathbb{N}$. For $\xi\in l_{2}(X,B)$, we write
$\xi_{x}$ for the value of $\xi$ at $x\in X$.  For $b\in B$,
we denote by $e_{x}b$ the function $z\mapsto\delta_{z,x}b$, with
$\delta_{z,x}$ the Kronecker delta (for unital $B$ one can perceive
$e_{x}$ as the $x$th unit vector in $l_{2}(X,B)$).

For a countable (or finite) set $X$, we use $\mat_{x,y\in X}a_{x,y}$
(or just $\mat_{x,y}a_{x,y}$) to denote the formal $X\textrm{-by-}X$-matrix
with $a_{x,y}$ in the $x$th row and $y$th column.  Similarly,
diagonal matrices will be written as $\diag_{x}a_{x}$. The symbol
$\epsilon_{x,y}$ will stand for the $(x,y)$th matrix unit.

Let $B$ be a $C^{*}$-algebra, $X$ a countable set, and let $m\coloneqq\mat_{x,y\in X}m_{x,y}$
be a formal matrix with entries in $\mathcal{M}(B)$. If for all $\xi\in l_{2}(X,B)$
and all $x\in X$, the sums $\zeta_{x}\coloneqq\sum_{y\in X}m_{x,y}\xi_{y}$
and $\sum_{x\in X}\zeta^{*}_{x}\zeta_{x}$ are norm convergent in
$B$, then we can regard $m$ as an element of $\mathcal{L}_{B}\left(l_{2}(X,B)\right)$
which maps $\xi$ to $\sum_{x}e_{x}(\sum_{y}m_{x,y}\xi_{y})$. It
is clear that formal matrices do not necessarily represent adjointable
operators, but the following lemma shows that the reverse is true.
\begin{lem}\label{lem:mat-rep}
Every $m\in\mathcal{L}_{B}\left(l_{2}(X,B)\right)$
admits the matrix form $m=\mat_{x,y}m_{x,y}$, where \[
m_{x,y}\coloneqq\strlim_{\lambda}\left\langle e_{x}u_{\lambda},m\left(e_{y}u_{\lambda}\right)\right\rangle \in\mathcal{M}(B)
\]
for any approximate unit $u_{\lambda}$ in $B$.
\end{lem}

\begin{proof}
For all $m\in\mathcal{L}_{B}(l_{2}(X,B))$, $b\in B$, and $x,y\in X$
we have the inequality
\begin{multline*}
\norm{\left\langle e_{x}u_{\lambda},me_{y}u_{\lambda}\right\rangle b-\left(me_{y}b\right)_{x}}\leq\\
\leq\norm{\left\langle e_{x}u_{\lambda},m\left(e_{y}(u_{\lambda}-1)b\right)\right\rangle }+\norm{\left\langle e_{x}u_{\lambda},me_{y}b\right\rangle -\left(me_{y}b\right)_{x}}\leq\\
\leq\norm m\norm{(u_{\lambda}-1)b}+\norm{\left(u_{\lambda}-1\right)\left(me_{y}b\right)_{x}}
\end{multline*}
which implies that 
\begin{equation}
\left\langle e_{x}u_{\lambda},me_{y}u_{\lambda}\right\rangle b\xrightarrow[\lambda]{}\left(me_{y}b\right)_{x}.\label{eq:21-00}
\end{equation}
Using that $x$, $y$, and $m$ in (\ref{eq:21-00}) are arbitrary,
we get
\[
\left\langle e_{x}u_{\lambda},me_{y}u_{\lambda}\right\rangle ^{*}b=\left\langle e_{y}u_{\lambda},m^{*}e_{x}u_{\lambda}\right\rangle b\xrightarrow[\lambda]{}\left(m^{*}e_{x}b\right)_{y}.
\]
Thus, the net $\left\{ \left\langle e_{x}u_{\lambda},me_{y}u_{\lambda}\right\rangle \right\} _{\lambda}$
is strictly Cauchy, and hence, strictly convergent. Let $\xi$ be
an element of $l_{2}(X,B)$. The statement of the lemma follows from
\[
\left(m\xi\right)_{x}=(m\sum_{y\in X}e_{y}\xi_{y})_{x}=\sum_{y\in X}\left(me_{y}\xi_{y}\right)_{x}\stackrel{(\ref{eq:21-00})}{=}\sum_{y\in X}\lim_{\lambda}\left(\left\langle e_{x}u_{\lambda},me_{y}u_{\lambda}\right\rangle \xi_{y}\right)=\sum_{y\in X}m_{x,y}\xi_{y}.
\]
\end{proof}

\begin{defn}\label{def:bbK-1}
For a countable set $X$ and a $C^{*}$-algebra
$B$, denote by $\uK_{X}B$ the closure of the $*$-subalgebra in
$\mathcal{L}_{B}(l_{2}(X,B))$ consisting of $X$-by-$X$ matrices,
with entries in $B$, with all but finitely many entries equal to
zero. Equivalently, $\uK_{X}B$ is the $C^{*}$-algebra $\mathcal{K}_{B}(l_{2}(X,B))$.
It is easy to show that $\uK_{X}$ is functorial, and we call $\uK_{X}$
the \emph{stabilization} functor. For an arbitrary $x_{0}\in X$,
we denote by $\iota_{x_{0},x_{0}}$ the ``corner embedding'': 
\[
\iota_{x_{0},x_{0}}\colon\Id\Rightarrow\uK_{X},\qquad\iota_{x_{0},x_{0}}B\colon b\mapsto\mat_{x,y}\delta_{x,x_{0}}\delta_{y,x_{0}}b,\qquad\textrm{with }\delta\textrm{ the Kronecker delta}.
\]
\end{defn}

\begin{rem}
Since $\uK_{X}\cong\uK_{X'}$ for any two countable sets $X$ and
$X'$, and $\iota_{x_{0},x_{0}}\simeq\iota_{x'_{0},x'_{0}}$ for any
two points $x_{0},x_{0}'\in X$, we often write simply $\uK$ instead
of $\uK_{X}$, and $\iota_{00}$ instead of $\iota_{x_{0},x_{0}}$.

We finish this section with a well-known fact from the theory of $C^{*}$-algebras. \end{rem}

\begin{lem}
[{\cite[Theorem 2.2.5]{murphy1990}}]\label{lem:pos} Let $A$ be
a $C^{*}$-algebra, and $a,b,c\in A$. Then 
\begin{enumerate}[label=\textup{(\roman*)}]
\item\label{enu:pos1}if $a,b\in A$ are self-adjoint, then $a\leq b$
implies $c^{*}ac\leq c^{*}bc$;
\item\label{enu:pos2}if $0\leq a\leq b$, then $\norm a\leq\norm b$.
\end{enumerate}
\end{lem}

\subsection{Roe functors}
\begin{defn}\label{def:bounded-geometry}
A discrete metric space $X$ has \emph{bounded
geometry}~\cite{roe_coarse_lectures} if for every $R>0$, all $R$-balls
have uniformly bounded cardinalities, i.e., $\sup_{x\in X}\left|\ball_{R}(x)\right|<\infty$.
\end{defn}

\begin{defn}
Let $m\coloneqq\mat_{x,y\in X}m_{x,y}$ be a formal $X\textrm{-by-}X$-matrix.
The \emph{support} of $m$ is the set
\[
\supp(m)\coloneqq\{(x,y)\in X\times X\mid m_{x,y}\neq0\}.
\]
If $X$ is a metric space, one can also define the \emph{propagation}
of $m$:
\[
\prp(m)\coloneqq\sup\left\{ \dist(x,y)\mid(x,y)\in\supp(m)\right\} .
\]
\end{defn}

\begin{lem}\label{lem:norm-fM-of-norm-its-elements}
Let $X$ be a metric space
of bounded geometry, let $B$ be a $C^{*}$-algebra, and $m$ a formal
$X\textrm{-by-}X$-matrix of finite propagation with uniformly norm
bounded entries in $\mathcal{M}(B)$. Then $m\in\mathcal{L}_{B}(l_{2}(X,B))$
and
\[
\norm m\leq\sup_{x\in X}\left|\ball_{\prp(m)}(x)\right|\cdot\sup_{x,y\in X}\norm{m_{x,y}}.
\]
\end{lem}

\begin{proof}
Set $R\coloneqq\prp(m)$, $N\coloneqq\sup_{x\in X}\left|\ball_{R}(x)\right|$,
and $M\coloneqq\sup_{x,y}\norm{m_{x,y}}$. For every $x\in X$ and
every $\xi\in l_{2}(X,B)$, we obtain, using inequality (\ref{eq:Cauchy}),
that
\[
\absB{\sum_{y\in\ball_{R}(x)}m_{x,y}\xi_{y}}^{2}=\abs{\left\langle (m^{*}_{x,y})_{y\in\ball_{R}(x)},(\xi_{y})_{y\in\ball_{R}(x)}\right\rangle }^{2}\leq\norm{(m^{*}_{x,y})_{y}}^{2}\abs{(\xi_{y})_{y}}^{2}\leq NM^{2}\sum_{y\in\ball_{R}(x)}\abs{\xi_{y}}^{2}.
\]
Hence, 
\begin{multline*}
\abs{m\xi}^{2}=\sum_{x\in X}\absB{\sum_{y\in X}m_{x,y}\xi_{y}}^{2}=\sum_{x\in X}\absB{\sum_{y\in\ball_{R}(x)}m_{x,y}\xi_{y}}^{2}\leq NM^{2}\sum_{x\in X}\sum_{y\in\ball_{R}(x)}\abs{\xi_{y}}^{2}\\
=NM^{2}\sum_{y\in X}\sum_{x\in\ball_{R}(y)}\abs{\xi_{y}}^{2}\leq N^{2}M^{2}\sum_{y\in X}\abs{\xi_{y}}^{2}=N^{2}M^{2}\abs{\xi}^{2}
\end{multline*}
and the statement follows from Lemma~\ref{lem:pos}\ref{enu:pos2}.
\end{proof}

Lemma~\ref{lem:norm-fM-of-norm-its-elements} allows us to make the
following definition.
\begin{defn}
Let $B$ be a $C^{*}$-algebra, and $X$ a countable metric space
of bounded geometry. Define $\mathfrak{M}^{u}_{X}B$ as the norm closure
in $\mathcal{L}_{B}(l_{2}(X,B))$ of the following $*$-subalgebra:
\begin{equation}
\left\{ m=\mat_{x,y}b_{x,y}\middlepipe b_{x,y}\in B,\;\prp(m)<\infty,\;\sup_{x,y\in X}\norm{b_{x,y}}<\infty\right\} \subset\mathcal{L}_{B}(l_{2}(X,B)).\label{eq:1}
\end{equation}
We call $\mathfrak{M}^{u}_{X}B$ the \emph{uniform Roe algebra}
of $X$ with coefficients in~$B$.
\end{defn}

\begin{lem}\label{lem:dense}
Let $B$ be a $C^{*}$-algebra, and $B'\subset B$
a dense subset. Then the set
\[
\left\{ \mat_{x,y}b_{x,y}\in\mathfrak{M}^{u}_{X}B\middlepipe b_{x,y}\in B'\right\} 
\]
is dense in $\mathfrak{M}^{u}_{X}B$.
\end{lem}

\begin{proof}
This follows from Lemma~\ref{lem:norm-fM-of-norm-its-elements}.
\end{proof}

\begin{prop}\label{prop:M-end}
$\mathfrak{M}^{u}_{X}$ is an endofunctor of the
category of $C^{*}$-algebras and $*$-homomorphisms. Furthermore,
if $\phi\colon A\to B$ is a $*$-homomorphism and $\mat_{x,y}a_{x,y}\in\mathfrak{M}^{u}_{X}A$,
then the following equality holds: $(\mathfrak{M}^{u}_{X}\phi)(\mat_{x,y}a_{x,y})=\mat_{x,y}\phi(a_{x,y})$.
\end{prop}

\begin{proof}
Let $m\in\mathfrak{M}^{u}_{X}A$, let $\phi\colon A\to B$ be a $*$-homomorphism,
and let $\varphi\coloneqq\uK_{X}\phi\colon\uK_{X}A\to\uK_{X}B$. Enumerate
the elements of $X\eqqcolon\{x_{k}\}^{\infty}_{k=1}$, and for each
$N\in\mathbb{N}$ define the projection $P_{N}\coloneqq\sum^{N}_{k=1}\epsilon_{x_{k},x_{k}}\in\mathcal{L}_{A}\left(l_{2}(X,A)\right)$.
It is easy to see that for all $m\in\mathfrak{M}^{u}_{X}A$, both
$mP_{N}$ and $P_{N}m$ belong to $\uK_{X}A$.%

We claim that for all $m\in\mathfrak{M}^{u}_{X}A$, the following
strict limits exist and coincide: 
\begin{equation}
\strlim_{N\to\infty}\varphi(P_{N}m),\qquad\strlim_{N\to\infty}\varphi(mP_{N}),\qquad\strlim_{N\to\infty}\varphi(P_{N}mP_{N}).\label{eq:lims}
\end{equation}
To see this, note first that for all $x\in X$, $b\in B$, and $m'\in\mathfrak{M}^{u}_{X}A$
such that $\prp(m')<\infty$, the following holds: 
\begin{alignat}{2}
 & \varphi\left((P_{N}-P_{M})m'\right)(e_{x}b)\xrightarrow[N,M\to\infty]{}0, & \qquad\quad & \varphi\left(m'(P_{N}-P_{M})\right)(e_{x}b)\xrightarrow[N,M\to\infty]{}0,\label{eq:lims111}\\
 & \left(\varphi(m'P_{N})-\varphi(P_{N}m'P_{N})\right)(e_{x}b)\xrightarrow[N\to\infty]{}0, &  & \left(\varphi(P_{N}m')-\varphi(P_{N}m'P_{N})\right)(e_{x}b)\xrightarrow[N\to\infty]{}0.\label{eq:lims222}
\end{alignat}
Using that matrices of finite propagation are dense in $\mathfrak{M}^{u}_{X}A$,
and $\spn\{e_{x}b\mid x\in X,~b\in B\}$ is dense in $l_{2}(X,B)$,
we conclude from~(\ref{eq:lims111}) that $\left\{ \varphi(P_{N}m)\right\} ^{\infty}_{N=1}$
and $\left\{ \varphi(mP_{N})\right\} ^{\infty}_{N=1}$ are strictly
Cauchy for all $m\in\mathfrak{M}^{u}_{X}A$, and hence, the first
two limits in (\ref{eq:lims}) exist. Similarly, we conclude from
(\ref{eq:lims222}) that the third limit in (\ref{eq:lims}) exists
and is equal to both the first and the second ones. The claim is proved.

We proceed to show that the map
\[
\overline{\varphi}\colon\mathfrak{M}^{u}_{X}A\to\mathcal{L}_{B}(l_{2}(X,B))\colon m\mapsto\strlim_{N\to\infty}\varphi\left(P_{N}m\right).
\]
is a $*$-homomorphism. Indeed, using the claim and keeping in mind
that the sequence $\left\{ \varphi(P_{n}m)\right\} ^{\infty}_{n=1}$
is norm bounded, we obtain the equality
\[
\overline{\varphi}(mm')=\strlim_{n}\varphi(P_{n}mm'P_{n})=\strlim_{n}\varphi(P_{n}m)\strlim_{n}\varphi(m'P_{n})=\overline{\varphi}(m)\overline{\varphi}(m').
\]
It is clear that $\overline{\varphi}$ is linear and preserves adjoints.

Finally, using Lemma~\ref{lem:mat-rep}, we obtain the equality
\begin{equation}
\left(\overline{\varphi}\left(m\right)\right)_{x,y}=\strlim_{\lambda}\left\langle e_{x}u_{\lambda},\overline{\varphi}\left(m\right)\left(e_{y}u_{\lambda}\right)\right\rangle =\strlim_{\lambda}\left(u_{\lambda}\phi\left(m_{x,y}\right)u_{\lambda}\right)=\phi\left(m_{x,y}\right),\label{eq:111}
\end{equation}
which implies that $\overline{\varphi}$ cannot increase propagation, and hence, $\overline{\varphi}\left(\mathfrak{M}^{u}_{X}A\right)\subset\mathfrak{M}^{u}_{X}B\subset\mathcal{L}_{B}\left(l_{2}(X,B)\right)$. 

Now we can define $\mathfrak{M}^{u}_{X}\phi\coloneqq\overline{\varphi}$.
Functoriality of $\mathfrak{M}^{u}_{X}$ follows from the equality~(\ref{eq:111}).
\end{proof}

\begin{prop}\label{prop:MX-properties}
Let $X$ be a discrete
metric space of bounded geometry. Then $\mathfrak{M}^{u}_{X}$
is: 
\begin{enumerate*}[label=\textup{(\roman*)}]
\item\label{enu:x-mp}monic-preserving;
\item\label{enu:x-ep}epic-preserving;
\item\label{enu:x-iip}ideal-inclusion-preserving;
\item\label{enu:x-pp}pullback-preserving;
\item\label{enu:x-zp}zero-object-preserving;
\item\label{enu:x-d}decent;
\item\label{enu:x-ps}product-separating;
\item\label{enu:x-wp}well-pointed.
\end{enumerate*}
\end{prop}

\begin{proof}
Statement~\ref{enu:x-mp} readily follows from second part of Proposition~\ref{prop:M-end}.
Statement~\ref{enu:x-ep} follows from the fact that for an epimorphism
$\psi\colon A\to B$, the image of $\mathfrak{M}^{u}_{X}\psi$ is
dense in $\mathfrak{M}^{u}_{X}B$. If $i\colon J\to A$ is an ideal
inclusion, then $\mathfrak{M}^{u}_{X}i$ is monic by statement~\ref{enu:x-mp},
and 
\[
(\mathfrak{M}^{u}_{X}i)(J)=\overline{\{m\coloneqq\mat_{x,y}i(c_{x,y})\mid c_{x,y}\in J,~\prp(m)<\infty\}}
\]
is obviously an ideal in $\mathfrak{M}^{u}_{X}A$, which proves statement~\ref{enu:x-iip}.
Statement~\ref{enu:x-zp} is obvious. Statement~\ref{enu:x-d} follows
from statements~\ref{enu:x-pp} and~\ref{enu:x-zp}. Statement~\ref{enu:x-ps}
is straightforward to prove. Statement~\ref{enu:x-wp} follows from
statement~\ref{enu:x-ps} and Lemma~\ref{lem:psmp-compos}.

Now we only need to prove statement~\ref{enu:x-pp}. Let $D$, $D_{1}$,
$D_{2}$ be $C^{*}$-algebras, let $\varphi_{j}\colon D_{j}\to D$
for $j=1,2$ be $*$-homomorphisms, and let 
\[
\Xi\colon\mathfrak{M}^{u}_{X}(D_{1}\oplus D_{2})\to\mathfrak{M}^{u}_{X}D_{1}\oplus\mathfrak{M}^{u}_{X}D_{2}
\]
be the canonical $*$-homomorphism induced by $\mathfrak{M}^{u}_{X}\pi_{j}\colon\mathfrak{M}^{u}_{X}(D_{1}\oplus D_{2})\to\mathfrak{M}^{u}_{X}D_{j}$
for $j=1,2$, where $\pi_{j}\colon D_{1}\oplus D_{2}\to D_{j}$ are
the projections. One directly checks that $\Xi$ is given explicitly
by the formula
\[
\Xi:\mat_{x,y}(a_{x,y},b_{x,y})\mapsto(\mat_{x,y}a_{x,y},\mat_{x,y}b_{x,y}).
\]
It is clear that $\Xi$ is monic and has a dense image. Hence, $\Xi$
is an isomorphism. To finish the proof it remains only to note, using
the second part of Proposition~\ref{prop:M-end}, that 
\[
m\in\mathfrak{M}^{u}_{X}(D_{1}\underset{D}{\oplus}D_{2})\subset\mathfrak{M}^{u}_{X}(D_{1}\oplus D_{2})
\]
 if and only if 
\[
\Xi(m)\in\mathfrak{M}^{u}_{X}D_{1}\underset{\mathfrak{M}^{u}_{X}D}{\oplus}\mathfrak{M}^{u}_{X}D_{2}\subset\mathfrak{M}^{u}_{X}D_{1}\oplus\mathfrak{M}^{u}_{X}D_{2}.
\]
\end{proof}

\begin{prop}\label{prop:MX-labeled}
For every $A\in\in\Cstar$, there is a well-defined
natural transformation
\[
\kappa^{A,\mathfrak{M}^{u}_{X}}\colon\mathfrak{O}_{A}\mathfrak{M}^{u}_{X}\Rightarrow\mathfrak{M}^{u}_{X}\mathfrak{O}_{A},\qquad\kappa^{A,\mathfrak{M}^{u}_{X}}B\colon a\otimes\mat_{x,y}(m_{x,y})\mapsto\mat_{x,y}(a\otimes m_{x,y}).
\]
Furthermore, the family $\{\kappa^{A,\mathfrak{M}^{u}_{X}}\}_{A\in\in\Cstar}$
turns $\mathfrak{M}^{u}_{X}$ into a labeled endofunctor.
\end{prop}

\begin{proof}
For $C^{*}$-algebras $A$ and $B$, denote by $A^{+}$ and $B^{+}$
their unitizations. Using that $\mmph\otimes A$, $A\otimes\mmph$,
and $\mathfrak{M}^{u}_{X}$ are ideal-inclusion-preserving endofunctors (see Lemma~\ref{lem:OA-decent-exact} and Proposition~\ref{prop:MX-properties}),
we obtain the following ideal inclusions:
\begin{alignat*}{1}
 & A\otimes B\vartriangleleft A\otimes B^{+}\vartriangleleft A^{+}\otimes B^{+},\\
 & \mathfrak{M}^{u}_{X}(A\otimes B)\vartriangleleft\mathfrak{M}^{u}_{X}(A^{+}\otimes B^{+}).
\end{alignat*}
The $*$-homomorphisms
\begin{alignat*}{1}
A\to\mathfrak{M}^{u}_{X}(A^{+}\otimes B^{+}) & \colon a\mapsto\diag_{x\in X}(a\otimes1),\\
\mathfrak{M}^{u}_{X}B\to\mathfrak{M}^{u}_{X}(A^{+}\otimes B^{+}) & \colon m\mapsto\mat_{x,y\in X}(1\otimes m_{x,y})
\end{alignat*}
have commuting ranges, and hence, by the universal property of the
maximal tensor product, give rise to the $*$-homomorphism
\[
A\otimes\mathfrak{M}^{u}_{X}B\to\mathfrak{M}^{u}_{X}(A^{+}\otimes B^{+})\colon a\otimes m\mapsto\mat_{x,y}(a\otimes m_{x,y})
\]
which clearly factors through the $C^{*}$-subalgebra $\mathfrak{M}^{u}_{X}(A\otimes B)\subset\mathfrak{M}^{u}_{X}(A^{+}\otimes B^{+})$.
The whole construction is clearly natural in $B$, and the first statement
is proved. The second statement is straightforward. 
\end{proof}

\begin{cor}\label{cor:MX-is-good}
$(\mathfrak{M}^{u}_{X},\{\kappa^{A,\mathfrak{M}^{u}_{X}}\}_{A})\in\in\GEFC$.
\end{cor}

\begin{proof}
This follows from Propositions~\ref{prop:MX-properties}\ref{enu:x-d}
and~\ref{prop:MX-labeled}.
\end{proof}

\begin{rem}
Note that the ordinary uniform Roe algebra and Roe algebra~\cite{higson-roe2000analytic,roe_coarse_lectures}
of $X$ can be expressed in our terms: \[
C^{*}_{u}(X)\cong\mathfrak{M}^{u}_{X}\mathbb{C},\qquad C^{*}(X)\cong\mathfrak{M}^{u}_{X}\uK\mathbb{C}.
\]
Motivated by this fact, we call $\mathfrak{M}^{u}_{X}$ and $\mathfrak{M}_{X}\coloneqq\mathfrak{M}^{u}_{X}\uK$,
respectively, the \emph{uniform Roe functor} and \emph{Roe functor}. 
\end{rem}

We finish this subsection with a simple but useful fact.
\begin{lem}\label{lem:norm-eval}
For $m\in\mathfrak{M}^{u}_{X}IB$, set $m_{t}\coloneqq\mathfrak{M}^{u}_{X}\ev_{t}B(m)$.
Then $\norm m=\sup_{t}\norm{m_{t}}$.
\end{lem}

\begin{proof}
It is straightforward to check that for every $\eta\in\mathcal{L}_{IB}(l_{2}(X,IB))$
there is a well-defined $\eta_{t}\in\mathcal{L}_{IB}(l_{2}(X,B))$
given by the formula $x\mapsto\eta(x)(t)$, and that $\norm{\eta}=\sup_{t}\norm{\eta_{t}}$.
For every $\varepsilon>0$, there is $\xi\in\mathcal{L}_{IB}(l_{2}(X,IB))$
such that $\norm{\xi}=1$ and $\norm{m\xi}>\norm m-\varepsilon$.
Using that $\norm{\xi}\geq\norm{\xi_{t}}$, we obtain 
\[
\sup_{t}\norm{m_{t}}\geq\sup_{t}\norm{m_{t}\xi_{t}}=\sup_{t}\norm{(m\xi)_{t}}=\norm{m\xi}>\norm m-\varepsilon,
\]
and hence $\norm m\leq\sup_{t}\norm{m_{t}}$. The reverse inequality
follows from the fact that $m\mapsto m_{t}$ is a $*$-homomorphism
for every $t$.
\end{proof}

\subsection{Relative Roe functors}

Let $X$ be a discrete metric space of bounded geometry, and let $X_{0}\subset X$
be some subspace (possibly empty). For every $B\in\in\Cstar$,
we introduce the following $C^{*}$-algebra:
\[
\mathfrak{M}^{u}_{X\supset X_{0}}B\coloneqq\overline{\bigcup_{R>0}\left\{ m\in\mathfrak{M}^{u}_{X}B\mid\supp(m)\subset\nbhd_{R}(X_{0})\right\} }\subset\mathfrak{M}^{u}_{X}B,
\]
where $\nbhd_{R}(X_{0})$ stands for the $R$-neighborhood of $X_{0}$
in $X$ (we adopt the convention that $\nbhd_{R}(\varnothing)=\varnothing$
for all $R>0$).

It is clear that $\mathfrak{M}^{u}_{X\supset X_{0}}B$ is an ideal
in $\mathfrak{M}^{u}_{X}B$ for any $C^{*}$-algebra $B$. Furthermore,
if $\varphi\colon A\to B$ is a $*$-homomorphism, then $\mathfrak{M}^{u}_{X}\varphi$
maps $\mathfrak{M}^{u}_{X\supset X_{0}}A$ into $\mathfrak{M}^{u}_{X\supset X_{0}}B$
which allows us to define the $*$-homomorphism
\[
\mathfrak{M}^{u}_{X\supset X_{0}}\varphi\coloneqq\mathfrak{M}^{u}_{X}\varphi|_{\mathfrak{M}^{u}_{X\supset X_{0}}A}.
\]

\begin{lem}
$\mathfrak{M}^{u}_{X\supset X_{0}}$ is an endofunctor. Furthermore:
\begin{itemize}
\item There is an obvious componentwise ideal inclusion
$\mathfrak{M}^{u}_{X\supset X_{0}}\Rightarrow\mathfrak{M}^{u}_{X}$;
\item $\mathfrak{M}^{u}_{X\supset X_{0}}$ is monic-preserving, epic-preserving,
and ideal-inclusion-preserving;
\item $\mathfrak{M}^{u}_{X\supset X_{0}}$ is well-pointed;
\item $\mathfrak{M}^{u}_{X\supset X_{0}}$ is a good endofunctor with the
labeling as in Proposition~\ref{prop:MX-labeled}.
\end{itemize}
\end{lem}

\begin{proof}
Functoriality of $\mathfrak{M}^{u}_{X\supset X_{0}}$ follows from
that of $\mathfrak{M}^{u}_{X}$. The rest is straightforward (cf.~Proposition~\ref{prop:MX-properties},
and Corollary~\ref{cor:MX-is-good}). \end{proof}

\begin{defn}
For $X$ and $X_{0}$ as above,  define the \emph{relative uniform
Roe functor} 
\[
\mathfrak{N}^{u}_{X,X_{0}}\coloneqq\mathfrak{M}^{u}_{X}/\mathfrak{M}^{u}_{X\supset X_{0}},
\]
endowed with the structure of a good endofunctor as in Lemma~\ref{lem:quotient-criterion},
and denote by $q_{X,X_{0}}\colon\mathfrak{M}^{u}_{X}\Rightarrow\mathfrak{N}^{u}_{X,X_{0}}$
the quotient projection (which is a labeled natural transformation
again by Lemma~\ref{lem:quotient-criterion}).

Note that $\mathfrak{M}^{u}_{X\supset\varnothing}\cong0$, and therefore
$\mathfrak{N}^{u}_{X,\varnothing}\cong\mathfrak{M}^{u}_{X}$.
\end{defn}

\begin{lem}\label{lem:N-monpres}
$\mathfrak{N}^{u}_{X,X_{0}}$ is monic-preserving. \end{lem}

\begin{proof}
Let $B$ be a $C^{*}$-algebra and $A\subset B$ a $C^{*}$-subalgebra.
Bearing in mind that $\mathfrak{M}^{u}_{X}$ and $\mathfrak{M}^{u}_{X\supset X_{0}}$
are monic-preserving we have the following inclusions:
\begin{alignat*}{5}
 &  &  &  & \mathfrak{M}^{u}_{X}A & ~ & \subset & ~ & \mathfrak{M}^{u}_{X}B,\\
\mathfrak{M}^{u}_{X\supset X_{0}}A & ~\subset~ &  &  & \mathfrak{M}^{u}_{X\supset X_{0}}B &  & \subset &  & \mathfrak{M}^{u}_{X}B.
\end{alignat*}
To prove the statement of the lemma it suffices to show that
\[
\mathfrak{M}^{u}_{X\supset X_{0}}B\cap\mathfrak{M}^{u}_{X}A=\mathfrak{M}^{u}_{X\supset X_{0}}A.
\]
To this end, for every $N\in\mathbb{N}$, introduce the projection
$P_{N}\coloneqq\sum_{x\in\nbhd_{N}(X_{0})}\epsilon_{x,x}$ and note
that for every $m\in\mathfrak{M}^{u}_{X\supset X_{0}}B\cap\mathfrak{M}^{u}_{X}A$,
we have
\[
\mathfrak{M}^{u}_{X\supset X_{0}}A\ni P_{N}\cdot m\cdot P_{N}\xrightarrow[N\to\infty]{}m,
\]
which in turn implies that $\mathfrak{M}^{u}_{X\supset X_{0}}B\cap\mathfrak{M}^{u}_{X}A\subset\mathfrak{M}^{u}_{X\supset X_{0}}A$.
The reverse inclusion is obvious.
\end{proof}

\begin{notation}
Let $X$ and $X_{0}$ be as above. For any $m\coloneqq\mat_{x,y}m_{x,y}\in\mathfrak{M}^{u}_{X}B$,
we shall denote by $\qmat_{x,y}m_{x,y}$ its image in $\mathfrak{N}^{u}_{X,X_{0}}B$
under the quotient projection $q_{X,X_{0}}B$.
\end{notation}

\begin{rem}\label{rem:expl-lbl-N}
Later on it will be useful to have an explicit
formula for the labeling of $\mathfrak{N}^{u}_{X,X_{0}}$, which takes
the following form:
\[
\kappa^{A,\mathfrak{N}^{u}_{X,X_{0}}}B\colon a\otimes\qmat_{x,y}m_{x,y}\mapsto\qmat_{x,y}(a\otimes m_{x,y}).
\]
\end{rem}

%% file: sec3.tex
\section{\label{sec:main}Continuous functions and Roe functors}

In this section we prove the main result of this paper: tensoring
with continuous functions on a scalable pair of proper metric spaces
and uniform Roe functors associated to its discretization are asymptotically
adjoint.  We start with several preliminary definitions.

Let $X$ be a set, and $X_{0}\subset X$ a subset. The tuple $(X,X_{0})$
is called a \emph{pair} of sets. A morphism of pairs $f\colon(X,X_{0})\to(Y,Y_{0})$
is a map $f\colon X\to Y$ such that $f(X_{0})\subset Y_{0}$. Later
on, we shall often use letters in bold type to denote such pairs,
e.g. $\mathbf{X}\coloneqq(X,X_{0})$. 

\begin{rem}
In order to easily distinguish between discrete and non-discrete metric
spaces, we adopt the following convention: discrete metric spaces
will be written in standard font, such as $X$, $Y$, etc., while
non-necessarily discrete metric spaces will be written in sans-serif
font, such as $\mathsf{X}$, $\mathsf{Y}$, etc.
\end{rem}

\begin{notation}
Let $(\mathsf{X},\mathsf{X}_{0})$ be a pair of proper metric spaces
with $\mathsf{X}_{0}$ closed in $\mathsf{X}$. For a $C^{*}$-algebra
$B$, define the $C^{*}$-algebra 
\[
C_{0}(\mathsf{X},\mathsf{X}_{0},B)\coloneqq\left\{ f\in C_{0}(\mathsf{X},B):f|_{\mathsf{X}_{0}}=0\right\} ,
\]
and the good endofunctor $\C_{\mathsf{X},\mathsf{X}_{0}}\coloneqq C_{0}(\mathsf{X},\mathsf{X}_{0},\mmph)$,
which is clearly tensor-type (cf.~Example~\ref{exa:tt}).
\end{notation}

\begin{lem}\label{lem:denseCXX0}
Let $(\mathsf{X},\mathsf{X}_{0})$ be as above.
Then the $*$-algebra $\bigcup_{r>0}\{f\in C_{0}(\mathsf{X},\mathsf{X}_{0}):f|_{\nbhd_{r}(\mathsf{X}_{0})}=0\}$
is dense in $C_{0}(\mathsf{X},\mathsf{X}_{0})$.
\end{lem}

\begin{proof}
Straightforward.
\end{proof}

\subsection{Scalable metric spaces}
\begin{defn}\label{def:my-scaleable}
A pair $(\mathsf{X},\mathsf{X}_{0})$ of
metric spaces is \emph{scalable} if there is a continuous function
\[
\sc\colon(\mathsf{X}\times\mathbb{R}_{+},\mathsf{X}_{0}\times\mathbb{R}_{+})\to(\mathsf{X},\mathsf{X}_{0})\colon(x,t)\mapsto\sc_{t}x
\]
referred to as the \emph{scaling}, and a family of (non-necessarily
continuous) functions $\left\{ \rho_{t}\colon\mathbb{R}_{+}\to\mathbb{R}_{+}\right\} _{t\geq0}$
such that:
\begin{enumerate}[label=\textup{\textbf{(S\arabic*)}}, ref=\textbf{(S\arabic*)}]
\item\label{cond:S:zero}$\sc_{0}=\id$;
\item\label{cond:S:proper} if $\mathsf{B}\subset\mathsf{X}$ is a bounded
subset, then $\bigcup_{t\in[0,T]}\sc^{-1}_{t}(\mathsf{B})$ is bounded
for all $T>0$;
\item\label{cond:S:leq}$\dist\left(\sc_{t}x,\sc_{t}y\right)\leq\rho_{t}(\dist(x,y))$
for all $x,y\in\mathsf{X}$;
\item\label{cond:S:monot}$r\mapsto\rho_{t}(r)$ is non-decreasing for
all $t\in\mathbb{R}_{+}$;
\item\label{cond:S:rho}$t\mapsto\rho_{t}(r)$ is bounded and vanishes
at infinity for all $r\in\mathbb{R}_{+}$.
\end{enumerate}
\end{defn}

\begin{notation}
For $f\in C_{0}(\mathsf{X},\mathsf{X}_{0})$, we use the following
notation: $\sc^{*}_{t}(f)\coloneqq f(\sc_{t}(\singlecdot))$.
\end{notation}

\begin{lem}\label{nbhd-X0-collapses}
 Let $\sc$ be a scaling function for a
pair $(\mathsf{X},\mathsf{X}_{0})$ of metric spaces. Then: 
\begin{enumerate}[label=\textup{(\roman*)}]
\item For every $R,r>0$, there is $N>0$ such that
\[
\dist(x,\mathsf{X}_{0})<R\qquad\Longrightarrow\qquad\sup_{t\geq N}\dist(\sc_{t}x,\mathsf{X}_{0})<r;
\]
\item For every $R>0$, there is $r>0$ such that
\[
\dist(x,\mathsf{X}_{0})<R\qquad\Longrightarrow\qquad\sup_{t\in\mathbb{R}_{+}}\dist(\sc_{t}x,\mathsf{X}_{0})<r.
\]
\end{enumerate}
\end{lem}

\begin{proof}
Let $R>0$ and let $x_{0}\in\mathsf{X}_{0}$ be such that $\dist(x,x_{0})<R$.
Using conditions~\ref{cond:S:leq},~\ref{cond:S:monot}, and~\ref{cond:S:rho}
together with the fact that $\{\sc_{t}x_{0}\}_{t\in\mathbb{R}_{+}}\subset\mathsf{X}_{0}$,
we obtain the inequalities
\begin{alignat*}{1}
 & \dist(\sc_{t}x,\sc_{t}x_{0})\leq\rho_{t}(\dist(x,x_{0}))\leq\rho_{t}(R)\xrightarrow[t\to\infty]{}0,\\
 & \dist(\sc_{t}x,\mathsf{X}_{0})\leq\dist(\sc_{t}x,\sc_{t}x_{0})\leq\rho_{t}(\dist(x,x_{0}))\leq\rho_{t}(R)\leq\sup_{t}\rho_{t}(R)<\infty,
\end{alignat*}
which imply respectively the first and the second statements.
\end{proof}

\begin{rem}
The notion of a scalable metric space first appears in~\cite{HigsonRoe1995CoarseBC}
in the context of the coarse Baum-Connes conjecture. It was defined
as follows: a metric space $\mathsf{X}$ is called \emph{scalable}
if there is a proper\footnote{A map is called \emph{proper} if the inverse image of every precompact
set is precompact.} continuous map
\[
h\colon\mathsf{X}\times[0,1]\to\mathsf{X}\colon(x,t)\mapsto h_{t}(x)
\]
 such that:
\begin{itemize}
\item $h_{0}=\id$;
\item $\left\{ h_{t}\right\} _{t\in[0,1]}$ is an equibornologous\footnote{A family $\{h_{t}\}$ of maps between metric spaces is called \emph{equibornologous}
if there is a function $\rho\colon\mathbb{R}_{+}\to\mathbb{R}_{+}$
such that $\dist(h_{t}(x),h_{t}(y))\leq\rho(\dist(x,y))$ for all
$t$.} family;
\item $\dist(h_{1}(x),h_{1}(y))<\frac{1}{2}\dist(x,y)$ for all $x,y\in\mathsf{X}$. \end{itemize}
It is easy to verify that  if $\mathsf{X}$ is scalable in the sense
of~\cite{HigsonRoe1995CoarseBC}, then the pair $(\mathsf{X},\varnothing)$
is scalable in the sense of Definition~\ref{def:my-scaleable}. To
see this, just define $\sc_{t}=h^{[t]}\circ h_{\{t\}}$, where $[t]$
and $\{t\}$ are respectively the integer and fractional parts of
$t$.

Hence, we immediately obtain examples of scalable metric spaces by
taking those from~\cite{HigsonRoe1995CoarseBC}, such as locally
finite trees and non-positively curved Riemannian manifolds. 
\end{rem}

\subsection{Asymptotic adjunction}
\begin{defn}
Let $\mathsf{Y}$ be a metric space, $X\subset\mathsf{Y}$ a discrete
subspace, and $\Delta>0$. We say that: \begin{itemize}
\item $X$ is a \emph{$\Delta$-net} in $\mathsf{Y}$ if $\dist(y,X)<\Delta$
for all $y\in\mathsf{Y}$; 
\item $X$ is \emph{$\Delta$-separated} if $\dist(x,x')\geq\Delta$ for
all $x,x'\in X$;
\end{itemize}
It follows from Zorn's lemma that a $\Delta$-separated $\Delta$-net
in $\mathsf{Y}$ exists for every $\Delta>0$. \end{defn}

\begin{defn}
A metric subspace $\mathsf{X}\subset\mathsf{Y}$ is \emph{coarsely
dense} in $\mathsf{Y}$ if there is $\Delta>0$ such that $\dist(y,\mathsf{X})<\Delta$
for all $y\in\mathsf{Y}$.
\end{defn}

\begin{defn}
A metric space $\mathsf{X}$ has \emph{bounded
coarse geometry} if it admits a coarsely dense discrete subspace of
bounded geometry.
\end{defn}

\begin{rem}
One can show that the definition of a metric space with bounded coarse
geometry given above is equivalent to~\cite[Definition~3.6]{HigsonRoe1995CoarseBC}.
\end{rem}

\begin{lem}\label{lem:DeltaBG}
Let $\mathsf{X}$ be a metric space with bounded
coarse geometry. Then for sufficiently large $\Delta$, every $\Delta$-separated
$\Delta$-net has bounded geometry.
\end{lem}

\begin{proof}
Let $r>0$, and let $Z\subset\mathsf{X}$ be an $r$-net with bounded
geometry. Let $\Delta>2r$, and let $Y$ be a $\Delta$-separated
$\Delta$-net in $\mathsf{X}$. Assume that $Y$ is not of bounded
geometry, i.e., there is $R>0$ and a sequence $\{y_{n}\}^{\infty}_{n=1}\subset Y\subset\mathsf{X}$
such that
\[
|Y\cap\ball_{R}(y_{n})|\geq n.
\]
For all $n\in\mathbb{N}$, choose $z_{n}\in Z$ such that $\dist(z_{n},y_{n})<r$.
The map $Y\cap\ball_{R}(y_{n})\to Z\cap\ball_{R+r}(y_{n})$ sending
$y$ to some $z\in\ball_{r}(y)$ is clearly injective, hence
\[
\abs{Z\cap B_{R+2r}(z_{n})}\geq|Z\cap\ball_{R+r}(y_{n})|\geq|Y\cap\ball_{R}(y_{n})|\geq n
\]
 which contradicts the fact that $Z$ has bounded geometry.
\end{proof}

\begin{defn}
Let $(\mathsf{X},\mathsf{X}_{0})$ be a pair of metric spaces, and
let $\Delta>0$. A \emph{$\Delta$-discretization} of $(\mathsf{X},\mathsf{X}_{0})$
is a pair $(X,X_{0})$ with:
\begin{itemize}
\item $X\subset\mathsf{X}$ a $\Delta$-separated $\Delta$-net; 
\item $X_{0}\coloneqq\{x\in X\mid\dist(x,\mathsf{X}_{0})<\Delta\}$.
\end{itemize}
\end{defn}

In the rest of this section, $(\mathsf{X},\mathsf{X}_{0})$ will stand
for a pair of locally compact metric spaces, with
$\mathsf{X}$ having bounded coarse geometry, and $\mathsf{X}_{0}$
closed in $\mathsf{X}$. We also denote by $(X,X_{0})$ a $\Delta$-discretization
of $(\mathsf{X},\mathsf{X}_{0})$ for some fixed $\Delta>0$, with
$X$ having bounded geometry ($(X,X_{0})$ exists by Lemma~\ref{lem:DeltaBG}).
It will be convenient sometimes to use the abbreviations $\mathsb X\coloneqq(\mathsf{X},\mathsf{X}_{0})$
and $\mathbf{X}\coloneqq(X,X_{0})$. Finally, we fix some square partition
of unity\footnote{We call $\{\alpha_{x}\}$ a \emph{square partition of unity} if
$\{\alpha^{2}_{x}\}$ is a partition of unity.} $\left\{ \alpha_{x}\right\} _{x\in X}\subset C_{0}\left(\mathsf{X}\right)$
subordinate to the locally finite cover $\{\ball_{\Delta}(x)\}_{x\in X}$
of $\mathsf{X}$ by $\Delta$-balls centered at the points of $X\subset\mathsf{X}$.

We remind the reader that we tacitly identify $\C_{\mathsf{X},\mathsf{X}_{0}}$
with $\mathfrak{O}_{C_{0}(\mathsf{X},\mathsf{X}_{0})}$ (cf.~(\ref{eq:convention})).
\begin{lem}\label{lem:extended-eta-well-defined}
There is a well-defined labeled
natural transformation
\begin{alignat}{4}
\overline{\eta} & \colon\Id\Rightarrow\mathfrak{N}^{u}_{X,X_{0}}\mathfrak{T}I\C_{\mathsf{X},\mathsf{X}_{0}},\quad & \overline{\eta}B\colon &  & b\mapsto & \, & \qmat_{x,y}\fun_{t}\fun_{s}\left(\chi_{x}\chi_{y}\sc^{*}_{ts}(\alpha_{x}\alpha_{y})\otimes b)\right),\label{eq:extented-eta-wd}\\
 &  &  &  &  &  & \text{where}~\chi_{x}=\begin{cases}
1, & \dist(x,X_{0})>\Delta;\\
0, & \text{otherwise}.
\end{cases}\nonumber 
\end{alignat}
\end{lem}

\begin{proof}
For an arbitrary $B\in\in\Cstar$, introduce the linear map
\[
\zeta\colon B\to\mathfrak{M}^{u}_{X}\mathfrak{T}I\C_{\mathsf{X},\mathsf{X}_{0}}B\colon b\mapsto\mat_{x,y}\fun_{t}\fun_{s}(\chi_{x}\chi_{y}\sc^{*}_{ts}(\alpha_{x}\alpha_{y})\otimes b),
\]
which is well-defined since $\chi_{x}\sc^{*}_{\tau}(\alpha_{x})\in C_{0}(\mathsf{X},\mathsf{X}_{0})$
for all $\tau\geq0$ and all $x\in X$. For $a,b\in B$, the discrepancy
\[
m\coloneqq\zeta(ab)-\zeta(a)\zeta(b)=\mat_{x,y}\fun_{t}\fun_{s}\left(\sum_{z\in X}(1-\chi_{z})\chi_{x}\chi_{y}\sc^{*}_{ts}(\alpha_{x}\alpha^{2}_{z}\alpha_{y})\otimes ab\right)
\]
has finite propagation. Moreover, $m_{x,y}\neq0$ only if there is
$z\in X$ such that the inequalities \begin{alignat*}{3}
\dist(z,X_{0})\leq\Delta, & \qquad & \dist(z,x)\leq2\Delta, & \qquad & \dist(z,y)\leq2\Delta
\end{alignat*}
hold, which implies that $m\in\mathfrak{M}^{u}_{X\supset X_{0}}\mathfrak{T}I\C_{\mathsf{X},\mathsf{X}_{0}}B$.
Using the obvious equality $\overline{\eta}B=q_{X,X_{0}}\mathfrak{T}I\C_{\mathsf{X},\mathsf{X}_{0}}B\circ\zeta$,
we conclude that $\overline{\eta}B$ is multiplicative. It is trivial
to check that $\overline{\eta}B$ is linear and preserves involution,
and that it is natural in $B$.

Denote by $\overline{\eta}_{0}$ the natural transformation
\[
\overline{\eta}_{0}\colon\Id\Rightarrow\mathfrak{M}^{u}_{X}\mathfrak{T}I\C_{\mathsf{X}},\qquad\overline{\eta}_{0}B\colon\,b\mapsto\mat_{x,y}\fun_{t}\fun_{s}(\sc^{*}_{ts}(\alpha_{x}\alpha_{y})\otimes b)
\]
and note that it fits into the commuting diagram
\[
\input{diagrams/d2463.tikz}
\]
where $\incl\colon\C_{X,X_{0}}\Rightarrow\C_{X}$ is the obvious componentwise
inclusion. The endofunctors $\mathfrak{M}^{u}_{X}$, $\mathfrak{T}$,
$\C_{X}$, and~$I$ are product-separating and monic-preserving,
and hence, $\mathfrak{T}I\C_{X}$ and $\mathfrak{M}^{u}_{X}\mathfrak{T}I\C_{X}$
are well-pointed by Lemma~\ref{lem:psmp-compos}. Note also that
$\overline{\eta}_{0}$, $\incl$, and $q_{X,X_{0}}$ are labeled natural
transformations, respectively by Proposition~\ref{prop:wp-lbl} and
Lemma~\ref{lem:quotient-criterion}. Thus, $\mathfrak{N}^{u}_{X,X_{0}}\mathfrak{T}I\incl$
and $q_{X,X_{0}}\mathfrak{T}I\C_{\mathsf{X},\mathsf{X}_{0}}$ are
also labeled. Using that $\mathfrak{N}^{u}_{X,X_{0}}$, $\mathfrak{T}$,
and $I$ are monic-preserving (see Lemma~\ref{lem:N-monpres}), we
conclude that $\mathfrak{N}^{u}_{X,X_{0}}\mathfrak{T}I\incl$ is componentwise
monic, and hence, by Lemma~\ref{lem:cw-epic-epic}, monic. It follows
from Lemma~\ref{lem:abg} that $\overline{\eta}$ is labeled.
\end{proof}

\begin{prop}
[{\cite[Corollary 3.4]{AkePedRoe73}}]\label{prop:Cbstr} Let $\mathsf{X}$
be a locally compact Hausdorff space, and let $C_{b,str}(\mathsf{X},\mathcal{M}(B))$
be the $C^{*}$-algebra of bounded strictly continuous functions on
$\mathsf{X}$ with values in $\mathcal{M}(B)$. Then
\[
\mathcal{M}(C_{0}(\mathsf{X},B))\cong C_{b,str}(\mathsf{X},\mathcal{M}(B)).
\]
\end{prop}

\begin{lem}\label{lem:incl}
The following maps are well-defined injective $*$-homomorphisms:
\begin{alignat*}{1}
\mathfrak{M}^{u}_{X}\mathfrak{T}\mathcal{M}(B)\to C_{b,str}([0,\infty),\mathcal{M}(\uK B))\colon & \mat_{x,y}(f_{x,y})\mapsto(t\mapsto\mat_{x,y}(f_{x,y}(t)));\\
\mathfrak{T}\uK B\to C_{b,str}([0,\infty),\mathcal{M}(\uK B))\colon & f\mapsto f.
\end{alignat*}
\end{lem}

\begin{proof}
Straightforward.
\end{proof}

\begin{lem}\label{lem:extented-epsilon-well-defined}
There is a well-defined
labeled natural transformation
\[
\overline{\varepsilon}\colon\C_{\mathsf{X},\mathsf{X}_{0}}\mathfrak{N}^{u}_{X,X_{0}}\mathfrak{T}\Rightarrow\mathfrak{A}\uK,\qquad\overline{\varepsilon}B\colon f\otimes\qmat_{x,y}\fun_{t}m_{x,y}(t)\mapsto\as_{t}\left((\diag_{x\in X}f(\sc_{t}x))\cdot\mat_{x,y}m_{x,y}(t)\right).
\]
\end{lem}

\begin{proof}
Define the $*$-homomorphism
\[
d\colon C_{0}(\mathsf{X})\to\mathfrak{T}\uK\mathbb{C}\colon f\mapsto\fun_{t}\diag_{x\in X}f(\sc_{t}(x)).
\]
Bearing in mind Proposition~\ref{prop:Cbstr} and Lemma~\ref{lem:incl},
we can assume that $\mathfrak{T}\uK\mathbb{C}$, $\mathfrak{T}\uK B$,
and $\mathfrak{M}^{u}_{X}\mathfrak{T}B$ are $C^{*}$-subalgebras
of $\mathcal{M}(\mathfrak{T}_{0}\uK B)$, which allows us to consider
the linear map
\[
C_{0}(\mathsf{X})\otimes_{\mathrm{alg}}\mathfrak{M}^{u}_{X}\mathfrak{T}B\to\mathcal{M}(\mathfrak{T}_{0}\uK B)\colon f\otimes m\mapsto d(f)\cdot m,
\]
where $\otimes_{\mathrm{alg}}$ stands for the algebraic tensor product.
It is straightforward to show that for all $f\in C_{0}(\mathsf{X})$
and $m\in\mathfrak{M}^{u}_{X}\mathfrak{T}B$, we have
\begin{alignat*}{2}
d(f)\cdot m\in\mathfrak{T}\uK B, & \qquad\qquad & [d(f),m]\in\mathfrak{T}_{0}\uK B.
\end{alignat*}
Thus, there is a well-defined $*$-homomorphism
\begin{alignat*}{2}
 & C_{0}(\mathsf{X})\otimes_{\mathrm{alg}}\mathfrak{M}^{u}_{X}\mathfrak{T}B\to\mathfrak{A}\uK B\colon f\otimes m &  & \mapsto(\as\uK B)(d(f)\cdot m)\\
 &  &  & =\as_{t}\mat_{x,y}f(\sc_{t}x)m_{x,y}(t)
\end{alignat*}
which by the universal property of the maximal tensor product can
be extended to $C_{0}(\mathsf{X})\otimes\mathfrak{M}^{u}_{X}\mathfrak{T}B$.
This, in turn, gives rise to the natural transformation
\[
\overline{\varepsilon}_{0}\colon\C_{\mathsf{X}}\mathfrak{M}^{u}_{X}\mathfrak{T}\Rightarrow\mathfrak{A}\uK,\qquad\overline{\varepsilon}_{0}B\colon f\otimes\mat_{x,y}\fun_{t}m_{x,y}(t)\mapsto\as_{t}\mat_{x,y}f(\sc_{t}x)m_{x,y}(t).
\]
That $\overline{\varepsilon}_{0}$ is labeled follows from the explicit
formulas for the labeling of $\C_{\mathsf{X},\mathsf{X}_{0}}\mathfrak{M}^{u}_{X}\mathfrak{T}$
and $\mathfrak{A}\uK$ (see Remark~\ref{rem:expl-lbl-N} and~Lemma~\ref{lem:fA-kappa}).

We claim that $\overline{\varepsilon}_{0}B(f\otimes m)=0$ whenever
$f\in C_{0}(\mathsf{X},\mathsf{X}_{0})$ and $m\in\mathfrak{M}^{u}_{X\supset X_{0}}\mathfrak{T}B$.
Indeed, using Lemma~\ref{lem:denseCXX0}, we can assume without loss
of generality that $\supp(f)\subset\mathsf{X}\setminus\nbhd_{r}(\mathsf{X}_{0})$
and $\supp(m)\subset\nbhd_{R}(X_{0})^{\times2}$ for some $r,R>0$.
It follows from the first statement of Lemma~\ref{nbhd-X0-collapses}
that $f(\sc_{t}x)m^{t}_{x,y}=0$ for sufficiently large $t\in\mathbb{R}_{+}$,
from which the claim follows.

Lemma~\ref{lem:26} combined with the claim yields the required natural
transformation $\overline{\varepsilon}$. 
\end{proof}

We are now ready to formulate the main result of this paper, which
we shall prove in the next section.
\begin{thm}\label{thm:adjunction}
Let $\mathsb X\coloneqq(\mathsf{X},\mathsf{X}_{0})$
be a scalable pair of locally compact metric spaces with $\mathsf{X}$
having bounded coarse geometry. Let $\mathbf{X}\coloneqq(X,X_{0})$
be a $\Delta$-discretization of $\mathsb X$ for some $\Delta>0$
with $X$ having bounded geometry. And let $\left\{ \alpha_{x}\right\} $
be a square partition of unity subordinate to the cover $\{\ball_{\Delta}(x)\}_{x\in X}$
of $\mathsf{X}$ by $\Delta$-balls. Then there is an asymptotic adjunction
$\C_{\mathsb X}\asadj\mathfrak{N}^{u}_{\mathbf{X}}$ witnessed by
the following unit and counit:
\begin{alignat*}{4}
\eta\colon & \Id\Rightarrow\mathfrak{N}^{u}_{\mathbf{X}}\C_{\mathsb X}, &  &  & \quad &  & \eta A\colon & a\mapsto\qmat_{x,y}\left(\chi_{x}\chi_{y}\alpha_{x}\alpha_{y}\otimes a\right),\\
\varepsilon\colon & \C_{\mathsb X}\mathfrak{N}^{u}_{\mathbf{X}}\Rightarrow\mathfrak{A}\uK, &  &  &  &  & \varepsilon A\colon & f\otimes\qmat_{x,y}m_{x,y}\mapsto\as_{t}\left((\diag_{x\in X}f(\sc_{t}x))\cdot\mat_{x,y}m_{x,y}\right),
\end{alignat*}
where
\[
\chi_{x}=\begin{cases}
1, & \dist(x,X_{0})>\Delta;\\
0, & \text{otherwise}.
\end{cases}
\]
\end{thm}

\begin{rem}
Applying Lemmas~\ref{lem:extended-eta-well-defined} and~\ref{lem:extented-epsilon-well-defined},
one can easily show that $\varepsilon$ and $\eta$ in Theorem~\ref{thm:adjunction}
are well-defined. That $\eta$ and $\varepsilon$ are labeled follows
from the factorizations \begin{alignat*}{2}
\varepsilon= & \overline{\varepsilon}\circ\mathfrak{N}^{u}_{\mathbf{X}}\C_{\mathsb X}\const, & \qquad\qquad & \eta=\mathfrak{N}^{u}_{\mathbf{X}}\ev_{0}\ev_{0}\C_{\mathsb X}\circ\overline{\eta}
\end{alignat*}
in which all the factors are labeled.
\end{rem}

\subsection{Proof of the main theorem}

In this subsection, $\mathsb X=(\mathsf{X},\mathsf{X}_{0})$, $\mathbf{X}=(X,X_{0})$,
$\Delta>0$, $\left\{ \alpha_{x}\right\} $, $\eta$, and $\varepsilon$
will be as in Theorem~\ref{thm:adjunction}.

\begin{lem}\label{lem:42}
Let $X$ be a discrete metric space of bounded geometry,
and let $\left\{ m_{x}\right\} _{x\in X}\subset\mathcal{M}(B)$ be
a family of multipliers satisfying the following conditions:
\begin{enumerate}[label=\textup{\textbf{(H\arabic*)}}, ref=\textbf{(H\arabic*)}]
\item\label{enu:m1}$\sum_{x\in X}m^{*}_{x}m_{x}=1$, with the sum converging
in the strict topology;
\item\label{enu:m2}there exists $N>0$ such that $m_{x}m^{*}_{y}=0$ whenever
$\dist(x,y)>N$.
\end{enumerate}
Then there is a well-defined $*$-homomorphism
\[
\Lambda\colon\mathcal{M}(B)\to\mathcal{M}\left(\uK_{X}B\right)\colon m\mapsto\mat_{x,y}(m_{x}mm^{*}_{y}).
\]
If additionally we have a $C^{*}$-subalgebra $A\subset B$ and a
dense $*$-subalgebra $A_{0}\subset A$ such that for all $a\in A_{0}$,
the elements $m_{x}a$ and $am_{x}$ are non-zero for only finitely
many $x\in X$, then $\Lambda$ restricts to
\[
\Lambda\colon A\to\uK_{X}B\colon a\mapsto\mat_{x,y}(m_{x}am^{*}_{y}).
\]
\end{lem}

\begin{proof}
We claim that $\norm{m_{x}}\leq1$ for all $x\in X$. Indeed, for
all $b\in B$ and $y\in X$ we have
\[
b^{*}m^{*}_{y}m_{y}b\leq\sum_{x\in X}b^{*}m^{*}_{x}m_{x}b=b^{*}b
\]
which by Lemma~\ref{lem:pos}\ref{enu:pos2} implies that $\norm{m_{y}b}^{2}\leq\norm b^{2}$
from which the claim follows.

Condition~\ref{enu:m2}, Lemma~\ref{lem:norm-fM-of-norm-its-elements},
and the claim imply that 
\begin{equation}
\mat_{x,y}(m_{x}m^{*}_{y})\in\mathcal{L}_{B}(l_{2}(X,B)).\label{eq:mm}
\end{equation}
Consider the following Hilbert $B$-module maps:
\begin{alignat*}{2}
 & \Xi & \colon & B\to l_{2}(X,B)\colon\,b\mapsto\left(m_{x}b\right)_{x\in X},\\
 & \Xi^{*} & \colon & l_{2}(X,B)\to B\colon(b_{x})_{x\in X}\mapsto\sum_{x\in X}m^{*}_{x}b_{x}.
\end{alignat*}
The map $\Xi$ is well-defined since the sum $\sum_{x\in X}b^{*}m^{*}_{x}m_{x}b$
is norm convergent by condition~\ref{enu:m1}. That $\Xi^{*}$ is
also well-defined follows from (\ref{eq:mm}) together with the equality
\[
\sum_{x,y\in X}b^{*}_{x}m_{x}m^{*}_{y}b_{y}=\left\langle b,(\mat_{x,y}m_{x}m^{*}_{y})b\right\rangle _{l_{2}(X,B)}
\]
for every $b\in l_{2}(X,B)$. It is clear that $\Xi^{*}$ is the adjoint
of $\Xi$, and that $\Xi^{*}\Xi=\id$; thus there is a $*$-homomorphism
\[
\Lambda\colon\mathcal{L}_{B}\left(B\right)\to\mathcal{L}_{B}\left(l_{2}(X)\otimes B\right)\colon m\mapsto\Xi m\Xi^{*}.
\]
Using the notation of Lemma~\ref{lem:mat-rep}, we get the equality
\begin{multline*}
\qquad\qquad\left(\Lambda(m)\right)_{x,y}=\strlim_{\lambda}\left(\left\langle e_{x}u_{\lambda},\Xi m\Xi^{*}(e_{y}u_{\lambda})\right\rangle \right)=\\
=\strlim_{\lambda}\left(\left\langle \Xi^{*}(e_{x}u_{\lambda}),m\Xi^{*}(e_{y}u_{\lambda})\right\rangle \right)=\strlim_{\lambda}\left(u_{\lambda}m_{x}mm^{*}_{y}u_{\lambda}\right)=m_{x}mm^{*}_{y}.\qquad\qquad
\end{multline*}
which proves the first statement of the lemma. The second statement
is obvious.
\end{proof}

\begin{lem}\label{lem:hh0-new}
 There is a well-defined labeled natural transformation
\[
\xi\colon\C_{\mathsf{X}}\Rightarrow\uK_{X}\C_{\mathsf{X}},\quad\xi B\colon f\otimes b\mapsto\mat_{x,y\in X}\left(\alpha_{x}\alpha_{y}f\otimes b\right)
\]
which is homotopic to $\iota_{00}\C_{\mathsf{X}}$.
\end{lem}

\begin{proof}
Enumerate the elements of $X\eqqcolon\{x_{i}\}^{\infty}_{i=0}$ and
denote $\alpha_{x_{i}}$ briefly by $\alpha_{i}$. For every $i,n\in\mathbb{Z}_{+}$,
define the function $\beta^{n}_{i}\in C_{0}(\mathsf{X})$ by the formula
\[
\beta^{n}_{i}\coloneqq\begin{cases}
\sqrt{\sum^{n}_{k=0}\alpha^{2}_{k}}, & i=0;\\
0, & 1\leq i\leq n;\\
\alpha_{i}, & i\geq n+1,
\end{cases}
\]
and for all $t\in[0,\infty]$, define \begin{alignat*}{3}
 &  & \beta^{t}_{i} & \coloneqq\sqrt{\left(1-\{t\}\right)(\beta^{[t]}_{i})^{2}+\{t\}(\beta^{[t]+1}_{i})^{2}}\in C_{0}(\mathsf{X}), & \quad & \textrm{if }t<\infty;\\
 &  & \beta^{t}_{i} & \coloneqq\delta_{i,0}\in C_{b}(\mathsf{X}), &  & \textrm{if }t=\infty,
\end{alignat*}
where $[t]$ and $\{t\}$ are respectively the integer and fractional
parts of $t$, and $\delta$ is the Kronecker delta.

The family $\left\{ \beta^{t}_{i}\right\} ^{\infty}_{i=0}\subset\mathcal{M}(C_{0}(\mathsf{X}))$
satisfies the conditions~\ref{enu:m1} and~\ref{enu:m2} of Lemma~\ref{lem:42}.
Furthermore, if $f\in C_{0}(\mathsf{X})$ is compactly supported,
then $\beta^{t}_{i}\beta^{t}_{j}f=0$ for all but finitely many $i,j\in\mathbb{Z}_{+}$.
Hence, by the second part of Lemma~\ref{lem:42}, we get the $*$-homomorphism
\[
\Phi_{t}\colon\C_{\mathsf{X}}B\to\uK_{X}\C_{\mathsf{X}}B\colon f\otimes b\mapsto\mat_{i,j}(\beta^{t}_{i}\beta^{t}_{j}f\otimes b)
\]
for every $t\in\mathbb{R}_{+}$. To prove the statement of the lemma
it suffices to show that
\begin{equation}
\Phi\colon\C_{\mathsf{X}}B\to\C_{[0,+\infty]}\uK_{X}\C_{\mathsf{X}}B\colon f\otimes b\mapsto\fun_{t}\Phi_{t}(f\otimes b)=\fun_{t}\mat_{i,j}\left(\beta^{t}_{i}\beta^{t}_{j}f\otimes b\right)\label{eq:Phi}
\end{equation}
is a well-defined $*$-homomorphism. To do that, note that for any
$b\in B$ and any compactly supported $f\in C_{0}(\mathsf{X})$, the
function 
\[
[0,\infty]\to\C_{\mathsf{X}}B\colon t\mapsto\beta^{t}_{i}\beta^{t}_{j}f\otimes b
\]
is continuous for all $i,j\in\mathbb{N}$, and is equal to zero for
all but finitely many $i,j\in\mathbb{N}$. Thus, $\fun_{t}\Phi_{t}(f\otimes b)\in\C_{[0,+\infty]}\uK_{X}\C_{\mathsf{X}}B$.
Using that $\Phi_{t}$ is a $*$-homomorphism for every $t\geq0$,
we conclude that $\Phi$ is a well-defined bounded $*$-homomorphism
on a dense $*$-subalgebra of $\C_{\mathsf{X}}B$, which by continuity
extends to a $*$-homomorphism (necessarily given by the formula~(\ref{eq:Phi})). 
\end{proof}

The proof of Theorem~\ref{thm:adjunction} is contained in the following
two propositions.
\begin{prop}\label{prop:lid}
The following diagram commutes in $\hGEFC$:
\[
\input{diagrams/d206.tikz}
\]
\end{prop}

\begin{proof}
Let us fix some $x_{0}\in X$, and consider the following natural
transformations\footnote{That $\varphi_{j}B$ are indeed $*$-homomorphisms will be clear from
the proof.}:
\begin{align*}
 &  & \varphi_{j}\colon & \C_{\mathsf{X},\mathsf{X}_{0}}\Rightarrow\mathfrak{A}\uK_{X}\C_{\mathsf{X},\mathsf{X}_{0}},\quad j=0,1,2,3,4;\\
 &  & \varphi_{0}B\colon & f\otimes b\mapsto\as_{t}\{\diag_{x}f\left(\sc_{t}x\right)\cdot\mat_{x,y}(\chi_{x}\chi_{y}\alpha_{x}\alpha_{y}\otimes b)\};\mydisplaybreak\\
 &  & \varphi_{1}B\colon & f\otimes b\mapsto\as_{t}\{\diag_{x}f\left(\sc_{t}x\right)\cdot\mat_{x,y}(\chi_{x}\chi_{y}\sc^{*}_{t}(\alpha_{x}\alpha_{y})\otimes b)\};\mydisplaybreak\\
 &  & \varphi_{2}B\colon & f\otimes b\mapsto\as_{t}\mat_{x,y}(\sc^{*}_{t}(\alpha_{x}\alpha_{y})(\sc^{2}_{t})^{*}(f)\otimes b);\mydisplaybreak\\
 &  & \varphi_{3}B\colon & f\otimes b\mapsto\as_{t}\mat_{x,y}(\alpha_{x}\alpha_{y}f\otimes b);\\
 &  & \varphi_{4}B\colon & f\otimes b\mapsto\as_{t}\mat_{x,y}(\delta_{x,x_{0}}\delta_{y,x_{0}}f\otimes b),
\end{align*}
where $\sc^{2}_{t}$ stands for $\sc_{t}\circ\sc_{t}$. Note that
$\varphi_{0}=\varepsilon\C_{\mathsf{X},\mathsf{X}_{0}}\circ\C_{\mathsf{X},\mathsf{X}_{0}}\eta$
and $\varphi_{4}=\alpha\iota_{00}$; thus, it suffices to show that
$\varphi_{0}\simeq\varphi_{4}$.

We use Lemmas~\ref{lem:extended-eta-well-defined} and~\ref{lem:extented-epsilon-well-defined}
to define the labeled natural transformation
\begin{alignat*}{1}
\varphi_{01}\colon & \C_{\mathsf{X},\mathsf{X}_{0}}\xRightarrow{\C_{\mathsf{X},\mathsf{X}_{0}}\overline{\eta}}\C_{\mathsf{X},\mathsf{X}_{0}}\mathfrak{N}^{u}_{X,X_{0}}\mathfrak{T}I\C_{\mathsf{X},\mathsf{X}_{0}}\xRightarrow{\overline{\varepsilon}I\C_{\mathsf{X},\mathsf{X}_{0}}}\mathfrak{A}\uK_{X}I\C_{\mathsf{X},\mathsf{X}_{0}},\\
\varphi_{01}B\colon & f\otimes b\mapsto\as_{t}\left(\diag_{x}f\left(\sc_{t}x\right)\cdot\mat_{x,y}\fun_{s}\left(\chi_{x}\chi_{y}\sc^{*}_{ts}(\alpha_{x}\alpha_{y})\otimes b\right)\right),
\end{alignat*}
which gives rise to a homotopy connecting $\varphi_{0}$ with $\varphi_{1}$
since $I$ and $\C_{\mathsf{X},\mathsf{X}_{0}}$ commute by Remark~\ref{rem:tt-commute}.

Let us show that $\varphi_{2}\simeq\varphi_{3}$. Fix $t\in\mathbb{R}_{+}$
and note that the family \[
\left\{ \fun_{s}\sc^{*}_{st}(\alpha_{x})\right\} _{x\in X}\subset I\C_{\mathsf{X},\mathsf{X}_{0}}\mathbb{C}\subset\mathcal{M}(I\C_{\mathsf{X},\mathsf{X}_{0}}B)
\]
satisfies the second part of Lemma~\ref{lem:42} (with $A\coloneqq I\C_{\mathsf{X},\mathsf{X}_{0}}B$
and $A_{0}\subset A$ the dense $*$-subalgebra of compactly supported
functions). Hence, the composite
\begin{alignat*}{1}
\Phi^{t}_{B}\colon & \C_{\mathsf{X},\mathsf{X}_{0}}B\to I\C_{\mathsf{X},\mathsf{X}_{0}}B\to\uK_{X}I\C_{\mathsf{X},\mathsf{X}_{0}}B\colon\\
 & f\otimes b\mapsto\fun_{s}((\sc^{2}_{st})^{*}(f)\otimes b)\mapsto\mat_{x,y}\fun_{s}\left(\sc^{*}_{st}(\alpha_{x}\alpha_{y})(\sc^{2}_{st})^{*}(f)\otimes b\right)
\end{alignat*}
is a $*$-homomorphism. It is clear that $\Phi^{t}_{B}(f\otimes b)$
is continuous in $t$ and functorial in $B$. Thus, we get the natural
transformation
\[
\Psi\colon\C_{\mathsf{X},\mathsf{X}_{0}}\Rightarrow\mathfrak{T}\uK_{X}I\C_{\mathsf{X},\mathsf{X}_{0}},\qquad\Psi B\colon f\otimes b\mapsto\fun_{t}\Phi^{t}_{B}(f\otimes b).
\]
The composite
\begin{alignat*}{1}
 & \C_{\mathsf{X},\mathsf{X}_{0}}\xRightarrow{\Psi}\mathfrak{T}\uK_{X}I\C_{\mathsf{X},\mathsf{X}_{0}}\xRightarrow{\as\uK_{X}I\C_{\mathsf{X},\mathsf{X}_{0}}}\mathfrak{A}\uK_{X}I\C_{\mathsf{X},\mathsf{X}_{0}},\\
 & f\otimes b\mapsto\as_{t}\Phi^{t}_{B}(f\otimes b)=\as_{t}\mat_{x,y}\fun_{s}\left(\sc^{*}_{st}(\alpha_{x}\alpha_{y})(\sc^{2}_{st})^{*}(f)\otimes b\right)
\end{alignat*}
induces a homotopy connecting $\varphi_{2}$ with $\varphi_{3}$.

The relation $\varphi_{3}\simeq\varphi_{4}$ follows from Lemma~\ref{lem:hh0-new}.

To finish the proof, we need only show that $\varphi_{1}=\varphi_{2}$.
Let us fix some $\varepsilon>0$ and $f\in C_{0}(\mathsf{X},\mathsf{X}_{0})$.
It follows from conditions~\ref{cond:S:leq} and~\ref{cond:S:rho}
of Definition~\ref{def:my-scaleable}, and uniform continuity of
$f$, that there is $N_{1}>0$ such that for $u,v\in\mathsf{X}$ and
$t>N_{1}$, we have
\begin{equation}
\dist(u,v)\leq\Delta\qquad\Longrightarrow\qquad|f(\sc_{t}u)-f(\sc_{t}v)|\leq\varepsilon.\label{eq:impl1}
\end{equation}
Moreover, Lemma~\ref{nbhd-X0-collapses}(i) and Lemma~\ref{lem:denseCXX0}
imply that there is $N_{2}>0$ such that for $t>N_{2}$, we have
\begin{equation}
\dist(y,\mathsf{X}_{0})\leq4\Delta\qquad\Longrightarrow\qquad|f(\sc_{t}y)|\leq\varepsilon/2.\label{eq:impl2}
\end{equation}

We claim that for all $x,y\in X$, $z\in\mathsf{X}$,
and $t>\max(N_{1},N_{2})$, the following inequality holds:
\begin{equation}
\left(\alpha_{x}\alpha_{y}\right)(\sc_{t}z)\left(f(\sc^{2}_{t}z)-\chi_{x}\chi_{y}f(\sc_{t}y)\right)\leq\varepsilon.\label{eq:diff1}
\end{equation}
To see this, first note that we can assume that
\begin{eqnarray}
\dist(\sc_{t}z,y) & \leq & \Delta,\label{eq:diff2}\\
\dist(x,y) & \leq & 2\Delta,\label{eq:diff3}
\end{eqnarray}
since otherwise the left hand side of~(\ref{eq:diff1}) is zero by
the definition of $\{\alpha_{x}\}$. Let us consider two cases.
\begin{itemize}
\item Case 1: 
\begin{eqnarray}
\dist(y,\mathsf{X}_{0}) & \leq & 3\Delta.\label{eq:pp1}
\end{eqnarray}
We have the following inequalities:
\begin{eqnarray}
\dist(\sc_{t}z,\mathsf{X}_{0}) & \leq & 4\Delta,\label{eq:pp2}\\
|f(\sc_{t}y)| & \leq & \varepsilon/2,\label{eq:pp3}\\
|f(\sc^{2}_{t}z)| & \leq & \varepsilon/2,\label{eq:pp4}
\end{eqnarray}
where (\ref{eq:pp2}) follows from~(\ref{eq:diff2}) and~(\ref{eq:pp1});
(\ref{eq:pp3}) follows from~(\ref{eq:impl2}) and~(\ref{eq:pp1});
(\ref{eq:pp4}) follows from~(\ref{eq:impl2}) and~(\ref{eq:pp2}).
Thus,~(\ref{eq:pp3}) and~(\ref{eq:pp4}) imply (\ref{eq:diff1}).
\item Case 2: 
\begin{equation}
\dist(y,\mathsf{X}_{0})>3\Delta.\label{eq:pp5}
\end{equation}
We deduce from~(\ref{eq:diff3}) and~(\ref{eq:pp5}) that $\dist(x,\mathsf{X}_{0})>\Delta$,
and therefore, $\chi_{x}=\chi_{y}=1$. It follows from~(\ref{eq:diff2})
and~(\ref{eq:impl1}) that $f(\sc^{2}_{t}z)-f(\sc_{t}y)\leq\varepsilon$,
which implies~(\ref{eq:diff1}). 
\end{itemize}
The claim is proved.

It follows from the claim that for all $b\in B$ and $f\in C_{0}(\mathsf{X},\mathsf{X}_{0})$,
all the entries of 
\[
\mat_{x,y}\{\left(\alpha_{x}\alpha_{y}\right)(\sc_{t}\cdot)\left(f(\sc^{2}_{t}\cdot)-\chi_{x}\chi_{y}f(\sc_{t}y)\right)\otimes b\}
\]
uniformly vanish as $t\to\infty$. By Lemma~\ref{lem:norm-fM-of-norm-its-elements},
the whole matrix tends to zero in norm, which implies the required
equality $\varphi_{1}=\varphi_{2}$.
\end{proof}

\begin{prop}\label{thm:rid}
The following diagram commutes in $\hGEFC$:
\[
\input{diagrams/d207.tikz}
\]
\end{prop}

\begin{proof}
Fix some $x_{0}\in X$ and consider the following natural transformations:
\begin{align*}
\varphi_{j}\colon & \mathfrak{N}^{u}_{X,X_{0}}\Rightarrow\mathfrak{N}^{u}_{X,X_{0}}\mathfrak{A}\uK_{X},\quad j=0,1,2,3,4;\\
\varphi_{0}B\colon & \qmat_{x,y}m_{x,y}\mapsto\qmat_{x,y}\as_{t}\mat_{x',y'}\left\{ (\alpha_{x}\alpha_{y})(\sc_{t}x')m_{x',y'}\right\} ;\\
\varphi_{1}B\colon & \qmat_{x,y}m_{x,y}\mapsto\qmat_{x,y}\as_{t}\mat_{x',y'}\left\{ \alpha_{x}\left(\sc_{t}x'\right)m_{x',y'}\alpha_{y}\left(\sc_{t}y'\right)\right\} ;\mydisplaybreak\\
\varphi_{2}B\colon & \qmat_{x,y}m_{x,y}\mapsto\qmat_{x,y}\as_{t}\mat_{x',y'}\left\{ \alpha_{x}\left(x'\right)m_{x',y'}\alpha_{y}\left(y'\right)\right\} ;\mydisplaybreak\\
\varphi_{3}B\colon & \qmat_{x,y}m_{x,y}\mapsto\qmat_{x,y}\as_{t}\mat_{x',y'}\left\{ m_{x,y}\delta_{x',x}\delta_{y',y}\right\} ;\\
\varphi_{4}B\colon & \qmat_{x,y}m_{x,y}\mapsto\qmat_{x,y}\as_{t}\mat_{x',y'}\left\{ m_{x,y}\delta_{x',x_{0}}\delta_{y',x_{0}}\right\} .
\end{align*}
Note that $\varphi_{0}=\mathfrak{N}^{u}_{X,X_{0}}\varepsilon\circ\eta\mathfrak{N}^{u}_{X,X_{0}}$
and $\varphi_{4}=\mathfrak{N}^{u}_{X,X_{0}}\alpha\iota_{00}$; and
hence, it suffices to show that $\varphi_{0}\simeq\varphi_{4}$.

The equality $\varphi_{0}=\varphi_{1}$ holds since 
\[
[\diag_{z\in X}\alpha_{y}\left(\sc_{t}z\right),m]\xrightarrow[t\to\infty]{}0
\]
 for every $m\in\mathfrak{M}^{u}_{X}B$ and $y\in X$.

The equality $\varphi_{2}=\varphi_{3}$ follows from the definition
of $\{\alpha_{x}\}$ combined with the fact that $X$ is $\Delta$-separated.

Let us show that $\varphi_{3}\simeq\varphi_{4}$. Set $U_{s}=\diag_{x\in X}u_{x}(s)$,
where $\{s\mapsto u_{x}(s)\}_{x\in X}$ is an equicontinuous family
of unitary operators on $l_{2}(X)$ such that $u_{x}(0)=\id$ and
$u_{x}(1)$ is the matrix which swaps the $x$th and $x_{0}$th basis
vectors. It is clear that $s\mapsto U_{s}$ is a continuous path of
unitaries in $\mathcal{M}(\mathfrak{M}^{u}_{X}\uK_{X}\mathbb{C})$;
thus, we can define the natural transformation
\[
\omega\colon\mathfrak{M}^{u}_{X}\uK_{X}\Rightarrow I\mathfrak{M}^{u}_{X}\uK_{X},\quad\omega B\colon m\mapsto\fun_{s}(U_{s}mU^{*}_{s})
\]
which is labeled by Proposition~\ref{prop:wp-lbl}. It is clear that
the components of $\omega$ preserve support of $X\textrm{-by-}X$-matrices;
hence, by Lemma~\ref{lem:26}, $\omega$ gives rise to the labeled
natural transformation 
\[
\overline{\omega}\colon\mathfrak{N}^{u}_{X,X_{0}}\uK_{X}\xRightarrow{\overline{\omega}}I\mathfrak{N}^{u}_{X,X_{0}}\uK_{X}
\]
which in turn, by Lemma~\ref{lem:stronger-homot}, induces a homotopy
connecting $\varphi_{3}$ with $\varphi_{4}$. 

We now only need to check that $\varphi_{1}\simeq\varphi_{2}$. To
this end, for every $\tau\in\mathbb{R}_{+}$ and $x\in X$, set 
\[
A^{\tau}_{x}\coloneqq\diag_{x'}\alpha_{x}\left(\sc_{\tau}x'\right)\in\uK_{X}\mathbb{C}.
\]

\textbf{Claim 1.} The map $\tau\mapsto A^{\tau}_{x}$ is continuous.
Indeed, this fact follows from axiom~\ref{cond:S:proper} and continuity
of the scaling function $\sc$.

Let us fix $t\in\mathbb{R}_{+}$, and introduce, using Claim~1, the
following family of multipliers:
\begin{equation}
\left\{ \fun_{s}A^{ts}_{x}\right\} _{x\in X}\subset I\uK_{X}\mathbb{C}\subset C_{b,str}([0,1],\mathcal{M}(\uK_{X}B))\cong\mathcal{M}(I\uK_{X}B).\label{eq:familyA}
\end{equation}

\textbf{Claim~2.} Family~(\ref{eq:familyA}) satisfies the conditions
of the first part of Lemma~\ref{lem:42}. Indeed, condition~\ref{enu:m2} is obvious, and we only need to check condition~\ref{enu:m1}.
To this end, note that the linear span of elements of the form $f\otimes\epsilon_{x,y}\otimes b$
is dense in $I\uK_{X}B$. Hence, it suffices to show that for every
$x'\in X$, we can choose some finite subset $\mathcal{F}\subset X$
such that $\sum_{x\in\mathcal{F}}\alpha^{2}_{x}(\sc_{st}x')=1$ for
all $s\in[0,1]$. But this fact is obvious since $\{\sc_{st}x'\mid s\in[0,1]\}$
is compact, and the claim is proved. 

By Lemma~\ref{lem:42}, the multipliers~(\ref{eq:familyA}) give
rise to the second $*$-homomorphism in the following composite (where
$t\in\mathbb{R}_{+}$ is fixed):
\begin{align*}
\Psi^{t}_{B}\colon & \mathcal{M}(\uK_{X}B)\to\mathcal{M}(I\uK_{X}B)\to\mathcal{M}(\uK_{X}I\uK_{X}B)\colon\\
 & m\mapsto(s\mapsto m)\mapsto\mat_{x,y}\fun_{s}(A^{st}_{x}mA^{st}_{y})\\
 & \phantom{m\mapsto(s\mapsto m).}=\mat_{x,y}\fun_{s}\mat_{x',y'}(\alpha_{x}(\sc_{ts}(x'))\alpha_{y}(\sc_{ts}(y'))m_{x',y'}).
\end{align*}

\textbf{Claim~3.} If $m\in\mathcal{M}(\uK_{X}B)$ is such that $\prp(m)<\infty$,
then $\sup_{t\geq0}\prp(\Psi^{t}_{B}(m))<\infty$\footnote{Here we regard $m$ and $\Psi^{t}_{B}(m)$ as an $X\textrm{-by-}X$-matrices;
cf. Lemma~\ref{lem:mat-rep}.}. Indeed, if $(x,y)\in\supp(\Psi^{t}_{B}(m))$ for some $t\in\mathbb{R}_{+}$,
then there exist $x',y'\in X$ and $\tau\in[0,t]$ such that 
\begin{equation}
\dist(\sc_{\tau}x',x)\leq\Delta,\qquad\qquad\dist(\sc_{\tau}y',y)\leq\Delta,\qquad\qquad\dist(x',y')\leq\prp(m).\label{eq:claim2}
\end{equation}
Using~(\ref{eq:claim2}) and axioms~\ref{cond:S:leq},~\ref{cond:S:monot},
and~\ref{cond:S:rho} of Definition~\ref{def:my-scaleable}, we
obtain the inequality
\[
\dist(x,y)-2\Delta\leq\dist(\sc_{\tau}x',\sc_{\tau}y')\leq\rho_{\tau}(\dist(x',y'))\leq\rho_{\tau}(\prp(m))\leq\sup_{\tau\geq0}\rho_{\tau}(\prp(m))<\infty
\]
which proves the claim.

It is clear from the definition of $A^{\tau}_{x}$ that for all $m\in\mathfrak{M}^{u}_{X}B\subset\mathcal{M}(\uK_{X}B)$,
all $x,y\in X$, and all $\tau\geq0$, we have $A^{\tau}_{x}mA^{\tau}_{y}\in\uK_{X}B$.
 This fact together with Claims~1 and~3 implies that there is a
well-defined $*$-homomorphism
\begin{alignat*}{3}
\Psi_{B}\colon\mathfrak{M}^{u}_{X}B & \to\mathfrak{M}^{u}_{X}\mathfrak{T}I\uK_{X}B\colon m &  & \mapsto &  & \mat_{x,y}\fun_{t}\fun_{s}(A^{ts}_{x}mA^{ts}_{y})\\
 &  &  & = &  & \mat_{x,y}\fun_{t}\fun_{s}\mat_{x',y'}(\alpha_{x}(\sc_{ts}x')\alpha_{y}(\sc_{ts}y')m_{x',y'}).
\end{alignat*}
This construction is obviously natural in $B$, and we obtain the
natural transformation
\[
\Psi\colon\mathfrak{M}^{u}_{X}\Rightarrow\mathfrak{M}^{u}_{X}\mathfrak{T}I\uK_{X}
\]
which is labeled by Proposition~\ref{prop:wp-lbl}.

\textbf{Claim~4.} There is an inclusion $\Psi_{B}\left(\mathfrak{M}^{u}_{X\supset X_{0}}B\right)\subset\mathfrak{M}^{u}_{X\supset X_{0}}\mathfrak{T}I\uK_{X}B$.
Indeed, suppose that $\supp(m)\subset\nbhd_{r}(X_{0})^{\times2}$
for some $r>0$. If $\Psi_{B}(m)_{x,y}\neq0$, then there exist $x'\in X$
and $\tau\geq0$ such that 
\begin{alignat}{1}
\dist(x',\mathsf{X}_{0}) & \leq r,\label{eq:c1}\\
\dist(\sc_{\tau}x',x) & \leq\Delta.\label{eq:c2}
\end{alignat}
Inequality~(\ref{eq:c1}) and Lemma~\ref{nbhd-X0-collapses}(ii)
imply that 
\begin{equation}
\dist(\sc_{\tau}x',\mathsf{X}_{0})\leq R\label{eq:c4}
\end{equation}
for some $R>0$. Combining~(\ref{eq:c2}) and~(\ref{eq:c4}) gives
$\dist(x,\mathsf{X}_{0})\leq R+\Delta$. The same inequality holds
for $y$; and the claim is proved.

Claim~4 together with Lemma~\ref{lem:26} gives us the labeled natural
transformation
\[
\overline{\Psi}\colon\mathfrak{N}^{u}_{X,X_{0}}\Rightarrow\mathfrak{N}^{u}_{X,X_{0}}\mathfrak{T}I\uK_{X},\qquad\overline{\Psi}B\colon\qmat_{x,y}(m_{x,y})\mapsto\qmat_{x,y}\fun_{t}\fun_{s}(A^{ts}_{x}mA^{ts}_{y}).
\]
Finally, the composite
\[
\varphi_{12}\colon\mathfrak{N}^{u}_{X,X_{0}}\xRightarrow{\overline{\Psi}}\mathfrak{N}^{u}_{X,X_{0}}\mathfrak{T}I\uK_{X}\xRightarrow{\mathfrak{N}^{u}_{X,X_{0}}\as I\uK_{X}}\mathfrak{N}^{u}_{X,X_{0}}\mathfrak{A}I\uK_{X}
\]
determines a homotopy connecting $\varphi_{1}$ with $\varphi_{2}$.

\end{proof}

\subsection{\label{subsec:appl}The isomorphism theorem and applications}

Now we can establish the following important result.
\begin{thm}\label{thm:main}
Let $\mathsb X$ and $\mathbf{X}$ be as in Theorem~\ref{thm:adjunction},
and let $A$ and $B$ be $C^{*}$-algebras. Then there is an isomorphism
of commutative monoids
\[
\colim_{k}\left[\C_{\mathsb X}A,\mathfrak{A}^{k}\uK,B\right]\cong\colim_{k}\left[A,\mathfrak{N}^{u}_{\mathbf{X}}\mathfrak{A}^{k}\uK,B\right]
\]
that is natural in $A$ and $B$. If in addition $A$ is separable,
then
\[
\left[\C_{\mathsb X}A,\mathfrak{A}\uK,B\right]\cong\left[A,\mathfrak{N}^{u}_{\mathbf{X}}\mathfrak{A}\uK,B\right].
\]
\end{thm}

\begin{proof}
The first statement follows from Theorems~\ref{thm:adjunction} and~\ref{thm:adj2brackets}.
The second statement follows from~\cite[Theorem~2.16]{GHT} and Theorem~\ref{thm:colim-is-redundant}
proved in the next section.
\end{proof}

Let us briefly discuss some applications of this result.

\subsubsection*{Unsuspended picture of $E$-theory}

While suspensions are an intrinsic part of the definition $E_{0}(A,B)\coloneqq\bbrackets{SA,S\mathbb{K}B}$,
several attempts have been made to find an alternative model for $E$-theory
groups with elements represented by maps from the $C^{*}$-algebra
$A$ itself rather than its suspension. For example, in~\cite{manuilov_KK-like_picture},
V.~Manuilov described $E_{0}(A,B)$ in terms of pairs of non-multiplicative
maps from $A$ to $C_{b}(\mathbb{R}_{+},\mathbb{K}B)$. Another important
work by M.~Dadarlat and T.~Loring~\cite{dadarlat1994} establishes
that $E_{0}(A,B)\cong\bbrackets{A,\mathbb{K}B}$ whenever $\bbrackets{A,\mathbb{K}B}$
is a group (rather than merely a monoid), which leads to the following
isomorphisms:
\begin{equation}
E_{0}(A,B)\cong\bbrackets{S^{2}A,\mathbb{K}B},\qquad\qquad E_{1}(A,B)\cong\bbrackets{SA,\mathbb{K}B}.\label{eq:e0e1}
\end{equation}
Applying Theorem~\ref{thm:main} to the right-hand sides of~(\ref{eq:e0e1})
(with $\mathsb X\coloneqq(\mathbb{R}^{2},\varnothing)$ and $\mathsb X\coloneqq(\mathbb{R},\varnothing)$,
respectively), we obtain an unsuspended description of $E$-theory
in terms of Roe functors:
\[
E_{0}(A,B)\cong\bbrackets{A,\mathfrak{M}^{u}_{\mathbb{Z}^{2}},B},\qquad\qquad E_{1}(A,B)\cong\bbrackets{A,\mathfrak{M}^{u}_{\mathbb{Z}},B}.
\]

\subsubsection*{Extensions}

Recall that for $C^{*}$-algebras $A$ and $B$, an \emph{extension}
is a $*$-homomorphism $A\to Q(\mathbb{K}B)$, where $Q$ indicates
the corona algebra
\[
Q(\mathbb{K}B)\coloneqq\frac{\mathcal{M}(\mathbb{K}\otimes B)}{\mathbb{K}\otimes B}.
\]
Using the obvious inclusion $\mathfrak{N}^{u}_{\mathbb{Z}_{+},\{0\}}\mathfrak{A}B\subset Q(\mathbb{K}\mathfrak{A}B)$,
we can regard the monoid $\bbrackets{A,\mathfrak{N}^{u}_{\mathbb{Z}_{+},\{0\}},B}$
as ``homotopy classes of extensions with asymptotic coefficients''.
Theorem~\ref{thm:main} yields the isomorphism
\[
E_{1}(A,B)\cong\bbrackets{SA,\mathbb{K}B}\cong\bbrackets{\C_{\mathbb{R}_{+},\{0\}}A,\Id,\mathbb{K}B}\cong\bbrackets{A,\mathfrak{N}^{u}_{\mathbb{Z}_{+},\{0\}},B},
\]
which can be viewed as an $E$-theoretic analog of the famous isomorphism
$KK_{1}(A,B)\cong\Ext^{-1}(A,B)$ between the Kasparov $KK_{1}$-theory
and the group of invertible extensions~\cite{Jensen-Thomsen}. We
refer the reader to~\cite{makeev_cyc_ext} for further details on
this correspondence (that work uses a slightly simpler categorical
framework, but the proof carries over with minimal modifications).

\subsubsection*{$K$-homology}

Theorem~\ref{thm:main} also allows us to express $K$-homology in
terms of relative Roe algebras. For a compact nonempty subset $\mathsf{X}$
of the unit sphere of a Euclidean space, define its
metric cone
\[
\mathcal{O}\mathsf{X}\coloneqq\{tx\mid x\in\mathsf{X},t\geq0\},
\]
and denote by $(\mathcal{O}\mathsf{X})^{\discr}\subset\mathcal{O}\mathsf{X}$
a coarsely dense subspace of bounded geometry. Using that the pair
$(\mathcal{O}\mathsf{X},\{0\})$ is scalable, we obtain the isomorphism
\begin{equation}
K_{1}(\mathsf{X})\cong E_{1}(C_{0}(\mathsf{X}),\mathbb{C})\cong\bbrackets{SC_{0}(\mathsf{X}),\uK}\cong\bbrackets{C_{0}(\mathcal{O}\mathsf{X},\{0\}),\Id,\uK}\cong\bbrackets{\mathbb{C},\mathfrak{N}^{u}_{(\mathcal{O}\mathsf{X})^{\discr},\{0\}},\uK}.\label{eq:k1-1}
\end{equation}
In fact, following the ideas from~\cite{HigsonRoe1995CoarseBC},
we can extend this construction to arbitrary finite-dimensional compact metrizable spaces. To do this, observe that every compact
metrizable space $\mathsf{X}$ can be embedded into the unit sphere
of a Euclidean space. That the right-hand side of~(\ref{eq:k1-1})
is well defined follows now from the two key facts:
\begin{itemize}
\item The coarse homotopy class~\cite{mitchener-norouzizadeh-2020coarse}
of the corresponding metric cone $\mathcal{O}\mathsf{X}$ is independent
of the choice of embedding; 
\item Roe functors are coarse homotopy invariant, meaning that if $(X,X_{0})$
and $(Y,Y_{0})$ are coarsely homotopy equivalent spaces of bounded
geometry, then $\mathfrak{N}^{u}_{X,X_{0}}\mathbb{K}$ and $\mathfrak{N}^{u}_{Y,Y_{0}}\mathbb{K}$
are homotopy equivalent.
\end{itemize}
The first fact is straightforward. The proof of the second (within
a slightly simpler categorical framework, and only for the case $X_{0}=\varnothing$
and $Y_{0}=\varnothing$) is given in~\cite{makeev_coarse_homotopies}.

While, as we have just observed, the existence of the isomorphism~(\ref{eq:k1-1})
is a direct consequence of Theorem~\ref{thm:main}, its naturality
in~$\mathsf{X}$ is a more subtle question which we leave for future
study.

%% file: diagrams/d2463.tikz
\begin{tikzcd}[ampersand replacement=\&]
{\Id} \&  \& {\mathfrak{N}_{X,X_{0}}^{u}\mathfrak{T}I\C_{X,X_{0}}} \\
{\mathfrak{M}_{X}^{u}\mathfrak{T}I\C_{X}} \&  \& {\mathfrak{N}_{X,X_{0}}^{u}\mathfrak{T}I\C_{X}}
\arrow["\overline{\eta}", Rightarrow, from=1-1, to=1-3]
\arrow["\overline{\eta}_{0}"', Rightarrow, from=1-1, to=2-1]
\arrow["\mathfrak{N}_{X,X_{0}}^{u}\mathfrak{T}I\incl", Rightarrow, from=1-3, to=2-3]
\arrow["q_{X,X_{0}}\mathfrak{T}I\C_{X}", Rightarrow, from=2-1, to=2-3]
\end{tikzcd}

%% file: diagrams/d206.tikz
\begin{tikzcd}[column sep=huge, ampersand replacement=\&]  
{\C_{\mathsf{X},\mathsf{X}_{0}}} \& {\C_{\mathsf{X},\mathsf{X}_{0}}\mathfrak{N}_{X,X_{0}}^{u}\C_{\mathsf{X},\mathsf{X}_{0}}} \\
 \& {\mathfrak{A}\uK\C_{\mathsf{X},\mathsf{X}_{0}}.}
\arrow["\C_{\mathsf{X},\mathsf{X}_{0}}\eta", Rightarrow, from=1-1, to=1-2]
\arrow["\alpha\iota_{00}\C_{\mathsf{X},\mathsf{X}_{0}}"', curve={height=12pt}, Rightarrow, from=1-1, to=2-2]
\arrow["\varepsilon\C_{\mathsf{X},\mathsf{X}_{0}}", Rightarrow, from=1-2, to=2-2]
\end{tikzcd}

%% file: diagrams/d207.tikz
\begin{tikzcd}[column sep=huge, ampersand replacement=\&]  
{\mathfrak{N}_{X,X_{0}}^{u}} \& {\mathfrak{N}_{X,X_{0}}^{u}\C_{\mathsf{X},\mathsf{X}_{0}}\mathfrak{N}_{X,X_{0}}^{u}} \\
 \& {\mathfrak{N}_{X,X_{0}}^{u}\mathfrak{A}\uK.}
\arrow["\eta\mathfrak{N}_{X,X_{0}}^{u}", Rightarrow, from=1-1, to=1-2]
\arrow["\mathfrak{N}_{X,X_{0}}^{u}\alpha\iota_{00}"', curve={height=12pt}, Rightarrow, from=1-1, to=2-2]
\arrow["\mathfrak{N}_{X,X_{0}}^{u}\varepsilon", Rightarrow, from=1-2, to=2-2]
\end{tikzcd}

%% file: sec4.tex
\section{\label{sec:tech}Technical theorem}

In this section, we show that one can skip the colimit construction
in the definition of generalized morphisms in the case of separable
$C^{*}$-algebras. 
\begin{thm}\label{thm:colim-is-redundant}
Let $A$ and $B$ be $C^{*}$-algebras,
with $A$ separable, and let $\mathbf{X}\coloneqq(X,X_{0})$ be a
pair of discrete metric spaces of bounded geometry. Then there is
an isomorphism of monoids
\begin{equation}
\left[A,\mathfrak{N}^{u}_{\mathbf{X}}\mathfrak{A}\uK,B\right]\cong\colim_{n}\left[A,\mathfrak{N}^{u}_{\mathbf{X}}\mathfrak{A}^{n}\uK,B\right].\label{eq:claim}
\end{equation}
\end{thm}

The proof of this technical theorem is a quite tedious modification
of the argument from~\cite[Chapter 2]{GHT} which occupies the remainder
of this section.

\subsection{Computing norms in quotients}
\begin{lem}\label{lem:norm-monot}
Let $A$ be a $C^{*}$-algebra, and let $a,b,c\in A$,
with $a$ and $b$ self-adjoint and $a^{2}\leq b^{2}$. Then $\norm{ac}\leq\norm{bc}$.
\end{lem}

\begin{proof}
We have
\begin{equation}
0\leq(ac)^{*}(ac)=c^{*}a^{2}c\leq c^{*}b^{2}c,\label{eq:11}
\end{equation}
in which the second inequality holds by Lemma~\ref{lem:pos}\ref{enu:pos1}.
Combining Lemma~\ref{lem:pos}\ref{enu:pos2} with~(\ref{eq:11}),
we obtain 
\[
\norm{ac}^{2}=\norm{c^{*}a^{2}c}\leq\norm{c^{*}b^{2}c}=\norm{bc}^{2}.
\]
\end{proof}

Let $\iota\in\GEFC(F_{0},F)$ be a componentwise
ideal inclusion with monic-preserving $F_{0}$ and~$F$. Let $B$
be a $C^{*}$-algebra, and let $0\xrightarrow{}B\xrightarrow{}B^{+}\stackrel{}{\leftrightarrows}\mathbb{C}\xrightarrow{}0$
be the split exact sequence associated to the unitization of~$B$.
These data give rise to the diagram
\[
\input{diagrams/d498.tikz}
\]
which allows us to assume, abusing notation, that $F_{0}\mathbb{C}$,~$F_{0}B$,
and~$FB$ are $C^{*}$-subalgebras of~$FB^{+}$. This allows us
to make the following definition.
\begin{defn}
Let $\iota\in\GEFC(F_{0},F)$ be a componentwise
ideal inclusion with monic-preserving $F_{0}$ and~$F$. We define
a \emph{quasiscalar approximate unit} (abbreviated~q.s.a.u.) assosiated
to $\iota$ to be a net $\{u_{\lambda}\}\subset F_{0}\mathbb{C}$
satisfying the following properties:
\begin{itemize}
\item $0\leq u_{\lambda}\leq u_{\lambda'}\leq1$ for all $\lambda\leq\lambda'$;
\item $u_{\lambda}f\in F_{0}B$ for all $B\in\in\Cstar$ and $f\in FB$;
\item $u_{\lambda}f_{0}\xrightarrow{}f_{0}$ for all $B\in\in\Cstar$ and
$f_{0}\in F_{0}B$.
\end{itemize}
\end{defn}

\begin{lem}\label{lem:quot-norm}
Let $0\Rightarrow F_{0}\xRightarrow{\iota}F\xRightarrow{q}G\Rightarrow0$
be a componentwise short exact sequence with monic-preserving
$F_{0}$ and~$F$. If $u_{\lambda}$ is a q.s.a.u. associated to~$\iota$,
then for every $B\in\in\Cstar$ and every $f\in FB$, the following
equality holds:
\[
\norm{qB(f)}=\lim_{\lambda}\norm{(1-u_{\lambda})f}.
\]
\end{lem}

\begin{proof}
For every $\varepsilon>0$, we can find $f_{0}\in F_{0}B$ such that
$\norm{f-f_{0}}\leq\norm{qB(f)}+\varepsilon$. Then
\begin{multline*}
\qquad\qquad\norm{f-u_{\lambda}f}-\norm{f_{0}-u_{\lambda}f_{0}}\leq\norm{(f-u_{\lambda}f)-(f_{0}-u_{\lambda}f_{0})}\\
=\norm{(1-u_{\lambda})(f-f_{0})}\leq\norm{f-f_{0}}\leq\norm{qB(f)}+\varepsilon,\qquad\qquad
\end{multline*}
and hence,
\[
\norm{f-u_{\lambda}f}\leq\norm{qB(f)}+\varepsilon+\norm{f_{0}-u_{\lambda}f_{0}}\xrightarrow[\lambda]{}\norm{qB(f)}+\varepsilon
\]
which implies that
\[
\lim_{\lambda}\norm{(1-u_{\lambda})f}\leq\norm{qB(f)}.
\]
The reverse inequality is obvious. 
\end{proof}

We define the functions \begin{alignat*}{3}
 & \chi^{0}_{r}\in\{0,1\}^{X}, & ~~ & \textrm{ for }r\in\mathbb{N}, & \qquad & \chi^{0}_{r}(x)\coloneqq\begin{cases}
1, & x\in\nbhd_{r}(X_{0});\\
0, & \textrm{otherwise,}
\end{cases}\\
 & \chi^{1}_{n}\in\mathfrak{T}_{0}\mathbb{C}, &  & \textrm{ for }n\in\mathbb{N}, & \qquad & \chi^{1}_{n}(t)\coloneqq\begin{cases}
1, & t\in[0,n-1];\\
n-t, & t\in[n-1,n];\\
0, & t\in[n,\infty),
\end{cases}\\
 & \chi^{2}_{f}\in\mathfrak{T}\mathfrak{T}_{0}\mathbb{C}, &  & \textrm{ for }f\in\mathbb{N}^{\mathbb{N}}, &  & \chi^{2}_{f}(t_{1},t_{2})\coloneqq(1-s)\chi^{1}_{f(n)}(t_{2})+s\chi^{1}_{f(n+1)}(t_{2}),\\
 &  &  &  &  & \hspc{19.8}\textrm{with }n\coloneqq[t_{1}],~s\coloneqq\{t_{1}\},
\end{alignat*}
and for abbreviation, set 
\[
\overline{\chi}^{0}_{r}\coloneqq1-\chi^{0}_{r},\qquad\overline{\chi}^{1}_{n}\coloneqq1-\chi^{1}_{n},\qquad\overline{\chi}^{2}_{f}\coloneqq1-\chi^{2}_{f}.
\]

Recall that $\mathfrak{M}^{u}_{X}$ and $\mathfrak{T}$ are ideal-inclusion-preserving endofunctors (see~Proposition~\ref{prop:MX-properties}\ref{enu:x-iip}
and~\cite[proof of Lemma~2.4]{GHT}).
\begin{lem}\label{lem:nets}
The following statements hold:
\begin{itemize}
\item $\{D^{0}_{\rho}\coloneqq\diag_{x}(\chi^{0}_{\rho}(x))\mid\rho\in\mathbb{N}\}$
is a q.s.a.u. associated to $\mathfrak{M}^{u}_{X\supset X_{0}}\xRightarrow{}\mathfrak{M}^{u}_{X}$;
\item $\{D^{1}_{\mu}\coloneqq\diag_{x}(\chi^{1}_{\mu(x)})\mid\mu\in\mathbb{N}^{X}\}$
is a q.s.a.u. associated to $\mathfrak{M}^{u}_{X}\mathfrak{T}_{0}\xRightarrow{}\mathfrak{M}^{u}_{X}\mathfrak{T}$;
\item $\{D^{2}_{\nu}\coloneqq\diag_{x}(\chi^{2}_{\nu(x)})\mid\nu\in\mathbb{N}^{X\times\mathbb{N}}\}$
is a q.s.a.u. associated to $\mathfrak{M}^{u}_{X}\mathfrak{T}\mathfrak{T}_{0}\xRightarrow{}\mathfrak{M}^{u}_{X}\mathfrak{T}\mathfrak{T}$,
\end{itemize}
where $\mathbb{N}$, $\mathbb{N}^{X}$, and $\mathbb{N}^{X\times\mathbb{N}}$
are regarded as directed sets with the natural order. 
\end{lem}

\begin{proof}
Straightforward.
\end{proof}

\begin{lem}\label{lem:net-1}
Let $\Lambda$ and $M$ be two directed sets, and
let $f\colon\Lambda\times M\to[0,\infty)$ be a function such that
$f(\singlecdot,\mu)$ and $f(\lambda,\singlecdot)$ are non-increasing
for all $\mu$ and $\lambda$, respectively. Then
\[
\lim_{\lambda\in\Lambda}\lim_{\mu\in M}f(\lambda,\mu)=\lim_{(\lambda,\mu)\in\Lambda\times M}f(\lambda,\mu).
\]
\end{lem}

\begin{proof}
The proof is left as a simple exercise.
\end{proof}

\begin{prop}\label{prop:quot-norms}
Let $B\in\in\Cstar$, $m\in\mathfrak{M}^{u}_{X}\mathfrak{T}\mathfrak{T}B$,
and $m'\in\mathfrak{M}^{u}_{X}\mathfrak{T}B$. Then the following
equalities hold:
\begin{alignat}{1}
\norm{q_{\mathbf{X}}\as\as B(m)} & =\lim_{(\rho,\mu,\nu)\in\mathbb{N}\times\mathbb{N}^{\mathbb{N}}\times\mathbb{N}^{X\times\mathbb{N}}}\norm{\overline{D}^{0}_{\rho}\cdot\overline{D}^{1}_{\mu}\cdot\overline{D}^{2}_{\nu}\cdot m},\label{eq:qn1}\\
\norm{q_{\mathbf{X}}\as B(m')} & =\hspc 4\lim_{(\rho,\mu)\in\mathbb{N}\times\mathbb{N}^{\mathbb{N}}}\norm{\overline{D}^{0}_{\rho}\cdot\overline{D}^{1}_{\mu}\cdot m'},\label{eq:qn2}
\end{alignat}
where $\overline{D}^{0}_{\rho}\coloneqq1-D^{0}_{\rho}$, $\overline{D}^{1}_{\mu}\coloneqq1-D^{1}_{\mu}$,
and $\overline{D}^{2}_{\nu}\coloneqq1-D^{2}_{\nu}$.
\end{prop}

\begin{proof}
Equality~(\ref{eq:qn1}) follows from 
\begin{multline*}
\norm{q_{\mathbf{X}}\as\as B(m)}=\lim_{\rho\in\mathbb{N}}\norm{\overline{D}^{0}_{\rho}\cdot\mathfrak{M}^{u}_{X}\as\as B(m)}=\lim_{\rho\in\mathbb{N}}\lim_{\mu\in\mathbb{N}^{X}}\norm{\overline{D}^{1}_{\mu}\cdot\overline{D}^{0}_{\rho}\cdot\mathfrak{M}^{u}_{X}\mathfrak{T}\as B(m)}\\
=\lim_{\rho\in\mathbb{N}}\lim_{\mu\in\mathbb{N}^{X}}\lim_{\nu\in\mathbb{N}^{X\times\mathbb{N}}}\norm{\overline{D}^{0}_{\rho}\cdot\overline{D}^{1}_{\mu}\cdot\overline{D}^{2}_{\nu}\cdot m}=\lim_{(\rho,\mu,\nu)}\norm{\overline{D}^{2}_{\nu}\cdot\overline{D}^{1}_{\mu}\cdot\overline{D}^{0}_{\rho}\cdot m},
\end{multline*}
where the first three equalities hold by Lemmas~\ref{lem:nets} and~\ref{lem:quot-norm},
and the last equality follows from Lemmas~\ref{lem:net-1} and~\ref{lem:norm-monot}
together with the fact that $\overline{D}^{0}_{\rho}$, $\overline{D}^{1}_{\mu}$,
and $\overline{D}^{2}_{\nu}$ commute. The proof of~(\ref{eq:qn2})
is similar.
\end{proof}

\subsection{Reparametrizations}
\begin{defn}
A \emph{reparametrization} $r$ is a strictly increasing continuous
bijection $r\colon\mathbb{R}_{+}\to\mathbb{R}_{+}$. Given two reparametrizations
$r$ and $r'$, we say that~$r$ is \emph{faster} than~$r'$ (and
write $r\geq r'$), if $r(t)\geq r'(t)$ for all $t\in\mathbb{R}_{+}$.
\end{defn}

Every reparametrization $r$ gives rise to the following natural transformations:
\begin{alignat*}{2}
\tau_{r}\colon & \mathfrak{T}^{2}\Rightarrow\mathfrak{T}, &  & \tau_{r}B\colon f\mapsto\fun_{t}f(t,r(t)),\\
\sigma_{r}\colon & \mathfrak{T}^{2}\Rightarrow\mathfrak{T}^{2}I, &  & \sigma_{r}B\colon f\mapsto\fun_{t_{1}}\fun_{t_{2}}\fun_{s}f(t_{1},st_{2}+(1-s)r(t_{1})),\\
r^{*}\colon & \mathfrak{A}\Rightarrow\mathfrak{A}, & \quad & r^{*}B\colon\as_{t}f(t)\mapsto\as_{t}f(r(t)).
\end{alignat*}

\begin{lem}\label{lem:repar-homot}
 $r^{*}\simeq\id_{\mathfrak{A}}$.
\end{lem}

\begin{proof}
This is witnessed by the homotopy given by the formula $\as_{t}f(t)\mapsto\as_{t}\fun_{s}f((1-s)r(t)+st)$.
\end{proof}

\begin{lem}\label{lem:new-reparam}
Let $D$ be a $C^{*}$-algebra,
and let $E\subset\mathfrak{N}^{u}_{\mathbf{X}}\mathfrak{A}^{2}D$
and $\widetilde{E}\subset\mathfrak{M}^{u}_{X}\mathfrak{T}^{2}D$ be
two separable $C^{*}$-subalgebras, where $E$ is the image of $\widetilde{E}$
under $q_{\mathbf{X}}\as\as D$. Then there is a reparametrization
$r_{0}$ such that the following inequalities hold
for all $r\geq r_{0}$ and all $m\in\widetilde{E}$:
\begin{alignat}{1}
\norm{(q_{\mathbf{X}}\as\as D)(m)} & \geq\norm{(q_{\mathbf{X}}(\as\as I\circ\sigma_{r})D)(m)},\label{eq:est1}\\
\norm{(q_{\mathbf{X}}\as\as D)(m)} & \geq\norm{(q_{\mathbf{X}}(\as\circ\tau_{r})D)(m)}.\label{eq:est2}
\end{alignat}
\end{lem}

\begin{proof}
Let $\left\{ m^{k}\right\} ^{\infty}_{k=1}\subset\widetilde{E}$ be
a dense subset of~$\widetilde{E}$. Proposition~\ref{prop:quot-norms}
implies that for all $n\in\mathbb{N}$, there exist $\rho_{n}\in\mathbb{N}$,
$\mu_{n}\in\mathbb{N}^{X}$, and $\nu_{n}\in\mathbb{N}^{X\times\mathbb{N}}$
such that
\begin{alignat}{2}
 &  &  & k=1,\dots,n;\quad\rho\geq\rho_{n};\quad\mu\geq\mu_{n};\quad\nu\geq\nu_{n}\nonumber \\
\Longrightarrow &  & ~~ & \sup_{t_{1},t_{2}}\norm{\mat_{x,y}\overline{\chi}^{0}_{\rho}(x)\overline{\chi}^{1}_{\mu(x)}(t_{1})\overline{\chi}^{2}_{\nu(x)}(t_{1},t_{2})m^{k}_{x,y}(t_{1},t_{2})}\leq\norm{(q_{\mathbf{X}}\as\as D)(m^{k})}+\frac{1}{n}.\label{eq:41}
\end{alignat}
Furthermore, we may assume that $\rho_{n}$, $\mu_{n}$, and $\nu_{n}$
are increasing in~$n$.

Enumerate the elements of $X\eqqcolon\{x_{j}\}^{\infty}_{j=1}$, and
find a sequence of reparametrizations $\{R_{N}\}^{\infty}_{N=1}$,
increasing in $N$, such that
\begin{alignat}{2}
l,j=1,\dots,N;\quad t_{1}\geq0;\quad t_{2}\geq R_{N}(t_{1}) & \quad\Longrightarrow\quad & \overline{\chi}^{2}_{\nu_{l}(x_{j})}(t_{1},t_{2})=1.\label{eq:42}
\end{alignat}
Set $r_{0}(N)\coloneqq R_{N}(N)$ for $N\in\mathbb{Z}_{+}$, define
the reparametrization $t\mapsto r_{0}(t)$ by linear interpolation,
and fix some $r\geq r_{0}$.

For $k,m\in\mathbb{Z}_{+}$, set $n\coloneqq\max(k,m)+1$ and assume,
without loss of generality, that
\begin{equation}
\mu_{m}(x_{k})\geq\max(k,m)+1=n.\label{eq:h2}
\end{equation}
 Note that
\begin{alignat}{2}
t_{1}\geq n-1;\quad t_{2}\geq r(t_{1}+1) & \quad\Longrightarrow\quad & \overline{\chi}^{2}_{\nu_{m}(x_{k})}(t_{1},t_{2})=1.\label{eq:h1}
\end{alignat}
Indeed, if $t_{1}\geq n-1$ and $t_{2}\geq r(t_{1}+1)$, then $t_{2}\geq r(t_{1}+1)\geq r_{0}(t_{1}+1)\geq r_{0}([t_{1}]+1)=R_{[t_{1}]+1}([t_{1}]+1)\geq R_{n}(t_{1})$,
and hence, $\overline{\chi}^{2}_{\nu_{m}(x_{k})}(t_{1},t_{2})=1$
by~(\ref{eq:42}).

It follows from~(\ref{eq:h1}) and~(\ref{eq:h2}) that for all $n\in\mathbb{Z}_{+}$,
$x\in X$, and $t_{1}\geq0$, we have
\begin{alignat}{4}
 &  & t_{2}\geq r(t_{1}+1) &  & \quad\Longrightarrow\quad &  & \overline{\chi}^{1}_{\mu_{n}(x)}(t_{1})\overline{\chi}^{2}_{\nu_{n}(x)}(t_{1},t_{2})=\overline{\chi}^{1}_{\mu_{n}(x)}(t_{1}).\label{eq:chi12}
\end{alignat}
To prove~(\ref{eq:est1}), it suffices to show that the following
inequalities hold for all $k\in\mathbb{N}$:
\begin{alignat}{2}
 &  &  & \norm{(q_{\mathbf{X}}\as\as D)(m^{k})}\nonumber \\
= & \thinspace & ~ & \mymakebox{2em}{\lim_{n}}\mymakebox{4em}{\sup_{t_{1},t_{2}}}\norm{\mat_{x,y}\overline{\chi}^{0}_{\rho_{n}}(x)\overline{\chi}^{1}_{\mu_{n}(x)}(t_{1})\overline{\chi}^{2}_{\nu_{n}(x)}(t_{1},t_{2})m^{k}_{x,y}(t_{1},t_{2})}\label{eq:ch1}\\
= &  &  & \mymakebox{2em}{\lim_{n}}\mymakebox{4em}{\sup_{t_{1},t_{2},s}}\norm{\mat_{x,y}\overline{\chi}^{0}_{\rho_{n}}(x)\overline{\chi}^{1}_{\mu_{n}(x)}(t_{1})\overline{\chi}^{2}_{\nu_{n}(x)}(t_{1},st_{2}+(1-s)r(t_{1}))m^{k}_{x,y}(t_{1},st_{2}+(1-s)r(t_{1}))}\label{eq:ch2}\\
\geq &  &  & \mymakebox{2em}{\lim_{n}}\mymakebox{4em}{\sup_{\substack{s,~t_{1},\\
t_{2}\geq r(t_{1}+1)
}
}}\norm{\mat_{x,y}\overline{\chi}^{0}_{\rho_{n}}(x)\overline{\chi}^{1}_{\mu_{n}(x)}(t_{1})m^{k}_{x,y}(t_{1},st_{2}+(1-s)r(t_{1}))}\label{eq:ch3}\\
= &  &  & \mymakebox{2em}{\lim_{n}}\mymakebox{4em}{\sup_{\substack{s,~t_{1},\\
t_{2}\geq r(t_{1}+1)
}
}}\norm{\mat_{x,y}\overline{\chi}^{0}_{\rho_{n}}(x)\overline{\chi}^{1}_{\mu_{n}(x)}(t_{1})\overline{\chi}^{2}_{\nu_{n}(x)}(t_{1},t_{2})m^{k}_{x,y}(t_{1},st_{2}+(1-s)r(t_{1}))}\label{eq:ch4}\\
\geq &  &  & \mymakebox{2em}{\lim_{\rho,\mu,\nu}}\mymakebox{4em}{\sup_{\substack{s,~t_{1},\\
t_{2}\geq r(t_{1}+1)
}
}}\norm{\mat_{x,y}\overline{\chi}^{0}_{\rho}(x)\overline{\chi}^{1}_{\mu(x)}(t_{1})\overline{\chi}^{2}_{\nu(x)}(t_{1},t_{2})m^{k}_{x,y}(t_{1},st_{2}+(1-s)r(t_{1}))}\label{eq:ch5}\\
= &  &  & \mymakebox{2em}{\lim_{\substack{\rho,\mu,\\
\hspc{-2}\nu\geq r(\cdot+1)+1
}
}}\mymakebox{4em}{\sup_{t_{1},t_{2},s}}\norm{\mat_{x,y}\overline{\chi}^{0}_{\rho}(x)\overline{\chi}^{1}_{\mu(x)}(t_{1})\overline{\chi}^{2}_{\nu(x)}(t_{1},t_{2})m^{k}_{x,y}(t_{1},st_{2}+(1-s)r(t_{1}))}\label{eq:ch6}\\
= &  &  & \mymakebox{2em}{\lim_{\substack{\rho,\mu,\\
\hspc{-2}\nu\geq r(\cdot+1)+1
}
}}\mymakebox{4em}{\sup_{t_{1},t_{2}}}\norm{\mat_{x,y}\overline{\chi}^{0}_{\rho}(x)\overline{\chi}^{1}_{\mu(x)}(t_{1})\overline{\chi}^{2}_{\nu(x)}(t_{1},t_{2})\fun_{s}m^{k}_{x,y}(t_{1},st_{2}+(1-s)r(t_{1}))}\label{eq:ch8}\\
= &  &  & \norm{(q_{\mathbf{X}}(\as\as I\circ\sigma_{r})D)(m^{k})}.\label{eq:ch7}
\end{alignat}
To see this, note first that the limits considered above exist since
the corresponding nets are decreasing by Lemma~\ref{lem:norm-monot}
and bounded from below by zero. Equality (\ref{eq:ch1}) follows from~(\ref{eq:41});
(\ref{eq:ch2}) is obvious; (\ref{eq:ch3}) follows from the implication
\begin{equation}
\nu(x,\cdot)\geq r(\cdot+1)+1;\quad t_{2}\leq r(t_{1}+1)\implies~\overline{\chi}^{2}_{\nu(x)}(t_{1},t_{2})=0,\label{eq:impl}
\end{equation}
which holds for all $x\in X$ by the definition of $\nu$; (\ref{eq:ch8})
follows from Lemma~\ref{lem:norm-eval}; (\ref{eq:ch4}) follows
from~(\ref{eq:chi12}); (\ref{eq:ch5}) holds by monotonicity of
the net; (\ref{eq:ch6}) follows from~(\ref{eq:impl}), and finally,
(\ref{eq:ch7}) follows from Proposition~\ref{prop:quot-norms}. 

The proof of~(\ref{eq:est2}) is similar: consider inequalities~(\ref{eq:ch1})--(\ref{eq:ch3})
with $s=0$; discard $t_{2}$; pass from the subnet $(\rho_{n},\mu_{n})$
to the whole net $(\rho,\mu)$; and apply Proposition~\ref{prop:quot-norms}
to obtain the desired estimate.
\end{proof}

\subsection{\label{subsec:proof-colim-is-redundant}Proof of the technical theorem}
\begin{lem}\label{lem:param-zho}
Let $\psi\colon A\to B$ and
$\psi'\colon A\to B'$ be $*$-homomorphisms such that $\psi$ is
surjective and $\norm{\psi'(a)}\leq\norm{\psi(a)}$ for all $a\in A$.
Then the map $\xi\colon B\to B'\colon\psi(a)\mapsto\psi'(a)$ is a
well-defined $*$-homomorphism such that $\xi\circ\psi=\psi'$.
\end{lem}

\begin{proof}
Straightforward.
\end{proof}

\begin{lem}\label{lem:sep-decr}
Let $D$ be a $C^{*}$-algebra, and let $E\subset\mathfrak{N}^{u}_{\mathbf{X}}\mathfrak{A}^{2}D$
be a separable $C^{*}$-subalgebra. Then there is a reparametrization~$r_{0}$
such that: 
\begin{enumerate}[label=\textup{(\roman*)}]
\item\label{enu:sep-descr-wd}the following $*$-homomorphism is well-defined
for all $r\geq r_{0}$:
\begin{alignat*}{1}
\psi\colon & E\to\mathfrak{N}^{u}_{\mathbf{X}}\mathfrak{A}D\colon\qmat_{x,y}\as_{t_{1}}\as_{t_{2}}m_{x,y}(t_{1},t_{2})\mapsto\qmat_{x,y}\as_{t}m_{x,y}(r^{-1}(t),t).
\end{alignat*}
\item\label{enu:sep-descr-homot}the $*$-homomorphism $\mathfrak{N}^{u}_{\mathbf{X}}\alpha\mathfrak{A}D\circ\psi$
is $\mathfrak{N}^{u}_{\mathbf{X}}\mathfrak{A}^{2}$-homotopic to the
inclusion $E\xrightarrow{\incl}\mathfrak{N}^{u}_{\mathbf{X}}\mathfrak{A}^{2}D$.
\end{enumerate}
\end{lem}

\begin{proof}
Let $r_{0}$ be the reparametrization given by Lemma~\ref{lem:new-reparam}.
Using inequalities~(\ref{eq:est1}) and~(\ref{eq:est2}) and Lemma~\ref{lem:param-zho},
we conclude that the following maps are well-defined $*$-homomorphisms
for every $r\geq r_{0}$:
\begin{alignat*}{4}
 &  & \psi'\colon & E\to & \mathfrak{N}^{u}_{\mathbf{X}}\mathfrak{A}D\colon &  &  & \qmat_{x,y}\as_{t_{1}}\as_{t_{2}}m_{x,y}(t_{1},t_{2})\mapsto\qmat_{x,y}\as_{t}m_{x,y}(t,r(t)),\\
 &  & \Phi\colon & E\to & \mathfrak{N}^{u}_{\mathbf{X}}\mathfrak{A}^{2}ID\colon &  &  & \qmat_{x,y}\as_{t_{1}}\as_{t_{2}}m_{x,y}(t_{1},t_{2})\\
 &  &  &  & \mapsto~ &  &  & \qmat_{x,y}\as_{t_{1}}\as_{t_{2}}\fun_{s\in[0,1]}m_{x,y}(t_{1},st_{2}+(1-s)r(t_{1})).
\end{alignat*}
By construction, $\Phi$ is a $\mathfrak{N}^{u}_{\mathbf{X}}\mathfrak{A}^{2}$-homotopy
connecting $\mathfrak{N}^{u}_{\mathbf{X}}\mathfrak{A}\alpha D\circ\psi'$
with the inclusion $E\xrightarrow{\incl}\mathfrak{N}^{u}_{\mathbf{X}}\mathfrak{A}^{2}D$.
The equality $\psi=\mathfrak{N}^{u}_{\mathbf{X}}(r^{-1})^{*}D\circ\psi'$
implies that $\psi$ is also well-defined, which proves statement~\ref{enu:sep-descr-wd}.
Furthermore, Lemma~\ref{lem:repar-homot} gives 
\begin{equation}
\psi\simeq_{\mathfrak{N}^{u}_{\mathbf{X}}\mathfrak{A}}\psi'.\label{eq:psipsi}
\end{equation}
Statement~\ref{enu:sep-descr-homot} follows from
\[
\mathfrak{N}^{u}_{\mathbf{X}}\alpha\mathfrak{A}D\circ\psi\simeq_{\mathfrak{N}^{u}_{\mathbf{X}}\mathfrak{A}^{2}}\mathfrak{N}^{u}_{\mathbf{X}}\alpha\mathfrak{A}D\circ\psi'\simeq_{\mathfrak{N}^{u}_{\mathbf{X}}\mathfrak{A}^{2}}\mathfrak{N}^{u}_{\mathbf{X}}\mathfrak{A}\alpha D\circ\psi'\simeq_{\mathfrak{N}^{u}_{\mathbf{X}}\mathfrak{A}^{2}}\incl,
\]
where the first relation holds by~(\ref{eq:psipsi}) and Lemma~\ref{lem:weak-rangle},
the second follows from Lemmas~\ref{lem:alpha-A} and~\ref{lem:weak-rangle},
and the last is witnessed by~$\Phi$.
\end{proof}

\begin{lem}\label{lem:some-technicality}
Let $D_{1}$ and $D_{2}$ be $C^{*}$-algebras,
and let $\varepsilon\colon D_{1}\to D_{2}$ be a $*$-homomorphism.
Let $E_{1}\subset\mathfrak{N}^{u}_{\mathbf{X}}\mathfrak{A}^{2}D_{1}$,
$E_{2}\subset\mathfrak{N}^{u}_{\mathbf{X}}\mathfrak{A}^{2}D_{2}$,
$C_{2}\subset\mathfrak{N}^{u}_{\mathbf{X}}\mathfrak{A}D_{2}$ be separable
$C^{*}$-subalgebras, and let $\xi\colon C_{2}\to E_{2}$ be a $*$-homomorphism
such that the left-hand diagram below commutes:
\begin{alignat*}{2}
\input{diagrams/d499.tikz} & \qquad & \input{diagrams/d500.tikz}
\end{alignat*}
where $\incl$ denotes inclusions. Then there exist $*$-homomorphisms
$\psi_{1}$ and $\psi_{2}$ making the right-hand diagram commutative.
\end{lem}

\begin{proof}
It follows from Lemma~\ref{lem:sep-decr} that for a sufficiently
fast reparametrization $r$, the maps
\begin{equation}
\psi_{j}\colon E_{j}\to\mathfrak{N}^{u}_{\mathbf{X}}\mathfrak{A}D_{j}\colon\qmat_{x,y}\as_{t_{1}}\as_{t_{2}}m_{x,y}(t_{1},t_{2})\mapsto\qmat_{x,y}\as_{t}m_{x,y}(r^{-1}(t),t)\label{eq:psi-1}
\end{equation}
for $j=1,2$, are $*$-homomorphisms. It is straightforward to check
that $\psi_{1}$ and $\psi_{2}$ make the right-hand diagram commutative.
\end{proof}

\begin{prop}\label{prop:colim-redund}
The following map is an isomorphism of monoids
for all $n\in\mathbb{N}$ whenever $A$ is separable:
\begin{equation}
\left[A,\mathfrak{N}^{u}_{\mathbf{X}}\mathfrak{A}^{n}\uK,B\right]\xrightarrow{\left\langle \mathfrak{N}^{u}_{\mathbf{X}}\alpha\mathfrak{A}^{n}\right\rangle }\left[A,\mathfrak{N}^{u}_{\mathbf{X}}\mathfrak{A}^{n+1}\uK,B\right].\label{eq:colim-in-general}
\end{equation}
\end{prop}

\begin{proof}
That $\left\langle \mathfrak{N}^{u}_{\mathbf{X}}\alpha\mathfrak{A}^{n}\right\rangle $
is a morphism of monoids follows from Theorem~\ref{thm:58}.

Let $\varphi\colon A\to\mathfrak{N}^{u}_{\mathbf{X}}\mathfrak{A}^{n+1}\uK B$
be a $*$-homomorphism, and set $D\coloneqq\mathfrak{A}^{n-1}\uK B$.
According to Lemma~\ref{lem:sep-decr}, there exists a $*$-homomorphism
$\psi\colon\varphi(A)\to\mathfrak{N}^{u}_{\mathbf{X}}\mathfrak{A}D$
making the left-hand diagram below commutative up to $\mathfrak{N}^{u}_{\mathbf{X}}\mathfrak{A}^{2}$-homotopy:
\begin{alignat*}{2}
\input{diagrams/d235.tikz} & \qquad & \input{diagrams/d236.tikz}
\end{alignat*}
Then the right-hand diagram commutes up to $\mathfrak{N}^{u}_{\mathbf{X}}\mathfrak{A}^{n+1}$-homotopy
(cf.~Lemma~\ref{lem:stronger-homot}), which proves that the map~(\ref{eq:colim-in-general})
is surjective.

Let now $\varphi_{0},\varphi_{1}\colon A\to\mathfrak{N}^{u}_{\mathbf{X}}\mathfrak{A}^{n}\uK B$
be two $*$-homomorphisms such that there is a homotopy $\Phi\colon A\to\mathfrak{N}^{u}_{\mathbf{X}}\mathfrak{A}^{n+1}\uK IB$
connecting $\mathfrak{N}^{u}_{\mathbf{X}}\alpha\mathfrak{A}^{n}\uK\circ\varphi_{0}$
with $\mathfrak{N}^{u}_{\mathbf{X}}\alpha\mathfrak{A}^{n}\uK\circ\varphi_{1}$.
Define the $*$-homomorphism
\begin{alignat*}{1}
e_{01}\colon & IB\to B\oplus B\colon f\mapsto(\ev_{0}B(f),\ev_{1}B(f)),\\
\varphi_{01}\colon & A\to\mathfrak{N}^{u}_{\mathbf{X}}\mathfrak{A}^{n}\uK B\oplus\mathfrak{N}^{u}_{\mathbf{X}}\mathfrak{A}^{n}\uK B\cong\mathfrak{N}^{u}_{\mathbf{X}}\mathfrak{A}^{n}\uK(B\oplus B)\colon a\mapsto(\varphi_{0}(a),\varphi_{1}(a)),
\end{alignat*}
and consider the diagram \[
\input{diagrams/d237.tikz}
\]
in which $E$ is the separable $C^{*}$-subalgebra in $\mathfrak{N}^{u}_{\mathbf{X}}\mathfrak{A}^{n+1}\uK(B\oplus B)$
generated by the subset 
\[
(\mathfrak{N}^{u}_{\mathbf{X}}\mathfrak{A}^{n}\uK e_{01})(\Phi(a))\cup(\mathfrak{N}^{u}_{\mathbf{X}}\alpha\mathfrak{A}^{n}\uK(B\oplus B))(\varphi_{01}(A)).
\]
One can easily verify that the solid part of this diagram commutes.
Applying Lemma~\ref{lem:some-technicality}, we obtain $*$-homomorphisms
$\psi_{1}$ and $\psi_{2}$ such that $\psi_{1}\circ\Phi$ is a homotopy
connecting $\varphi_{0}$ and $\varphi_{1}$. This implies that the
map~(\ref{eq:colim-in-general}) is injective.
\end{proof}

Theorem~\ref{thm:colim-is-redundant} follows directly from Proposition~\ref{prop:colim-redund}.

%% file: diagrams/d498.tikz
\begin{tikzcd}[ampersand replacement=\&]
{F_{0}B} \& {FB} \\
{F_{0}B^{+}} \& {FB^{+}} \\
{F_{0}\mathbb{C}}
\arrow["\iota B", hook', from=1-1, to=1-2]
\arrow[hook', from=1-1, to=2-1]
\arrow[hook', from=1-2, to=2-2]
\arrow["\iota B^{+}", hook', from=2-1, to=2-2]
\arrow[shift right, from=2-1, to=3-1]
\arrow[shift right, hook', from=3-1, to=2-1]
\end{tikzcd}

%% file: diagrams/d499.tikz
\begin{tikzcd}[ampersand replacement=\&]
{E_{1}} \& {\mathfrak{N}_{\mathbf{X}}^{u}\mathfrak{A}^{2}D_{1}} \&  \& {\mathfrak{N}_{\mathbf{X}}^{u}\mathfrak{A}D_{1}} \\
{E_{2}} \& {\mathfrak{N}_{\mathbf{X}}^{u}\mathfrak{A}^{2}D_{2}} \&  \& {\mathfrak{N}_{\mathbf{X}}^{u}\mathfrak{A}D_{2}} \\
{C_{2},}
\arrow["\incl", from=1-1, to=1-2]
\arrow["\mathfrak{N}_{\mathbf{X}}^{u}\mathfrak{A}^{2}\varepsilon|_{E_{1}}"', from=1-1, to=2-1]
\arrow["\mathfrak{N}_{\mathbf{X}}^{u}\mathfrak{A}^{2}\varepsilon", from=1-2, to=2-2]
\arrow["\mathfrak{N}_{\mathbf{X}}^{u}\alpha\mathfrak{A}D_{1}"', from=1-4, to=1-2]
\arrow["\mathfrak{N}_{\mathbf{X}}^{u}\mathfrak{A}\varepsilon", from=1-4, to=2-4]
\arrow["\incl", from=2-1, to=2-2]
\arrow["\mathfrak{N}_{\mathbf{X}}^{u}\alpha\mathfrak{A}D_{2}"', from=2-4, to=2-2]
\arrow["\xi", from=3-1, to=2-1]
\arrow["\incl"', curve={height=6pt}, from=3-1, to=2-4]
\end{tikzcd}

%% file: diagrams/d500.tikz
\begin{tikzcd}[ampersand replacement=\&]
{E_{1}} \&  \& {\mathfrak{N}_{\mathbf{X}}^{u}\mathfrak{A}D_{1}} \\
{E_{2}} \&  \& {\mathfrak{N}_{\mathbf{X}}^{u}\mathfrak{A}D_{2}} \\
{C_{2},}
\arrow["\psi_{1}", from=1-1, to=1-3]
\arrow["\mathfrak{N}_{\mathbf{X}}^{u}\mathfrak{A}^{2}\varepsilon|_{E_{1}}"', from=1-1, to=2-1]
\arrow["\mathfrak{N}_{\mathbf{X}}^{u}\mathfrak{A}\varepsilon", from=1-3, to=2-3]
\arrow["\psi_{2}", from=2-1, to=2-3]
\arrow["\xi", from=3-1, to=2-1]
\arrow["\incl"', curve={height=6pt}, from=3-1, to=2-3]
\end{tikzcd}

%% file: diagrams/d235.tikz
\begin{tikzcd}[column sep=tiny, ampersand replacement=\&]  
{A} \&  \& {\varphi(A)} \&  \& {\mathfrak{N}_{\mathbf{X}}^{u}\mathfrak{A}^{2}D} \\
 \& {\hspc 1} \&  \& {\mathfrak{N}_{\mathbf{X}}^{u}\mathfrak{A}D}
\arrow["\varphi", from=1-1, to=1-3]
\arrow["\incl", from=1-3, to=1-5]
\arrow["\psi"', from=1-3, to=2-4]
\arrow["\mathfrak{N}_{\mathbf{X}}^{u}\alpha\mathfrak{A}D"', from=2-4, to=1-5]
\end{tikzcd}

%% file: diagrams/d236.tikz
\begin{tikzcd}[column sep=tiny, ampersand replacement=\&]  
{A} \&  \& {\varphi(A)} \&  \& {\mathfrak{N}_{\mathbf{X}}^{u}\mathfrak{A}^{n+1}\uK B} \\
 \& {\hspc 1} \&  \& {\mathfrak{N}_{\mathbf{X}}^{u}\mathfrak{A}^{n}\uK B}
\arrow["\varphi", from=1-1, to=1-3]
\arrow["\incl", from=1-3, to=1-5]
\arrow["\psi"', from=1-3, to=2-4]
\arrow["\mathfrak{N}_{\mathbf{X}}^{u}\alpha\mathfrak{A}^{n}\uK B"', from=2-4, to=1-5]
\end{tikzcd}

%% file: diagrams/d237.tikz
\begin{tikzcd}[ampersand replacement=\&]
 \& {{\Phi(A)}} \&  \& {{\mathfrak{N}_{\mathbf{X}}^{u}\mathfrak{A}^{n+1}\uK IB}} \&  \& {{\mathfrak{N}_{\mathbf{X}}^{u}\mathfrak{A}^{n}\uK IB}} \\
{{A}} \& {{E}} \&  \& {{\mathfrak{N}_{\mathbf{X}}^{u}\mathfrak{A}^{n+1}\uK(B\oplus B)}} \& {{\hspc 3}} \& {{\mathfrak{N}_{\mathbf{X}}^{u}\mathfrak{A}^{n}\uK(B\oplus B)}} \\
 \& {{\varphi_{01}(A)}}
\arrow["\incl", from=1-2, to=1-4]
\arrow["\psi_{1}", curve={height=-30pt}, dashed, from=1-2, to=1-6]
\arrow["\mathfrak{N}_{\mathbf{X}}^{u}\mathfrak{A}^{n+1}\uK e_{01}|_{\Phi(A)}", from=1-2, to=2-2]
\arrow["\mathfrak{N}_{\mathbf{X}}^{u}\mathfrak{A}^{n+1}\uK e_{01}", from=1-4, to=2-4]
\arrow["\mathfrak{N}_{\mathbf{X}}^{u}\alpha\mathfrak{A}^{n}\uK IB"', from=1-6, to=1-4]
\arrow["\mathfrak{N}_{\mathbf{X}}^{u}\mathfrak{A}^{n}\uK e_{01}", from=1-6, to=2-6]
\arrow["\Phi", from=2-1, to=1-2]
\arrow["\varphi_{01}"', from=2-1, to=3-2]
\arrow["\incl", from=2-2, to=2-4]
\arrow["\psi_{2}"', curve={height=12pt}, dashed, from=2-2, to=2-6]
\arrow["\mathfrak{N}_{\mathbf{X}}^{u}\alpha\mathfrak{A}^{n}\uK(B\oplus B)"', from=2-6, to=2-4]
\arrow["\mathfrak{N}_{\mathbf{X}}^{u}\alpha\mathfrak{A}^{n}\uK(B\oplus B)|_{\varphi_{01}(A)}"'{pos=0.4}, from=3-2, to=2-2]
\arrow["\incl"', curve={height=18pt}, from=3-2, to=2-6]
\end{tikzcd}